%% file: main.tex
\documentclass[11pt]{amsart}
\usepackage[colorlinks, linkcolor=blue, citecolor=blue,pagebackref]{hyperref}
\usepackage{amssymb}
\usepackage{mathtools}
\usepackage{amsrefs}
\usepackage{ytableau}
\usepackage{tabularray}
\usepackage{caption}
\usepackage{eucal}
\usepackage{bbm}
\usepackage{stmaryrd}
\usepackage{tikz}

\usepackage{fullpage}
\theoremstyle{plain}
\newtheorem{theorem}{Theorem}[section]
\newtheorem{prop}[theorem]{Proposition}
\newtheorem{lemma}[theorem]{Lemma}

\theoremstyle{definition}
\newtheorem{dfn}[theorem]{Definition}
\newtheorem{ex}[theorem]{Example}
\newtheorem{prob}{Problem}

\theoremstyle{remark}
\newtheorem{rem}[theorem]{Remark}
\numberwithin{equation}{section}

\DeclareMathAlphabet\mathsf{OT1}{cmss}{m}{n}

\newcommand{\ssA}{\mathsf{A}}
\newcommand{\ssC}{\mathsf{C}}
\newcommand{\ssD}{\mathsf{D}}
\newcommand{\C}{\mathbb C}
\newcommand{\R}{\mathbb R}
\DeclareMathOperator{\GL}{GL}

\newcommand{\su}{\mathfrak{su}}
\newcommand{\gl}{\mathfrak{gl}}
\DeclareMathOperator{\Sp}{Sp}
\renewcommand{\sp}{\mathfrak{sp}}
\DeclareMathOperator{\Or}{O}

\newcommand{\so}{\mathfrak{so}}
\renewcommand{\sl}{\mathfrak{sl}}
\newcommand{\g}{\mathfrak{g}}
\renewcommand{\k}{\mathfrak{k}}
\newcommand{\h}{\mathfrak{h}}
\renewcommand{\q}{\mathfrak{q}}
\newcommand{\p}{\mathfrak{p}}
\newcommand{\F}[2]{F^{(#2)}_{#1}}
\renewcommand{\O}{\mathcal{O}}
\newcommand{\B}{\mathcal{B}}
\newcommand{\W}{\mathcal{W}}
\newcommand{\kW}{\prescript{\k}{}{\mathcal{W}}}

\DeclareMathOperator{\rk}{rk}
\DeclareMathOperator{\Par}{Par}
\DeclareMathOperator{\tr}{tr}
\DeclareMathOperator{\M}{M}
\DeclareMathOperator{\SM}{SM}
\DeclareMathOperator{\AM}{AM}
\newcommand{\st}[1]{\llbracket \Phi_{#1}\rrbracket}
\newcommand{\stla}[1]{\llbracket \Phi_{#1}\rrbracket_{\la}}

\newcommand{\la}{\lambda}
\newcommand{\al}{\alpha}
\newcommand{\be}{\beta}
\newcommand{\ep}{\varepsilon}
\newcommand{\bfx}{\mathbf{x}}
\newcommand{\bfy}{\mathbf{y}}
\renewcommand{\neg}[1]{\overline{#1}}
\DeclareMathOperator{\rows}{\mathbf{Rows}}
\DeclareMathOperator{\cols}{\mathbf{Cols}}
\DeclareMathOperator{\ch}{ch}
\DeclareMathOperator{\tL}{\widetilde{\Lambda}}
\newcommand{\Lplusk}{\Lambda^{\!+\!}(\k)}
\newcommand{\bbox}{\blacksquare}
\newcommand{\posarrow}[2]{ \underset{\substack{\textstyle\uparrow\\\hidewidth\mathstrut#2\hidewidth}}{#1}}

\makeatletter
\providecommand*{\cupdot}{%
  \mathbin{%
    \mathpalette\@cupdot{}%
  }%
}
\newcommand*{\@cupdot}[2]{%
  \ooalign{%
    $\m@th#1\cup$\cr
    \hidewidth$\m@th#1\cdot$\hidewidth
  }%
}
\makeatother

\allowdisplaybreaks

\begin{document}

\title{Dimension identities, almost self-conjugate partitions,\\ and BGG complexes for Hermitian symmetric pairs}

\author{William Q.~Erickson$^{\text{a}}$}
\thanks{$^\text{a}$ Baylor University, William.Q.Erickson@gmail.com (Corresponding author)}

\author{Markus Hunziker$^\text{b}$}
\thanks{$^\text{b}$ Baylor University, Markus\_Hunziker@baylor.edu}

\begin{abstract}
An \emph{almost self-conjugate} (ASC) partition has a Young diagram in which each arm along the diagonal is exactly one box longer than its corresponding leg.
Classically, the ASC partitions and their conjugates appear in two of Littlewood's symmetric function identities.  
These identities can be viewed as Euler characteristics of BGG complexes of the trivial representation, for classical Hermitian symmetric pairs.
In this paper, we consider partitions  in which the arm--leg difference is an arbitrary constant $m$.  
By viewing these partitions as highest weights, we establish an infinite family of dimension identities between $\gl_n$- and $\gl_{n+m}$-modules.
We then interpret this result in the context of blocks in parabolic category $\O$: 
in particular, we exhibit six infinite families of congruent blocks whose corresponding posets of highest weights consist of the partitions in question.
These posets, in turn, lead to generalizations of the Littlewood identities and their corresponding BGG complexes.
Our results in this paper shed light on the surprising combinatorics underlying the work of Enright and Willenbring (2004).
\end{abstract}



\maketitle

\tableofcontents

\section{Introduction}

\subsection{\!\!\!}
\label{sub:first intro}
An \emph{almost self-conjugate (ASC) partition} is a weakly decreasing tuple of positive integers whose Young diagram has a special shape: 
for each box on the main diagonal (marked with a dot in the picture below), its \emph{arm} (the part of its row strictly to its right) is longer than its \emph{leg} (the part of its column strictly below it) by exactly one box.
For example, $(7,5,4,2,1,1)$ is an ASC partition, as is apparent from its Young diagram below:

\[
\ytableausetup{smalltableaux}
\begin{ytableau}
\bullet &{}&{}&{}&{}&{}&{}&\none[\scriptstyle \rightarrow]&\none[\scriptstyle 6]\\
{}&\bullet&{}&{}&{}&\none[\scriptstyle\rightarrow]&\none[\scriptstyle 3]\\
{}&{}&\bullet&{}&\none[\scriptstyle\rightarrow]&\none[\scriptstyle 1]\\
{}&{}&\none[\scriptstyle\downarrow]\\
{}&\none[\scriptstyle\downarrow]&\none[\scriptstyle 0]\\
{}&\none[\scriptstyle 2]\\
\none[\scriptstyle\downarrow]\\
\none[\scriptstyle 5]
\end{ytableau}
\]
We follow~\cite{Dong} in adopting the ``ASC'' terminology; these partitions also feature in~\cite{Linusson}, where they are called ``shift-symmetric'' partitions.  
Partitions with the conjugate shape (i.e., where each arm is one box \emph{shorter} than its corresponding leg) are known in the literature as ``threshold partitions,'' since they are precisely the partitions that can be realized as the degree sequence of a threshold graph; see~\cite{Hammer}*{Lemma 10}.
Our own interest in ASC partitions arose from their appearance in symmetric function identities due to Dudley Littlewood, and in related BGG complexes.
In a sense, these complexes are the ``natural habitat'' for ASC partitions $\pi$ and their conjugates $\pi'$. 
In this paper, we classify the BGG complexes acting as the natural habitat for a generalization of the ASC partitions, namely, partitions for which the arm--leg difference is an arbitrary nonnegative integer $m$.
In recent work~\cites{AK,Albion,JW25}, these objects are called \emph{$(-m)$-asymmetric partitions}.

We begin by recalling three classical identities which will be a recurring theme of this paper, each identity involving the Schur polynomials $s_\pi$. 
First we have the dual Cauchy identity~\cite{Stanley}*{Thm.~7.14.3}:
\begin{equation}
\tag{I}
\label{Dual-Cauchy}
    \prod_{\mathclap{\substack{1 \leq i \leq p, \\ 1 \leq j \leq q \phantom{,}}}} \: (1+x_i y_j) = \sum_{\mathclap{\pi \in \Par(p \times q)}}s_\pi(x_1,\ldots,x_p) s_{\pi'}(y_1,\ldots,y_q),
\end{equation}
where $\Par(p\times q)$ is the set of partitions whose Young diagram fits inside a $p \times q$ rectangle.  
Next we have two of the Littlewood identities ~\cite{Littlewood}*{\S11.9}:
\begin{align}
\prod_{\mathclap{1 \leq i\leq j \leq n}} \: (1-x_i x_j) &= \sum_{\substack{\pi: \\ \text{$\pi$ is ASC}}} (-1)^{|\pi|/2} s_\pi(x_1,\ldots,x_n), \tag{II}\label{Littlewood-C}\\[1ex]
\prod_{\mathclap{1 \leq i<j \leq n}} \: (1- x_i x_j) &= \sum_{\substack{\pi: \\ \text{$\pi'$ is ASC}}} (-1)^{|\pi|/2} s_{\pi}(x_1,\ldots,x_{n}), \tag{III}\label{Littlewood-D}
\end{align}
where the partitions $\pi$ in each sum have at most $n$ parts.
With the identities~\eqref{Dual-Cauchy}--\eqref{Littlewood-D} in hand, we outline the results and methods of this paper below.

\subsection{Dimension identities}  

We begin by proving two new identities (Theorems~\ref{theorem:ID-dim-GLn} and \ref{thm:dim GLn pairs}) relating the dimensions of certain modules for $\gl_n$ and $\gl_{n+m}$, when the highest weights are partitions whose arm--leg difference is $m$.  
In the special case $m=1$, the highest weights are precisely the ASC partitions appearing in the Littlewood identities~\eqref{Littlewood-C} and \eqref{Littlewood-D}.  
(See Figure~\ref{fig:example intro}, which illustrates an example of the dimension identity in Theorem~\ref{thm:dim GLn pairs}.) 
These dimension identities are interesting in their own right from a combinatorial viewpoint, but they play a larger role later in the paper, in the proof of our main result (Theorem~\ref{thm:Cong and Conj}).  
This work arose from trying to understand when the ratio appearing in certain dimension identities in~\cite{EW} is equal to 1.

\subsection{Generalized BGG resolutions} 
 
 The identities~\eqref{Dual-Cauchy}--\eqref{Littlewood-D} can be viewed as Euler characteristics of the Bernstein--Gelfand--Gelfand (BGG) complex for the trivial representation of each classical group. 
 The following example is an informal preview.

\begin{figure}[ht]
    \centering
    \input{Example_Introduction.tex}
    \caption{Two isomorphic posets from Example~\ref{ex:D4 and C3}.
    Note that the Young diagrams on the right-hand side are precisely the ASC partitions with at most $3$ parts; meanwhile, each Young diagram on the left is the conjugate of its corresponding Young diagram on the right.  
    On the left-hand side, each Young diagram represents the (dual of the) $\gl_4$-module with corresponding highest weight; on the right-hand side, each Young diagram represents a $\gl_3$-module.  
    Each diagram is labeled with the dimension of the corresponding $\gl_4$- or $\gl_3$-module. 
    By viewing the Young diagrams as parabolic Verma modules and the arrows as the canonical maps between them, one can interpret the left (resp., right) poset as the BGG complex of the trivial representation of $\so_8$ (resp., $\sp_6$).}
    \label{fig:example intro}
\end{figure}
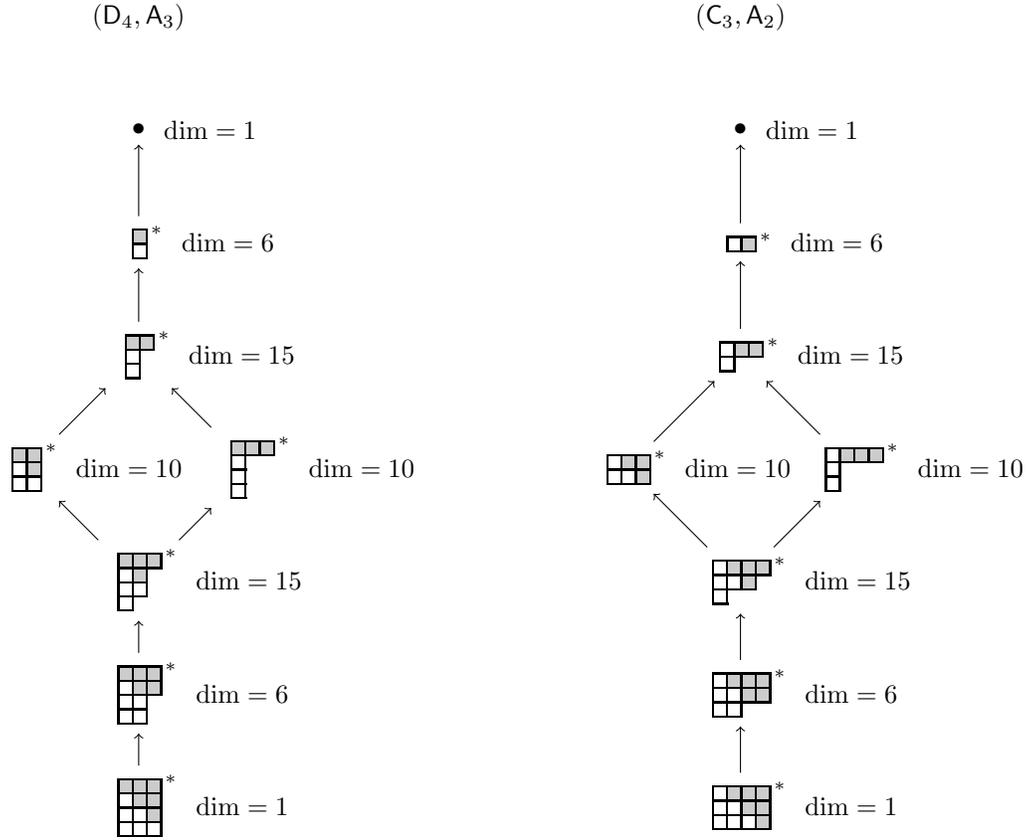

\begin{ex}
\label{ex:D4 and C3}
Throughout this example, we refer to Figure~\ref{fig:example intro}, which shows two posets. 
The poset elements are highest weights for $\gl_4$ and $\gl_3$, respectively.  
On the left side we consider the Hermitian symmetric pair $(\g,\k) = (\ssD_4,\ssA_3) = (\so_8, \gl_4)$; see Section~\ref{sub:Hermitian pairs} for the general theory of Hermitian symmetric pairs $(\g,\k)$ and parabolic subalgebras $\q$ of Hermitian type.  
The poset shown on the left-hand side represents the BGG--Lepowsky complex of the trivial representation of $\so_8$; that is to say, the complex is a free resolution in terms of parabolic Verma modules $N_{\pi^*} \coloneqq U(\g) \otimes_{U(\q)} F_{\pi^*}$, where $F_{\pi^*}$ is the (dual of the) simple $\gl_4$-module whose highest weight is the partition $\pi$.  
In the figure, we represent each parabolic Verma module $N_{\pi^*}$ by the Young diagram of the partition $\pi$ (decorated with the symbol $*$).
The empty diagram $\bullet$ on top therefore represents the first term $N_0$ in the resolution, while the diagram just below it represents the second term $N_{(1,1)^*}$, and so forth. 
Two Young diagrams at the same level in the resolution should be understood as the direct sum of the corresponding parabolic Verma modules.  
Each arrow denotes the canonical map between parabolic Verma modules.  
We also label the Young diagram of each $\pi$ with the dimension of the corresponding $\gl_4$-module $F_{\pi^*}$.  Similarly, on the right-hand side of Figure~\ref{fig:example intro}, we consider the Hermitian symmetric pair $(\ssC_3,\ssA_2) = (\sp_6,\gl_3)$.
Just as on the left, the poset represents the BGG complex of the trivial representation of $\sp_6$, where this time the Young diagrams are highest weights for $\gl_3$.

From Figure~\ref{fig:example intro}, we first observe that the poset for $(\ssD_4,\ssA_3)$ is isomorphic to that for $(\ssC_3, \ssA_2)$.
Second, each partition $\pi$ appearing in the left-hand poset occupies the same position as its conjugate partition $\pi'$ in the right-hand poset (i.e., the Young diagrams are transposes of each other).
More specifically, the partitions appearing on the left-hand side are precisely the partitions $\pi$ occurring in the Littlewood identity~\eqref{Littlewood-D}, for $n=4$;
likewise, their conjugate partitions appearing on the right-hand side are precisely the ASC partitions appearing in~\eqref{Littlewood-C}, for $n=3$.  
As it turns out, the alternating sums in~\eqref{Littlewood-D} and \eqref{Littlewood-C} are the Euler characteristics of these two BGG complexes.
We make another conspicuous observation: not only do the corresponding $\gl_4$- and $\gl_3$-modules have highest weights that are conjugate to each other, but they also have the same dimension.
This equality of dimensions is the aforementioned special case of Theorem~\ref{thm:dim GLn pairs}.

The upshot of these informal observations --- namely, the poset isomorphism $\pi \mapsto \pi'$ that preserves BGG complexes and the dimension of the $\k$-modules --- describes a phenomenon that we will make rigorous in Section~\ref{sec:Congruence}, by means of the notion of \emph{congruence of blocks}.
Using this language, the present example is a special case of the fact that the principal blocks for $(\ssD_{n+1},\ssA_n)$ and $(\ssC_n,\ssA_{n-1})$ are congruent.
\end{ex}

\subsection{Diagrams of Hermitian type} 

For us, the striking fact in Example~\ref{ex:D4 and C3} is that in this instance of congruence, the poset isomorphism is defined by taking conjugate partitions.  
Our goal in this project was to find other congruences of blocks (in the context of Hermitian symmetric pairs) with this intriguing property.  
We thus build upon the work of Armour \cite{Armour}, who observed the appearance of conjugate partitions in congruences of singular and regular blocks.  
To illustrate the main idea behind Section~\ref{s:Diagrams of Hermitian Type}, we have shaded the Young diagrams in Figure~\ref{fig:example intro} to show how they can be constructed via a ``stacking'' operation.  
On either side of Figure~\ref{fig:example intro}, if we consider only the shaded boxes of the diagrams, then moving downward in the Hasse diagram adds one shaded box, such that the shaded boxes form a shifted Young diagram.
 Note that the shaded part of each diagram is the same in both posets.  
 On the right-hand side, to make the Young diagram of $\pi$ itself, we stack the shaded diagram ``horiztonally,'' to the right of its transpose (the white boxes).  
 This stacking construction automatically produces an ASC partition $\pi$. 
 Likewise, on the left-hand side of Figure~\ref{fig:example intro}, each shaded diagram is stacked ``vertically,'' above its transpose, thereby producing the conjugate of an ASC partition.

\subsection{Congruence of blocks and conjugate partitions}  

Our main result (Theorem~\ref{thm:Cong and Conj}, summarized in Table~\ref{table:WC}) exhibits six infinite families of congruent blocks given by conjugate partitions, just as in Example~\ref{ex:D4 and C3}.  
The congruence in Example~\ref{ex:D4 and C3} is a special subfamily --- in some sense, a degenerate case --- in which the two resolutions are for finite-dimensional modules; in general, the complex on the left is the resolution of an infinite-dimensional $\g$-module (see \cite{EW}).  
In verifying the six families of congruences in Table~\ref{table:WC}, our primary tool is the process of \emph{Enright--Shelton reduction}. 
This reduction, which has a strong combinatorial flavor, produces a poset isomorphism by deleting certain coordinates in a weight of $\g$; the result of the reduction is therefore a weight of a Lie algebra $\g'$ with generally smaller rank than $\g$.  
(Even in Example~\ref{ex:D4 and C3}, we can see a sort of proto-Enright reduction from $\ssD_4$ to $\ssC_3$.)

\subsection{Hilbert series and generalized Littlewood identities} 

As an application of our main result, we write down explicit Hilbert series for the infinite-dimensional modules in each of our six families (see Table \ref{table:Hilbert-series}).  
For certain of the families --- namely, those in which the infinite-dimensional module is a Wallach representation of $\g$ --- we thereby recover the well-known Hilbert series of determinantal varieties.
By computing the Euler characteristic of BGG resolutions of the finite-dimensional $\g'$-modules, we derive six new families of identities (see Table~\ref{table:identities}) generalizing the classical identities~\eqref{Dual-Cauchy}--\eqref{Littlewood-D}.

 \subsection{Open problems and related work}  

Our dimension identities in Theorems~\ref{theorem:ID-dim-GLn} and \ref{thm:dim GLn pairs} raise further questions that we leave as open problems in Section~\ref{sec:open probs}.
 In a preliminary version of this paper, we posed the problem of finding bijective proofs of the dimension identities;
 this problem was since solved by Kumari~\cite{Kumari}, using Krattenthaler's bijective proof~\cite{KrattenthalerHCF} of Stanley's hook--content formula~\cite{Stanley}.
 Another natural problem is to find (and ultimately classify) other equalities among the dimensions of $\gl$-modules.  
 This problem is similar in flavor to that of classifying the equalities among binomial coefficients, in the work of Lind~\cite{Lind}, Singmaster~\cite{Singmaster}, and de Weger~\cite{deWeger}.
 Our main result in this paper also leads to the problem of classifying \emph{all} instances of congruent blocks in the context of Hermitian symmetric pairs; see Figure~\ref{fig:sporadic example} for one ``sporadic'' example lying outside the six families mentioned above.
 
The algebraic combinatorics literature is replete with variations of tableau- and poset-based approaches to representation theory; see~\cites{Kwon2018,HoweKimLee}, for example.
We especially highlight the papers by Sam--Snowden--Weyman~\cites{Sam-Weyman-2013,Sam-Weyman-2015}, which take a different approach to viewing the Littlewood identities as the Euler characteristics of BGG complexes.
Also related is the very recent preprint~\cite{Schreier-Aigner}, which studies the Littlewood identities from the perspective of growth diagrams. 

\section{Dimension identities for $\gl_{n}$- and $\gl_{n+m}$-modules}
\label{sec:Dimension IDs}

\subsection{Partitions}

A \emph{partition} is a finite, weakly decreasing sequence $\pi = (\pi_1, \ldots, \pi_\ell)$ of positive integers.
We write $|\pi| \coloneqq \sum_i \pi_i$.
Often we fix some $n \geq \ell$, and view $\pi$ as the $n$-tuple $(\pi_1, \ldots, \pi_\ell, 0, \ldots, 0)$.
In this setting, we write $\pi^*$ to denote the $n$-tuple obtained by negating and reversing the coordinates:
\begin{equation}
    \label{star notation}
    \pi^* \coloneqq (0, \ldots, 0, -\pi_\ell,\ldots,-\pi_1).
\end{equation}

The \emph{Young diagram} of a partition $\pi$ is a left-justified arrangement of boxes such that the $i$th row from the top contains $\pi_i$ many boxes.
We often identify a partition with its Young diagram.
Note that $|\pi|$ is the number of boxes in the Young diagram of $\pi$.  We define the \emph{rank}, denoted by $\rk\pi$, to be the length of the main diagonal of the Young diagram of $\pi$.
We write $\Par(p \times q)$ for the set of partitions whose Young diagram fits inside a rectangle with $p$ rows and $q$ columns.
If we do not wish to restrict the number of columns, then we write $\Par(p \times \infty)$.
The \emph{conjugate} partition of $\pi$ is the partition whose Young diagram is that of $\pi$ reflected about the main diagonal; 
we write $\pi'$ to denote the conjugate of $\pi$ (but see our disclaimer at the beginning of Section~\ref{sec:Congruence}).
We use the symbol $\bullet$ to denote the empty Young diagram, corresponding to the zero partition $0 \coloneqq (0, \ldots, 0)$.
We say that a partition has \emph{even rows} (resp., \emph{columns}) if all rows (resp., columns) of its Young diagram contain an even number of boxes.
Upon filling the boxes in a Young diagram with entries, we call the resulting object a \emph{tableau}.

A partition $\pi$ of rank $r$ can be uniquely described by its arm lengths $\al_1> \cdots > \al_r \geq 0$ and leg lengths $\be_1 > \cdots > \be_r \geq 0$, as follows.
Define $\al_i$ to be the number of boxes in the $i$th row strictly to the right of the main diagonal in the Young diagram of $\pi$; 
likewise, define $\be_j$ to be the number of boxes in the $j$th column strictly below the main diagonal.
In this way, one can denote a partition by its \emph{Frobenius symbol}, writing 
\[
\pi = (\al | \be) = (\al_1, \ldots, \al_r \mid \be_1, \ldots, \be_r).
\]
If $\pi = 0 = (\; \mid \;)$, then $\al$ and $\be$ are empty. Clearly if $\pi = (\al|\be)$, then $\pi' = (\be|\al)$.  
For a nonnegative integer $m$, we adopt the shorthand 
\[
\al + m \coloneqq (\al_1 + m, \ldots, \al_r + m).
\]  
For example, if $\pi = (\al | \be)$, then $(\al + m \mid \be)$ denotes the partition obtained by adding $m$ boxes to all the arms of $\pi$.
We observe that an ASC partition, introduced informally in Section~\ref{sub:first intro}, is a partition of the form $(\al + 1 \mid \al)$.

\subsection{Dimension identities}

Throughout the paper, we let $\F{\mu}{n}$ denote the finite-dimensional simple $\gl_n$-module with highest weight $\mu$, where $\mu$ is a weakly decreasing $n$-tuple of integers.  
In the proof below, we will write $(i,j) \in \pi$ to denote the box in the $i$th row (from the top) and $j$th column (from the left) of the Young diagram of a partition $\pi$.  
From each box $(i,j) \in \pi$ there emanates a \emph{hook}, which consists of all the boxes weakly to the right in row $i$ or weakly below in column $j$.
The corresponding \emph{hook length}, denoted by $h(i,j)$, is the number of boxes in the hook emanating from $(i,j)$.
The \emph{content} of a box is defined as $c(i,j) \coloneqq j-i$.  

\begin{theorem}
\label{theorem:ID-dim-GLn}
    Let $\pi = (\al | \be) \in \Par(p\times q)$, and let $m$ be a nonnegative integer.
    We have
    \begin{equation}
    \label{ID-dim-GLn}
    \dim \F{(\al+m\mid \be)}{p} \dim \F{(\be+m|\al)}{q} = \dim \F{(\al|\be+m)}{p+m} \dim \F{(\be|\al+m)}{q+m}.
    \end{equation}
\end{theorem}

\begin{proof}
    By the hook--content formula~\cite{Stanley}*{Thm.~7.21.2}, we have
    \begin{equation}
    \label{hookcontentformula}
        \dim \F{\pi}{n} = \prod_{(i,j) \in \pi} \frac{n + c(i,j)}{h(i,j)}.
    \end{equation}
    We first rewrite the product of the numerators in~\eqref{hookcontentformula} in terms of hooks: 
    letting $h_k$ denote the hook emanating from box $(k,k)$, we have
    \[
    \prod_{(i,j) \in h_k} (n + c(i,j)) = \prod_{\ell = -\be_k}^{\al_k} (n + \ell) = \frac{(n+\al_k)!}{(n-\be_k-1)!}
    \]
    and so, putting $r = \rk \pi$, we have
    \begin{equation}\label{hookrewrite}
    \prod_{(i,j) \in \pi} (n + c(i,j)) = \prod_{k=1}^r \frac{(n+\al_k)!}{(n-\be_k-1)!}.
    \end{equation}
Note that $\rk(\alpha+m \mid \beta) = \rk(\alpha \mid \beta+m) = \rk(\alpha|\beta) = r$.
Thus, using~\eqref{hookcontentformula} and~\eqref{hookrewrite}, we can rewrite the left-hand side of~\eqref{ID-dim-GLn} as
\begin{equation}\label{LHS}
\frac{\displaystyle\prod_{k =1}^r \frac{(p+\al_k + m)!}{(p-\be_k - 1)!}}{\displaystyle\prod_{\mathclap{(i,j) \in (\al+m|\be)}} h(i,j)} \: \frac{\displaystyle\prod_{k =1}^r \frac{(q+\be_k + m)!}{(q-\al_k -1)!}}{\displaystyle\prod_{\mathclap{(i',j') \in (\be+m|\al)}} h(i',j')},
\end{equation}
and the right-hand side of~\eqref{ID-dim-GLn} as
\begin{equation}    \label{RHS}
\frac{\displaystyle\prod_{k =1}^r \frac{(p+m+\al_k)!}{(p+m-\be_k-m - 1)!}}{\displaystyle\prod_{\mathclap{(i',j') \in (\al|\be+m)}} h(i',j')} \:
\frac{\displaystyle\prod_{k =1}^r \frac{(q+m +\be_k)!}{(q+m-\al_k - m -1)!}}{\displaystyle\prod_{\mathclap{(i,j) \in (\be|\al+m)}} h(i,j)}.
\end{equation}
Clearly the numerator in~\eqref{LHS} equals that in~\eqref{RHS}. 
 Moreover, since $(\al \mid \be + m) = (\be + m \mid \al)'$, and since the multiset of hook lengths is preserved under conjugation of partitions, the denominators in~\eqref{LHS} and~\eqref{RHS} are also equal. 
 Hence the left- and right-hand sides of~\eqref{ID-dim-GLn} are equal, and the theorem follows.
\end{proof}

\begin{theorem}
\label{thm:dim GLn pairs}
Let $m$ be a nonnegative integer.
Then $\dim \F{\pi}{n} = \dim \F{\pi'}{n+m}$ for every nonnegative integer $n$ if and only if $\pi$ has the form $(\al + m \mid \al)$.
\end{theorem}

\begin{proof}
    Suppose there is some $\alpha$ such that $\pi = (\al + m \mid \al)$.
    In Theorem~\ref{theorem:ID-dim-GLn}, by setting $\al = \be$ and $p=q=n$, and then taking the square root of both sides of~\eqref{ID-dim-GLn}, we have $\dim \F{\pi}{n} = \dim \F{\pi'}{n+m}$.
    Conversely, suppose that $\dim \F{\pi}{n} = \dim \F{\pi'}{n+m}$ for every nonnegative integer $n$.
    Let $\pi = (\al | \be)$.
    Since conjugate partitions have the same multiset of hook lengths, and since $c(j,i) = -c(i,j)$, the hook--content formula~\eqref{hookcontentformula} yields
\[
\prod_{(i,j) \in \pi} n + c(i,j) = \prod_{(i,j) \in \pi} n+m - c(i,j).
\]
Treating each side as a polynomial in $n$, by unique factorization we must have the equality of multisets
\begin{equation}
\label{multiset-content}
C_1 \coloneqq \{c(i,j) \mid (i,j) \in \pi\} = \{ m - c(i,j) \mid (i,j) \in \pi\}.
\end{equation}
Thus $\al_1 = \max(C_1)$ and $-\be_1 = \min(C_1)$. But by~\eqref{multiset-content}, we must also have $\min(C_1) = m -\al_1 = -\be_1$, and thus $\al_1 = \be_1 + m$. 
Since the outermost hook $h_1$ of $\pi$ contains exactly one box with content $c$ for each $-\be_1 \leq c \leq \al_1$, we delete these contents from $C_1$ to obtain the new multiset
\[
C_2 \coloneqq C_1 \setminus \{-\be_1, -\be_1 + 1, \ldots, \al_1\}.
\]
Just as before, $\al_2 = \max(C_2)$ and $-\be_2 = \min(C_2) = m - \al_2$, and so $\al_2 = \be_2 + m$.  
Continuing in this way to define each new multiset $C_{i+1} \coloneqq C_i \setminus \{ -\be_i, -\be_i + 1, \ldots, \al_i\}$, we obtain $\al_i = \be_i + m$ for all $1 \leq i \leq \rk\pi$. 
 Therefore $\pi = (\be + m \mid \be)$, which completes the proof.
 \end{proof}

\begin{rem}
We also have the following $q$-analogue of Theorem~\ref{thm:dim GLn pairs}:
\[
s_{\pi}(q^{n-1}, q^{n-3}, \ldots, q^{3-n}, q^{1-n}) = s_{\pi'}(q^{n+m-1}, q^{n+m-3}, \ldots, q^{3-n-m}, q^{1-n-m})
\]
when $\pi = (\al + m \mid \al)$.
In terms of representation theory, this means that the modules in Theorem~\ref{thm:dim GLn pairs} not only have the same dimension, but are equivalent as $\sl_2$-modules upon the restriction of $\gl_n$ and $\gl_{n+m}$ to their principal $\sl_2$ subalgebra.
\end{rem}

\section{Generalized BGG resolutions}
\label{section:ID's and BGG}

\subsection{\!\!\!} With Theorems~\ref{theorem:ID-dim-GLn} and~\ref{thm:dim GLn pairs} in mind, we  recall the three classical identities~\eqref{Dual-Cauchy}, \eqref{Littlewood-C}, and \eqref{Littlewood-D} from the introduction. 
We observe the similarity between~\eqref{Dual-Cauchy} and Theorem~\ref{theorem:ID-dim-GLn} in the case $m=0$, since the two highest weights on the left-hand side of~\eqref{ID-dim-GLn} are conjugates of each other, with the first being an element of $\Par(p \times q)$.  
Likewise the sums in the Littlewood identities~\eqref{Littlewood-C} and \eqref{Littlewood-D} range over the same ASC partitions that appear as highest weights for the $\gl_n$-modules in Theorem~\ref{thm:dim GLn pairs}, in the case $m=1$. 
In the remainder of this section, we will explain how each of the identities~\eqref{Dual-Cauchy},~\eqref{Littlewood-C}, and~\eqref{Littlewood-D} can be viewed as the Euler characteristic of the BGG complex of the trivial representation; see also~\cites{Sam-Weyman-2013,Sam-Weyman-2015} for a different approach.

\subsection{Hermitian symmetric pairs}
\label{sub:Hermitian pairs}

Let $\g_{\R}$ be a real simple noncompact reductive Lie algebra, with Cartan decomposition $\g_{\R} = \k_{\R} \oplus \p_{\R}$.
We write the complexified Cartan decomposition $\g = \k \oplus \p$.
From the general theory, there exists a distinguished element $h_0 \in \mathfrak{z}(\k)$ such that $\operatorname{ad} h_0$ acts on $\g$ with eigenvalues $0$ and $\pm 1$.
We thus have a triangular decomposition $\g = \p^- \oplus \k \oplus \p^+$, where $\p^{\pm} = \{ x \in \g \mid [h_0, x] = \pm x\}$.
The subalgebra $\q = \k \oplus \p^+$ is a maximal parabolic subalgebra of $\g$, with Levi subalgebra $\k$ and abelian nilradical $\p^+$.  
Parabolic subalgebras of complex simple Lie algebras that arise in this way are called parabolic subalgebras of \emph{Hermitian type}, and $(\g,\k)$ is called a \emph{Hermitian symmetric pair}.

In this paper, we focus our attention on the three families of Hermitian symmetric pairs that arise in the dual pair setting~\cite{Howe89}.
We call these Types I, II, and III, as follows:
\[
\renewcommand{\arraystretch}{1.5}
\begin{array}{|c|c|c|c|}
\hline
    & (\g_\R, \k_\R) & (\g, \k) & \text{Cartan type} \\ \hline
    \text{Type I}  & (\su(p,q), \: \mathfrak{s}(\mathfrak{u}(p) \oplus \mathfrak{u}(q))) &(\sl_{p+q},\:\mathfrak{s}(\gl_p \oplus \gl_q)) & (\ssA_{p+q-1},\:\ssA_{p-1}\times\ssA_{q-1}) \\ \hline

    \text{Type II}  & (\sp(2n,\R), \:\mathfrak{u}(n)) & (\sp_{2n},\:\gl_n) & (\ssC_n,\:\ssA_{n-1}) \\ \hline

    \text{Type III}  & (\so^*(2n), \: \mathfrak{u}(n)) & (\so_{2n},\:\gl_n) & (\ssD_n,\:\ssA_{n-1})

    \\ \hline
\end{array}
\]
For each type above, we give explicit realizations of $\g$, $\k$, and $\p^+$ on the left-hand side of Table~\ref{table:Type123}.  In the $\g$ column, for Type I, the block matrix $\left[\begin{smallmatrix}A&B\\C&D\end{smallmatrix}\right]$ has dimensions $(p+q)\times(p+q)$, while for Types II and III it has dimensions $2n\times 2n$.
In the $\p^+$ column, we write $\M_{p,q}$ for the space of complex $p \times q$ matrices, while $\SM_n$ (resp., $\AM_n$) denotes the symmetric (resp., alternating) complex $n \times n$ matrices.  In the next column, we write down the character of $S(\p^-) \cong \C[\p^+]$ as a $\k$-module, where
\begin{equation}
    \label{x_i y_j}
    x_i \coloneqq e^{-\ep_i} \quad \text{and} \quad y_j \coloneqq e^{\ep_j},
\end{equation}
with $e$ a formal indeterminate and with $\ep_i$ to be defined in the following subsection.
In Table~\ref{table:Type123}, we also include the well-known expansions of these characters in terms of Schur polynomials:
in Type I, this yields the Cauchy identity, while in Types II and III the identities were recorded by Littlewood~\cite{Littlewood}*{\S11.9} on the same page as his identities~\eqref{Littlewood-C} and~\eqref{Littlewood-D}.

\begin{table}
\centering
\input{Table_basics.tex}
\caption{Summary of data for Hermitian symmetric pairs of Types I, II, and III.
For Type I, we write $\bfx = (x_1, \ldots, x_p)$ and $\bfy = (y_1, \ldots, y_q)$; for Types II and III, $\bfx = (x_1, \ldots, x_n)$.}
\label{table:Type123}
\end{table}

\subsection{Roots and weights}

Suppose $(\g,\k)$ is a Hermitian symmetric pair, and let $\h$ be a Cartan subalgebra of both $\g$ and $\k$.
Let $\Phi$ be the root system of the pair $(\g,\h)$, and $\g_\al$ the root space corresponding to $\al \in \Phi$. 
Then put $\Phi(\k) \coloneqq \{ \al \in \Phi \mid \g_\al \subseteq \k\}$ and $\Phi(\p^+) \coloneqq \{ \al \in \Phi \mid \g_\al \subseteq \p^+\}$.  
 Choose a set $\Phi^+$ of positive roots so that $\Phi(\p^+)\subseteq \Phi^+$, and let $\Phi^- \coloneqq -\Phi^+$ denote the negative roots.
 We write $\Phi^{\!+\!}(\k)$ for $\Phi^+ \cap \Phi(\k)$.
 Let $\Pi 
 = \{\al_1,\ldots,\al_r\} \subset \Phi^+$ denote the set of simple roots.  
 We write $\langle \; , \; \rangle$ to denote the nondegenerate bilinear form on $\h^*$ induced from the Killing form of $\g$.  
 For $\al \in \Phi$, we write $\al^\vee \coloneqq 2\al / \langle \al,\al \rangle$.  
 We define the fundamental weights $\omega_i$ such that $\langle \omega_i,\al_j^\vee\rangle = \delta_{ij}$.
 We let $\rho \coloneqq \sum_{i} \omega_i$.

In the list below, we give all these weights in explicit coordinates, where $\ep_i \in \h^*$ is the functional mapping a diagonal matrix to its $i$th diagonal entry.
In general, we express weights in $\h^*$ as tuples in these $\ep$-coordinates:
thus in Type I we write a weight as $(a_1, \ldots, a_p ; b_1, \ldots, b_q)$, while in Types II and III we write a weight as $(a_1, \ldots, a_n)$.

\begin{alignat*}{3}
 & \text{Type I ($\g = \sl_{p+q}$):} \quad  &  
\Phi^{\!+\!}(\k) &= \{ \ep_i - \ep_j \mid 1 \leq i < j \leq p \: \text{ or } \: p+1 \leq i < j \leq p+q\},  \\
    &  & \Phi(\p^+) & = \{ \ep_i - \ep_j \mid 1 \leq i \leq p \: \text{
 and } \: p+1 \leq j \leq p+q\},\\
     &  & \Pi & = \{\al_i = \ep_i - \ep_{i+1} \mid 1 \leq i \leq p+q-1\},\\
     & & \omega_i & = \ep_1 + \cdots + \ep_i, \quad \text{for } 1 \leq i \leq p + q -1,\\
     & & \rho & = (p+q-1, \ldots, q+1, q; \, q-1, \ldots, 2, 1, 0).\\[1em]
& \text{Type II ($\g = \sp_{2n}$):} \quad &  
\Phi^{\!+\!}(\k) & = \{ \ep_i - \ep_j \mid 1 \leq i < j \leq n \},  \\
    & & \Phi(\p^+) & = \{ \ep_i + \ep_j \mid 1 \leq i \leq j \leq n \},\\
    & & \Pi & = \{ \al_i = \ep_i - \ep_{i+1} \mid 1 \leq i \leq n-1\} \cup \{ \al_n = 2 \ep_n\},\\
    & & \omega_i & = \ep_1 + \cdots + \ep_i, \quad \text{for } 1 \leq i \leq n,\\
    & & \rho & = (n, \ldots, 3, 2, 1).\\[1em]
& \text{Type III ($\g = \so_{2n}$):} \quad  &  
\Phi^{\!+\!}(\k) & = \{ \ep_i - \ep_j \mid 1 \leq i < j \leq n \},  \\
    & & \Phi(\p^+) & = \{ \ep_i + \ep_j \mid 1 \leq i < j \leq n \},\\
  & & \Pi & = \{\al_i = \ep_i - \ep_{i+1} \mid 1 \leq i \leq n-1\} \cup \{ \al_n = \ep_{n-1}+\ep_n\},\\
  & & \omega_i & = \begin{cases}
      \ep_1 + \cdots + \ep_i,& 1 \leq i \leq n-2,\\
      \frac{1}{2}(\ep_1 + \cdots + \ep_{n-1} - \ep_n), & i = n-1,\\
      \frac{1}{2}(\ep_1 + \cdots + \ep_n), & i = n,
  \end{cases} \\
  & & \rho & = (n-1, \ldots, 2, 1, 0).
\end{alignat*}

Note that $\Phi(\p^+)$ inherits the usual poset structure from $\h^*$, where $\mu \leq \la$ if and only if $\la - \mu$ can be written as a nonnegative integer combination of positive roots.  
There is a unique element of $\Pi \cap \Phi(\p^+)$, namely the noncompact simple root, which is the smallest element of the poset $\Phi(\p^+)$.

Let $\W$ be the Weyl group of the pair $(\g,\h)$, and $\W(\k) \subseteq \W$ the Weyl group of the pair $(\k,\h)$.
For each $w \in \W$, let 
\[
\Phi_w \coloneqq \Phi^+ \cap w\Phi^-.
\]
Then we have the usual length function $\ell$ on $\W$, whereby 
\begin{equation}
\label{length-size-Delta-w}
\ell(w) = |\Phi_w|.
\end{equation}
Following Kostant~\cite{Kostant}*{(5.13.1)}, we define
\[
\kW \coloneqq \{ w \in \W \mid \Phi_w \subseteq \Phi(\p^+)\},
\]
the subset of minimal-length right coset representatives of $\W(\k)$ in $\W$.
To refine $\kW$ by length, we further define
\[
\kW_i \coloneqq \{w \in \kW \mid \ell(w) = i\}.
\]
The ``dot'' action by the Weyl group is defined as follows:
\[
w \cdot \la \coloneqq w(\la + \rho) - \rho
\]
for $w \in \W$ and $\la \in \h^*$.
We define the reflection $s_\al:\h^* \longrightarrow \h^*$ by $s_\al(\la) \coloneqq \la - \langle \la, \al^\vee \rangle \al$.
Let 
\begin{align*} 
    \Lambda^+ &\coloneqq \{ \la \in \h^* \mid \langle \la+\rho, \: \al^\vee\rangle \in \mathbb{Z}_{>0} \text{ for all } \al \in \Phi^+\},\\
    \Lplusk &\coloneqq \{ \la \in \h^* \mid \langle \la+\rho, \: \al^\vee\rangle \in \mathbb{Z}_{>0} \text{ for all } \al \in \Phi^{\!+\!}(\k)\}
\end{align*}
denote the sets of dominant integral weights with respect to $\Phi^+$ and $\Phi^{\!+\!}(\k)$, respectively.

\subsection{BGG--Lepowsky resolutions}

For $\mu \in \Lplusk$, let $F_{\mu}$ be the finite-dimensional simple $\k$-module with highest weight $\mu$.
Then $F_{\mu}$ is also a module for $\q = \k\oplus\p^+$, with $\p^+$ acting by zero.
We define the \emph{parabolic Verma module}
\begin{equation}
    \label{Verma}
    N_\la \coloneqq U(\g) \otimes_{U(\q)} F_{\la}.
\end{equation}
When $(\g,\k)$ is a Hermitian symmetric pair, $\p^-$ is abelian and therefore we can identify $U(\p^-)$ with $S(\p^-)$.  
By the PBW theorem, we thus obtain
\begin{equation}
    \label{N lambda}
    N_\la \cong S(\p^-) \otimes F_{\la}
\end{equation}
as a $\k$-module.
Recall that the character of $S(\p^-)$ is given in Table~\ref{table:Type123}.
In Type I, upon restriction to $\k = \mathfrak{s}(\gl_p \oplus \gl_q)$, a weight $(\mu;\nu)$ becomes an ordered pair $(\mu,\nu)$, which (assuming it lies in $\Lplusk$) we will also write as $\mu \otimes \nu$, reflecting the fact that $F_{\mu \otimes \nu} =  \F{\mu}{p} \otimes \F{\nu}{q}$.
Due to our conventions~\eqref{star notation} and~\eqref{x_i y_j} regarding highest weights of contragredient representations, we have $F_{\mu^*} = (F_\mu)^*$, and the characters of finite-dimensional simple $\k$-modules are given by Schur polynomials as follows, where $\mu$ and $\nu$ are partitions:
\begin{alignat}{3}
    &\text{Type I:} \qquad & \ch F_{\mu^* \otimes \nu} & = \ch \left( \F{\mu^*}{p} \otimes \F{\nu}{q} \right) = s_\mu(\bfx) s_\nu(\bfy), \label{characters as Schur polys}\\
    &\text{Types II and III:} \qquad & \ch F_{\mu^*} & = \ch \F{\mu^*}{n} = s_\mu(\bfx), \nonumber
\end{alignat}
where $\bfx = (x_1, \ldots, x_p)$ and $\bfy = (y_1, \ldots, y_q)$ in Type I, and $\bfx = (x_1, \ldots, x_n)$ in Types II and III.

Let $\la \in \Lambda^+$, and let $L_{\la}$ denote the finite-dimensional simple $\g$-module with highest weight $\la$. 
Generalizing the Bernstein--Gelfand--Gelfand (BGG) resolution in terms of ordinary Verma modules~\cite{BGG}, Lepowsky~\cite{Lepowsky} showed that there exists a resolution for $L_{\la}$ in terms of parabolic Verma modules: 
\[
0 \longrightarrow N_s \longrightarrow N_{s-1} \longrightarrow \cdots \longrightarrow N_1 \longrightarrow N_0 \longrightarrow L_{\la} \longrightarrow 0,
\]
where
\[
    N_i = \bigoplus_{\mathclap{w \in \kW_i}} N_{w\cdot\lambda}
\]
and $s = |\Phi(\p^+)|$.  
We now state the result (to be proved in Section~\ref{s:Diagrams of Hermitian Type}) that connects these BGG resolutions (where $\la = 0$) with the three classical identities~\eqref{Dual-Cauchy}--\eqref{Littlewood-D} above:

\begin{prop}
\label{prop:BGG}
    In each of Types I--III, the trivial $\g$-module $\mathbbm{1}$ has the generalized BGG resolution
    \[
    0 \longrightarrow N_s \longrightarrow N_{s-1} \longrightarrow \cdots \longrightarrow N_1 \longrightarrow N_0 \longrightarrow \mathbbm{1} \longrightarrow 0,
    \]
with the terms given as follows:
\[
\def\arraystretch{2.5}
\begin{array}{lll}

\text{\normalfont{Type I:}} &
\displaystyle N_i = \bigoplus_{\substack{\pi \in \Par(p \times q),\\ |\pi| = i}} N_{\pi^* \otimes \pi'} & \text{\normalfont{and }} s= pq.\\

\text{\normalfont{Type II:}} & \displaystyle N_i = \bigoplus_{\substack{\pi = (\al +1 | \al),\\ \al_1 < n, \\ |\al+1|=i}} N_{\pi^*} & \text{\normalfont{and }} s= \binom{n+1}{2}.\\

\text{\normalfont{Type III:}} & \displaystyle N_i = \bigoplus_{\substack{\pi = (\al|\al+1),\\ \al_1 < n-1,\\ |\al+1|=i}} N_{\pi^*} & \text{\normalfont{and }} s= \binom{n}{2}.

\end{array}
\]

\end{prop}

Observing that $|(\al + 1 \mid \al)| = 2|\al+1|$, and using~\eqref{N lambda} and~\eqref{characters as Schur polys} to take the alternating sum of the characters of the $N_i$'s in Proposition~\ref{prop:BGG}, we obtain the following identities:
\[
\def\arraystretch{3}
\begin{array}{ll}

\text{\normalfont{Type I:}} &
\displaystyle 1 = \ch \mathbbm{1} = \ch S(\p^-) \cdot \sum_{\mathclap{\pi \in \Par(p \times q)}} (-1)^{|\pi|} \ch \left(\F{\pi^*}{p}\otimes\F{\pi'}{q}\right) = \frac{\sum_\pi (-1)^{|\pi|} s_\pi(\bfx)s_{\pi'}(\bfy)}{\prod_{i,j} (1-x_i y_j)}.
\\
\text{\normalfont{Type II:}} &
\displaystyle 1 = \ch \mathbbm{1} = \ch S(\p^-) \cdot \sum_{\mathclap{\substack{\pi=(\al+1|\al),\\ \al_1 < n}}} (-1)^{|\pi|/2} \ch \F{\pi^*}{n} = \frac{\sum_\pi (-1)^{|\pi|/2} s_\pi(\bfx)}{\prod_{i \leq j} (1-x_i x_j)}.
\\
\text{\normalfont{Type III:}} &
\displaystyle 1 = \ch \mathbbm{1} = \ch S(\p^-) \cdot \sum_{\mathclap{\substack{\pi=(\al|\al+1),\\ \al_1 < n-1}}} (-1)^{|\pi|/2} \ch \F{\pi^*}{n} = \frac{\sum_\pi (-1)^{|\pi|/2} s_{\pi}(\bfx)}{\prod_{i < j} (1-x_i x_j)}.
\end{array}
\]
In fact, for Type I, rearranging the equation above yields the dual Cauchy identity~\eqref{Dual-Cauchy} upon substituting $x_i \mapsto -x_i$.
For Types II and III, rearrangement yields the Littlewood identities~\eqref{Littlewood-C} and~\eqref{Littlewood-D}, respectively.

\begin{rem}
    The element $h_0 \in \mathfrak{z}(\k)$ induces a grading on highest weight $\g$-modules, to be defined below in~\eqref{Hilbert series definition}.
    One can obtain the \emph{graded} character of a parabolic Verma module by refining its character via the additional indeterminate $t$.
    Upon passing to graded characters in the BGG resolutions in Proposition~\ref{prop:BGG}, we obtain the refined identities below, where the indices $i,j$ coincide with those in~\eqref{Dual-Cauchy}--\eqref{Littlewood-D}, respectively: 
\[
\frac{\displaystyle\sum_{\pi \in \Par(p \times q)} (-1)^{|\pi|}t^{|\pi|} s_\pi(\bfx) s_{\pi'}(\bfy)}{\displaystyle\prod_{i,j} (1 - tx_i y_i)} = 1,
\]
\[
\frac{\displaystyle\sum_{\text{$\pi: \pi$ is ASC}} (-1)^{|\pi|/2} t^{|\pi|/2} s_{\pi}(\bfx)}{\displaystyle\prod_{i \leq j} (1 - t x_i x_j)} = 1, \qquad \qquad
\frac{\displaystyle\sum_{\text{$\pi: \pi'$ is ASC}} (-1)^{|\pi|/2} t^{|\pi|/2} s_{\pi}(\bfx)}{\displaystyle\prod_{i<j} (1 - t x_i x_j)} = 1.
\]
Rearranging and replacing $t$ with $-t$, we obtain the following refinements of~\eqref{Dual-Cauchy}--\eqref{Littlewood-D}:
\begin{align*}
\prod_{\mathclap{\substack{1 \leq i \leq p, \\ 1 \leq j \leq q \phantom{,}}}} (1 + tx_i y_j) & = \sum_{\pi \in \Par(p \times q)} t^{|\pi|} s_{\pi}(x_1, \ldots, x_p) s_{\pi'}(y_1, \ldots, y_q),\\[1ex]
\prod_{\mathclap{1 \leq i\leq j \leq n}} (1 + t x_i x_j) &= \sum_{\text{$\pi: \pi$ is ASC}} t^{|\pi|/2} s_{\pi}(x_1,\ldots,x_n),\\[1ex]
\prod_{\mathclap{1 \leq i<j \leq n}} (1 + t x_i x_j) &= \sum_{\text{$\pi: \pi'$ is ASC}} t^{|\pi|/2} s_{\pi}(x_1,\ldots,x_n).
\end{align*}
\end{rem}

\section{Diagrams of Hermitian type}
\label{s:Diagrams of Hermitian Type}

\subsection{\!\!\!}
In this section, we introduce certain diagrams that encode the highest weights $w \cdot \la \in \Lplusk$ occurring in BGG resolutions.
We will then be able to  write down all of these weights directly from the diagrams, eliminating the need for calculations in terms of the Weyl group.  
This will lead to a short proof of Proposition~\ref{prop:BGG} above.

\subsection{A recursive formula for $w \cdot \la$}

We follow the exposition in~\cite{EHP}*{\S3.7}, to which we refer the reader for details.
Suppose $(\g,\k)$ is a Hermitian symmetric pair.  Then there exists a unique map $f:\Phi(\p^+)\longrightarrow \Pi$, such that 
\begin{equation}
\label{fdef}
\Phi_w = \Phi_v \cupdot \{\be\} \quad (w,v \in \kW) \quad \Longrightarrow \quad w = vs_{f(\be)} = s_{_\be} v.
\end{equation}
In this case, we have the explicit formula 
\begin{equation}
\label{f-beta}
f(\be) = v^{-1}\be.
\end{equation}  
It follows inductively that if we write $\Phi_w = \{\be_1, \ldots,\be_\ell\}$ such that for every $i = 1, \ldots, \ell$ the set $\{\be_1, \ldots, \be_i\}$ is a lower order ideal of $\Phi(\p^+)$, then
\begin{equation}
    \label{w-in-terms-of-v}
    w = s_{f(\be_1)} \cdots s_{f(\be_\ell)} = s_{\be_\ell} \cdots s_{\be_1}.
\end{equation}
In this paper, the motivation behind the map $f$ is the following observation (generalizing a result of Kostant in the case $\la = 0$):
\begin{lemma}
\label{lemma:w-dot-lambda}
    For $\la \in \Lplusk$ and $w \in \kW$, we have 
    \begin{equation}
        \label{w-dot-la-equation}
    w \cdot \la = \la - \sum_{\mathclap{\be \in \Phi_w}} \big\langle\la+\rho,\:f(\be)^\vee\big\rangle \be.
    \end{equation}
\end{lemma}

\begin{proof}
    We prove this by induction on $\ell(w)$.  In the base case $w = {\rm id}$, we have $\Phi_w = \varnothing$ and so the sum in~\eqref{w-dot-la-equation} is empty, as desired.  
    Now let $v \in \kW$ and assume that~\eqref{w-dot-la-equation} holds for all elements in $\kW$ with length at most $\ell(v)$.  Let $w \in \kW$ such that $\Phi_w = \Phi_v \cupdot \{\be\}$; then $\ell(w) = \ell(v)+1$.  From~\eqref{fdef} we have $w = vs_{f(\be)}$, and so
    \begin{align*}
    w \cdot \la &= w(\la + \rho) - \rho\\
    &= vs_{f(\be)}(\la+\rho) - \rho\\
    &= v\Big(\la+\rho - \langle\la+\rho,\:f(\be)^\vee\rangle f(\be)\Big)-\rho\\
    &= v(\la+\rho) - \langle\la+\rho,\:f(\be)^\vee\rangle v\big(f(\be)\big) - \rho\\
    &= v \cdot \la - \langle\la+\rho, \: f(\be)^\vee\rangle\be,
    \end{align*}
where in the last equality we have used~\eqref{f-beta} to obtain $v(f(\be)) = vv^{-1}\be = \be$.  By the induction hypothesis, we have $v \cdot \la = \la - \sum_{\gamma \in \Phi_v} \langle\la+\rho,\:f(\gamma)^\vee\rangle\gamma$. Since $\Phi_w = \Phi_v \cupdot \{\be\}$, we see that $w$ satisfies~\eqref{w-dot-la-equation}.
\end{proof}

\subsection{Diagrams of Hermitian type}

The proof of Lemma~\ref{lemma:w-dot-lambda} suggests a recursive method for computing $w \cdot \la$, supposing that we know the reduced expression $w=s_{\be_\ell}\cdots s_{\be_1}$ on the right-hand side of~\eqref{w-in-terms-of-v}.  We can do even better, however: in this section, we interpret Lemma~\ref{lemma:w-dot-lambda} diagrammatically, which will enable us in~\eqref{w-dot as boxes} to write down $w \cdot \la$ directly in terms of $\Phi_w$.

Recall that $\Phi(\p^+)$ inherits the usual poset structure from $\h^*$.
Let $[\Phi(\p^+)]$ denote the diagram obtained by rotating the Hasse diagram of $\Phi(\p^+)$ so that its minimal element (i.e., the unique noncompact simple root) is in the northwest corner, and then replacing each root $\be \in \Phi(\p^+)$ by a shaded box:

\input{Empty_diagrams_all}

\noindent In the figures above depicting $[\Phi(\p^+)]$, each box represents the root given by the sum of the epsilons in its row and column labels.
Note that we have $\beta \leq \beta'$ if and only if the box corresponding to $\beta$ lies weakly northwest of the box corresponding to $\beta'$.

From the general theory of Hermitian symmetric pairs, we have an isomorphism of posets
\begin{align}
\label{kW poset}
\begin{split}
\kW & \cong \{\text{lower order ideals in }\Phi(\p^+)\},\\
    w & \mapsto \Phi_w,
\end{split}
\end{align}
with $\kW$ a poset under the Bruhat order, and the set of lower order ideals in $\Phi(\p^+)$ ordered by inclusion.
In particular,  the poset $\kW$ is a distributive lattice with a unique minimal and maximal element.
In fact, for $v,w\in \kW$, one can show that $v \leq w$ if and only if 
$\Phi_v \subseteq \Phi_w$. 
(This implies that the Bruhat order and the weak Bruhat order coincide when restricted to $\kW$.)
For $w \in \kW$, define
\begin{equation}
    \label{Phi_w brackets}
    [\Phi_w] \coloneqq \text{the subdiagram of $[\Phi(\p^+)]$ corresponding to the roots in $\Phi_w$}.
\end{equation}
It follows from~\eqref{kW poset} that $[\Phi_w]$ is a Young diagram (Type I) or a \emph{shifted} Young diagram (Types II and III) sitting in the northwest corner of $[\Phi(\p^+)]$.
We will view $[\Phi_w]$ both as a Young diagram and as the partition given by its row lengths; in particular, $[\Phi_w]_i$ denotes the number of boxes in the $i$th row. 
We observe from~\eqref{length-size-Delta-w} that \begin{equation}
    \label{length-size-diagram}
    \ell(w) = \big|[\Phi_w]\big|.
\end{equation}

Fix $\la \in \Lambda^+$.
Set the following shorthand for the roots in $\Phi(\p^+)$:
\[
\be_{ij} \coloneqq \begin{cases}
    \ep_i - \ep_{p+j},& (\g,\k) \text{ is of Type I }(1\leq i \leq p, 1 \leq j \leq q),\\
    \ep_i + \ep_j, & (\g, \k) \text{ is of Type II }(n \geq i \geq j \geq 1),\\
    \ep_i + \ep_j, & (\g, \k) \text{ is of Type III }(n \geq i > j \geq 1).
\end{cases}
\]
Correspondingly, we define
\begin{equation}
\label{def:d_ij}
    d_{ij} \coloneqq \big\langle\la+\rho, \: f(\be_{ij})^\vee\big\rangle,
\end{equation}
which allows us to rewrite Lemma~\ref{lemma:w-dot-lambda} as 
\begin{equation} 
\label{w-dot-lambda-rewrite} 
    w \cdot \la = \la - \sum_{\mathclap{\be_{ij} \in \Phi_w}} d_{ij} \be_{ij}.
\end{equation}
Let $[\Phi(\p^+)]_\la$ denote the tableau obtained from $[\Phi(\p^+)]$ by filling the box corresponding to the root $\be_{ij}$ with the entry $d_{ij}$:

\input{Filled_diagrams_dij}

\noindent It will sometimes be convenient to index these inner products $d_{ij}$ in terms of the simple roots $\alpha_i$ rather than in terms of the roots $\be_{ij} \in \Phi(\p^+)$.
Hence we define
\begin{equation}
\label{def:d_i}
    d_i \coloneqq \big\langle\la+\rho,\:\al_i^\vee\big\rangle.
\end{equation}
We record the following dictionary (see~\cite{EHP}*{appendix}) between the $d_{ij}$'s and the $d_i$'s:
\begin{equation}
\label{f}
\renewcommand{\arraystretch}{1.5}
    \begin{array}{ll}
    \text{Type I:} & d_{ij} = d_{i+j-1}.  \\
    \text{Type II:} & d_{ij} = d_{n-i+j}.\\[2ex]
    \text{Type III:} & d_{ij} = 
    \begin{cases}
        d_n,& i-j=1 \text{ with }n-i \text{ even},\\
        d_{n-i+j} & \text{otherwise}.
    \end{cases}
    \end{array}
\end{equation}
Using~\eqref{dij} and~\eqref{f}, we obtain the somewhat simpler description of $[\Phi(\p^+)]_\la$ below:

\input{Filled_diagrams_di}

\noindent 
The diagonal dots in~\eqref{di} indicate constant entries along the diagonals of $[\Phi(\p^+)]_\la$ (with the exception of the main diagonal in Type III, where the entries alternate between $d_n$ and $d_{n-1}$).
Let 
\[
    [\Phi_w]_\la \coloneqq \text{ the subtableau of $[\Phi(\p^+)]_\la$ with underlying diagram $[\Phi_w]$}.
\]
Following~\cite{EHP}, we refer to $[\Phi_w]_\la$ as a \emph{diagram of Hermitian type}.
Note that we can further rewrite~\eqref{w-dot-lambda-rewrite} in terms of this diagram:
\begin{equation}
    \label{w-dot as boxes}
    w \cdot \lambda = \la - \sum_{\mathclap{\substack{\text{boxes}\\ B \text{ in } [\Phi_w]_\la}}} (\text{entry in $B$})(\text{root $\beta \in \Phi(\p^+)$ corresponding to $B$)}.
\end{equation}

\subsection{Stacking diagrams}

We now introduce a ``stacking'' construction that converts $[\Phi_w]_\la$ into a new diagram $\stla{w}$ with twice the size.
First, however, we define this stacking operation on an unfilled diagram $[\Phi_w]$, to obtain a new unfilled diagram $\st{w}$ as follows:
\ytableausetup{centertableaux,smalltableaux}
\begin{equation}
\label{stack table}
\begingroup
\setlength{\arraycolsep}{3ex}
\begin{array}{|c|c|c|c|}

\hline
& \phantom{\bigg(}[\Phi_w]\phantom{\bigg)} & \text{Stacking} & \st{w} \\ \hline

\text{Type I}

&

\ydiagram[*(lightgray)]{5,3,2} 

&

\begin{tikzpicture}[baseline]
    \node at (0,0) {\ydiagram[*(lightgray)]{0,0,0,0,0,0,3+5,3+3,3+2,0}*{0,2+1,2+1,1+2,3,3,3+5,3+3,3+2,0}};
    \draw [dashed] (-1,-1) -- ++ (2,2);
    \draw [->,thick] (0.5,0) -- (0,.5);
\end{tikzpicture}

&

\begin{tikzpicture}[baseline]
    \node at (0,0) {\ydiagram[*(lightgray)]{0,0,0,0,0,0,3+5,3+3,3+2,0}*{0,2+1,2+1,1+2,3,3,3+5,3+3,3+2,0}};
\end{tikzpicture}

\\ \hline

   \text{Type II}

    &
    
    \ydiagram[*(lightgray)]{5,1+3,2+2}
    
    &
    
    \begin{ytableau}
      \none\\
      {} & \none & \none[\raisebox{-1pt}{$\longleftrightarrow$}] & \none & *(lightgray){} & *(lightgray){} & *(lightgray){} & *(lightgray){} & *(lightgray){}\\
      {} & {} & \none & \none[\raisebox{-1pt}{$\longleftrightarrow$}] & \none & *(lightgray){} & *(lightgray){} & *(lightgray){}\\
      {} & {} & {} & \none & \none[\raisebox{-1pt}{$\longleftrightarrow$}] & \none & *(lightgray){} & *(lightgray){} \\
      {} & {} & {}\\
      {}\\
      \none
  \end{ytableau}

  &

  \ydiagram{0,6,5,5,3,1,0}*[*(lightgray)]{0,1+5,2+3,3+2,0}

  \\[1ex] \hline

  \text{Type III}

  &

  \ydiagram[*(lightgray)]{5,1+3,2+2}

  &

  \begin{ytableau}
      \none\\
      *(lightgray) & *(lightgray) & *(lightgray) & *(lightgray) & *(lightgray) \\
      \none[\raisebox{-5pt}{$\updownarrow$}] & *(lightgray) & *(lightgray) & *(lightgray) \\
      \none & \none[\raisebox{-5pt}{$\updownarrow$}] & *(lightgray) & *(lightgray)\\
      & \none & \none[\raisebox{-5pt}{$\updownarrow$}]\\
      & & \none\\
      & & \\
      & & \\
      \none
  \end{ytableau}

  &

  \ydiagram[*(white)]{0,1,2,3,3,1}*[*(lightgray)]{5,4,4,3,3,1}

  \\ \hline

\end{array}
\endgroup
\end{equation}
In particular, in Type I, we combine $[\Phi_w]$ with its reflection across the axis passing through its northwest corner at an angle of 45 degrees; the rows of the reflected diagram are viewed as ``negative'' rows of $\st{w}$, extending toward the left rather than toward the right.
In Type II (resp., Type III), we stack $[\Phi_w]$ horizontally (resp., vertically) with its transpose to obtain a true Young diagram $\st{w}$. 
In explicit coordinates giving row lengths, we have the following: 
\begin{equation*}
\renewcommand{\arraystretch}{1.5}
\begin{array}{lll}
 \text{Type I:} & \st{w}_j = -[\Phi_w]'_{q+1-j} & \text{ for } 1 \leq j \leq q, \\
  & \st{w}_{q+i} = [\Phi_w]_i & \text{ for } 1 \leq i \leq p. \\
 \text{Type II:} & \st{w}_i = [\Phi_w]_i + [\Phi_w]'_i & \text{ for } 1 \leq i \leq n. \\
\text{Type III:} & \st{w}'_i = [\Phi_w]_i + [\Phi_w]'_i & \text{ for } 1 \leq i \leq n-1.
\end{array}
\end{equation*}
For example, in Type I in~\eqref{stack table}, the row lengths of $\st{w}$ from top to bottom are 
\[
(0, \ldots, 0, -1,-1,-2,-3,-3;5,3,2,0, \ldots,0),
\]
where we have written the semicolon after the first $q$ coordinates (i.e., the row lengths of the reflected diagram);
ultimately, in light of Theorem~\ref{thm:w-dot-la}, we will work with the ``dual'' tuple 
\[
(0, \ldots, 0, -2,-3,-5;3,3,2,1,1, 0, \ldots, 0),
\]
where as usual the semicolon is written after the first $p$ coordinates. 
In general for Type I, if $[\Phi_w] = \pi \in \Par(p \times q)$, then we have the following description of $\st{w}$ as a $(q+p)$-tuple of its row lengths:
\begin{equation}
    \label{stack-Type I}
    \text{Type I:} \quad \st{w} = (\pi'^*; \pi).
\end{equation}
In Types II and III, if we modify the Frobenius notation to write $[\Phi_w] = (\al|0)$, where the ``0'' records the fact that all leg lengths in a shifted Young diagram are zero, then we can describe the stacking construction in terms of ASC partitions:
\begin{equation}
    \label{stack-ASC}
    \renewcommand{\arraystretch}{1.5}
    \begin{array}{ll}
        \text{Type II:} & \st{w} = (\al+1 \mid \al). \\
        \text{Type III:} & \st{w} = (\al \mid \al+1).
    \end{array}
\end{equation}

In the same way that we constructed the unfilled diagram $\st{w}$ from $[\Phi_w]$, we can construct a filled diagram $\stla{w}$ from $[\Phi_w]_\la$.
Specifically, the shape of $\stla{w}$ is given by $\st{w}$, while the filling is induced by the filling of $[\Phi_w]_\la$.  In Type I, we fill the reflection of $[\Phi_w]_\la$ with the negatives of the original entries, which we denote by a bar (e.g., $\neg{3} = -3$).  Given a filled diagram $D$, we write $\rows D$ (resp. $\cols D$) to denote the tuple whose $i$th coordinate is the sum of the entries in the $i$th row (resp., column) of $D$, counting from top to bottom (resp., left to right).  We then observe the following:
\begin{equation}
\label{rowcol}
\rows\,\stla{w} =
\begin{cases}
    \left(\cols\,[\Phi_w]^*_\la;\: \rows\,[\Phi_w]_\la\right), & (\g,\k) \text{ is of Type I},\\[1ex]
    \rows\,[\Phi_w]_\la+\cols\,[\Phi_w]_\la, & (\g, \k) \text{ is of Type II},\\[1ex]
    (\rows\,[\Phi_w]_\la,\:0) + (0,\:\cols\,[\Phi_w]_\la), & (\g, \k) \text{ is of Type III,}
\end{cases}
\end{equation}
where $(-, 0)$ and $(0, -)$ denote the $n$-tuples obtained by appending and prepending a zero, respectively.

\begin{ex}[Type I]
    Let $p=3$ and $q=4$.  Let $w\in\kW$ such that $\Phi_w = \{\be_{31}, \be_{32}, \be_{33}, \be_{34}, \be_{21}, \be_{22}, \\ \be_{11}\}$.  (Although we will not need to work with $w$ directly, it is easy enough to see that $w = s_{11}s_{22}s_{21}s_{34}s_{33}s_{32}s_{31}$, where $s_{ij} \coloneqq s_{\be_{ij}}$.)
    By consulting~\eqref{dij}, we see that $[\Phi_w] = (4,2,1)$.
    Now fix $\la = (3,3,3;0,0,0,0)$.  Then $\la+\rho = (9, 8, 7; 3, 2, 1, 0)$.  Hence $d_1 = d_2 = d_4 = d_5 = d_6 = 1$, while $d_3 = 4$.  Therefore, following the filling given in~\eqref{di}, we have
    \[
    \ytableausetup{boxsize=1em}
    [\Phi_w]_\la = \ytableaushort[*(lightgray) \scriptstyle]{4111,14,1} \qquad \leadsto \qquad \stla{w} = \ytableaushort[\scriptstyle]{\none\none {\neg{1}},\none\none {\neg{1}}, \none {\neg{4}} {\neg{1}},{\neg{1}} {\neg{1}} {\neg{4}},\none\none\none{*(lightgray)4}{*(lightgray)1}{*(lightgray)1}{*(lightgray)1},\none\none\none{*(lightgray)1}{*(lightgray)4},\none\none\none{*(lightgray)1}}
    \]
    from which we see that $\rows\,\stla{w} = (-1,-1,-5,-6;7,5,1)$.
\end{ex}

\begin{ex}[Type II]
    Let $n=4$.  Let $w \in \kW$ such that $\Phi_w = \{\be_{44}, \be_{43}, \be_{42}, \be_{41}, \be_{33}\}$.  
    By~\eqref{w-in-terms-of-v} we have $w = s_{33}s_{41}s_{42}s_{43}s_{44}$.
    Then $[\Phi_w] = (\al|0)$ where $\al = (3,0)$.  
    Now fix $\la = (9,5,3,3)$.
    Then $\la+\rho = (13, 8, 5, 4)$.
    Hence $d_1 = 5$, $d_2 =3$, $d_3 = 1$, and $d_4 = 4$.  
    Therefore we have
    \[
    \ytableausetup{smalltableaux}
    [\Phi_w]_\la = \ytableaushort[*(lightgray)]{4135,\none4}
    \qquad \leadsto \qquad \stla{w}= \ytableaushort{4{*(lightgray)4}{*(lightgray)1}{*(lightgray)3}{*(lightgray)5},14{*(lightgray)4},3,5}
    \]
    from which we see that $\rows\,\stla{w} = (17, 9, 3, 5)$.
\end{ex}

\begin{ex}[Type III]
    Let $n=4$.
    Let $w \in \kW$ such that $\Phi_w = \{\be_{43}, \be_{42}, \be_{41}, \be_{32}, \be_{31}\}$.  
    By~\eqref{w-in-terms-of-v} we have $w = s_{31}s_{32}s_{41}s_{42}s_{43}$.
    Then $[\Phi_w] = (\al|0)$ where $\al = (2,1)$.  
    Now fix $\la = \left(\frac{3}{2}, \frac{3}{2}, \frac{3}{2}, -\frac{3}{2}\right)$.  
    Then $\la+\rho = (\frac{9}{2}, \frac{7}{2}, \frac{5}{2}, -\frac{3}{2})$.  
    Hence $d_1 = d_2 = d_4 = 1$, while $d_3 =4$.  Therefore we have
    \[
    \ytableausetup{smalltableaux}
    [\Phi_w]_\la = \ytableaushort[*(lightgray)]{111,\none41}
    \qquad \leadsto \qquad \stla{w}= \ytableaushort{{*(lightgray)1}{*(lightgray)1}{*(lightgray)1},1{*(lightgray)4}{*(lightgray)1},14,11}
    \]
    from which we see that $\rows\,\stla{w} = (3, 6, 5, 2)$.
\end{ex}

\subsection{BGG resolutions via the diagrams}

We arrive at the main result of this section:

\begin{theorem}
\label{thm:w-dot-la}
    Suppose $(\g,\k)$ is of Type I, II, or III.  Let $\la \in \Lplusk$ and $w \in \kW$.  Then we have 
    \[
    w \cdot \la = \la + {\rows\,\stla{w}^*}_{\textstyle .}
    \]
\end{theorem}

\begin{proof}
The arguments below follow from the filling of $[\Phi(\p^+)]_\la$ illustrated in~\eqref{dij}:

\bigskip

\textbf{Type I:} We have
\[
\Sigma \coloneqq \sum_{\mathclap{\be_{ij}\in \Phi_w}} d_{ij}\be_{ij} = \sum_{\mathclap{\substack{(i,j):\\ \be_{ij}\in \Phi_w}}} d_{ij}(\ep_i - \ep_{p+j}) = \sum_{i=1}^p \overbrace{\left(\sum_j d_{ij}\right)}^{\mathclap{\text{$i$th row sum from the bottom}}}\ep_i - \sum_{j=1}^q \underbrace{\left(\sum_i d_{ij}\right)}_{\mathclap{\text{$j$th column sum}}}\ep_{p+j}
\]
where the row and column sums refer to $[\Phi_w]_\la$.
Therefore the first $p$ coordinates of $\Sigma$ are those of $\rows\,[\Phi_w]_\la$ in reverse order, while the final $q$ coordinates of $\Sigma$ are those of $-\cols\,[\Phi_w]_\la$.  
Hence by~\eqref{rowcol}, the coordinates of $\Sigma$ are the coordinates of $\rows\,\stla{w}$ in reverse order.
The result follows from~\eqref{w-dot-lambda-rewrite}.

\bigskip

\textbf{Type II:} We have
\[
\Sigma \coloneqq  \sum_{\mathclap{\be_{ij} \in \Phi_w}} d_{ij}\be_{ij} = \sum_{\mathclap{\substack{(i,j):\\ \be_{ij}\in \Phi_w}}} d_{ij}(\ep_i + \ep_j) = \sum_{i=1}^n \overbrace{\left(\sum_j d_{ij}\right)}^{\mathclap{\text{$i$th row sum from the bottom}}} \ep_i + \sum_{j=1}^n \underbrace{\left(\sum_i d_{ij}\right)}_{\mathclap{\text{$j$th column sum from the right}}} \ep_j
\]
where the row and column sums refer to $[\Phi_w]_\la$.
Therefore the coordinates of $\Sigma$ are those of $\rows\,[\Phi_w]_\la+\cols\,[\Phi_w]_\la$ in reverse order.
By~\eqref{rowcol}, these are also the coordinates of $\rows\,\stla{w}$ in reverse order.
The result follows from~\eqref{w-dot-lambda-rewrite}.

\textbf{Type III:} We have
\[
\Sigma \coloneqq \sum_{\mathclap{\be_{ij} \in \Phi_w}} d_{ij}\be_{ij} = \sum_{\mathclap{\substack{(i,j):\\ \be_{ij}\in \Phi_w}}} d_{ij}(\ep_i + \ep_j) = \sum_{i=2}^{n} \overbrace{\left(\sum_j d_{ij}\right)}^{\mathclap{\text{$(i-1)$th row sum from the bottom}}} \ep_i + \sum_{j=1}^{n-1} \underbrace{\left(\sum_i d_{ij}\right)}_{\mathclap{\text{$j$th column sum from the right}}} \ep_j
\]
where the row and column sums refer to $[\Phi_w]_\la$.
Therefore the coordinates of $\Sigma$ are those of
\[
(\rows\,[\Phi_w]_\la,\:0) + (0, \: \cols\,[\Phi_w]_\la)
\]
in reverse order.  
By~\eqref{rowcol}, these are also the coordinates of $\rows\,\stla{w}$ in reverse order.
The result follows from~\eqref{w-dot-lambda-rewrite}.
\end{proof}

\begin{proof}[Proof of Proposition~\ref{prop:BGG}]

We must show that the weights $w \cdot 0$ are precisely those appearing in the resolutions in Proposition~\ref{prop:BGG}. 
When $\la = 0$, we have $d_{i}=\langle\rho,\:\al_i^\vee\rangle=1$ for all $i$.
Hence in this case, each box in $[\Phi_w]_\la$ is filled with a ``1,'' and therefore $\rows\,\st{w}_0 = \st{w}$.

In Type I, we have a bijection $\kW \longrightarrow \Par(p\times q)$ given by $w \longmapsto [\Phi_w]$.  
Setting $\pi = [\Phi_w] \in \Par(p \times q)$, we combine Theorem~\ref{thm:w-dot-la} with~\eqref{stack-Type I} to obtain
\[
w \cdot 0 = \st{w}^* = (\pi'^*;\pi)^* = (\pi^*;\pi').
\]
In Type II, we have a bijection between $\kW$ and strictly decreasing partitions $\al$ such that $n > \al_1 > \cdots > \al_r$ with $r \leq n$; the bijection is given by $w \longmapsto [\Phi_w] = (\al|0)$.
Setting $\pi = \st{w}$, we combine Theorem~\ref{thm:w-dot-la} with~\eqref{stack-ASC} to obtain
\[
w \cdot 0 = \st{w}^* = \pi^* =
    (\al + 1 \mid \al)^*.
\]
The argument is identical in Type III.

It remains to show that each individual term $N_i$ in the resolutions of Proposition~\ref{prop:BGG} contains the correct partitions $\pi$, but in all three types this is immediate from~\eqref{length-size-diagram}.
The value of $s = |\Phi(\p^+)|$ is easily calculated from the diagrams in~\eqref{dij}.
\end{proof}

\begin{ex}[Type I]
    Let $p=q=2$, so that $(\g,\k) = (\ssA_3, \ssA_1 \times \ssA_1)$.
    Let $\la = (6,3;3,1)$. 
    Then $\la+\rho = (9,5;4,1)$, and we have $d_1=4$, $d_2 =1$, and $d_3 = 3$.  
    Below we depict the poset $\kW$, where each element $w$ is labeled with the diagram $[\Phi_w]_\la$:

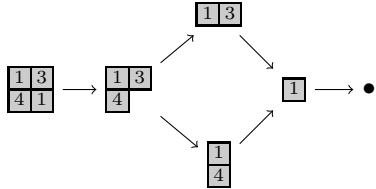
    \input{Example_Res_A.tex}
    
    \noindent Each diagram $[\Phi_w]_\la$ is shorthand for the stacked diagram $\stla{w}$.
    For example, the arrow along the lower-left edge of the diamond represents the following map:
\begin{equation}
\label{map example I}
    \ytableausetup{boxsize=1em}
    \ytableaushort[\scriptstyle]{\none{\neg{3}},{\neg{4}}{\neg{1}},\none\none{*(lightgray)1}{*(lightgray)3},\none\none{*(lightgray)4}} \longrightarrow \ytableaushort[\scriptstyle]{{\neg{4}}{\neg{1}},\none\none{*(lightgray)1},\none\none{*(lightgray)4}}
\end{equation}
    In order to interpret this as a map between parabolic Verma modules, we apply Theorem~\ref{thm:w-dot-la}, which directs us to read off the tuple of row sums for each stacked diagram, and then add its dual to $\la$:
    \begin{alignat*}{3}
        \la + (-3,-5;4,4)^* &=(6,3;3,1) + (-4,-4;5,3) &&= (2,-1;8,4),\\
        \la + (0,-5;1,4)^* &= (6,3;3,1) + (-4, -1;5,0) &&= (2,2;8,1).
    \end{alignat*}
    Therefore the map in~\eqref{map example I} is the map $N_{(2,-1;8,4)} \longrightarrow N_{(2,2;8,1)}$ in the BGG resolution of $L_{\la}$.
\end{ex}

\begin{ex}[Type II]
Let $(\g,\k) = (\ssC_3, \ssA_2)$, and let $\la = (3,1,1)$.
Then $\la+\rho = (6,3,2)$, which gives us $d_1 = 3$, $d_2 = 1,$ and $ d_3 = 2$.
Below is the poset $\kW$, where each element $w$ is labeled with the diagram $[\Phi_w]_\la$:
\input{Example_Res_C.tex}

\noindent As in the previous example, each diagram $[\Phi_w]_\la$ is shorthand for the diagram $\stla{w}$.
For example, the second arrow from the left represents the following map:
\begin{equation}
\label{map example II}
    \ytableaushort{2{*(lightgray)2}{*(lightgray)1}{*(lightgray)3},12{*(lightgray)2}{*(lightgray)1},31} \longrightarrow \ytableaushort{2{*(lightgray)2}{*(lightgray)1}{*(lightgray)3},12{*(lightgray)2},3}
\end{equation}
Applying Theorem~\ref{thm:w-dot-la}, we compute the row sums and add the dual to $\la$:
\begin{alignat*}{3}
    \la + (8,6,4)^* &= (2,0,0) - (4,6,8) &&= (-2, -6, -8),\\
    \la + (8, 5, 3)^* &= (2,0,0) - (3,5,8) &&= (-1, -5, -8).
\end{alignat*}
Therefore the map in~\eqref{map example II} is the map $N_{(8,6,2)^*} \longrightarrow N_{(8,5,1)^*}$ in the BGG resolution of $L_{\la}$.
\end{ex}

\begin{rem}
Suppose we fix an origin on $\stla{w}$, at the common corner (Type I) or the northwest corner (Types II and III).  
If we adopt the convention (as in Type I) that boxes in the right (resp., left) half-plane contain positive (resp., negative) entries, then taking the dual of $\rows\stla{w}$ is equivalent to rotating $\stla{w}$ by 180 degrees about the origin. \end{rem}

\section{Congruence of blocks and conjugate partitions}
\label{sec:Congruence}

\subsection{\!\!\!}
This section is the heart of the paper.  
Our goal is to classify the occurrences of the phenomena we observed in Example~\ref{ex:D4 and C3}.
In particular, we wish to find (pairs of) pairs $(\g,\k)$ and $(\g', \k')$ of Types I--III, along with $\la \in \Lplusk$ and $\la' \in \Lambda^{\!+\!}(\k')$, such that the highest weights $\{\mu = w\cdot \la \mid w \in \kW\}$ and $\{\mu' = w' \cdot \la' \mid w' \in \prescript{\k'\!}{}{\W'}\}$ form two isomorphic posets; 
moreover, this poset isomorphism should preserve BGG resolutions, and also preserve the dimension of each $\k$-module $F_\mu$.
Finally, this isomorphism should send Young diagrams of (duals of) poset elements to their conjugate diagrams.
Our main result, namely Theorem~\ref{thm:Cong and Conj} and its summary in Table~\ref{table:WC}, consists of six infinite families that enjoy the properties observed in the pair from Example~\ref{ex:D4 and C3}.

A word of warning: in this section, we use the prime symbol to denote the image of a weight under a certain reduction operation.
Hence, \emph{a priori} the symbol $\mu'$ now has nothing to do with the conjugate partition of $\mu$.
Nonetheless, it will turn out that in the special settings of Theorem~\ref{thm:Cong and Conj}, the poset elements $\pi$ and $\pi'$ truly are conjugate partitions, as suggested by the notation.

One major advantage of our methods in this paper is that they eliminate the need for explicit calculations in the Weyl group; 
as a result, until now we have not even needed to write down the actions of the reflections $s_{ij} \coloneqq s_{\be_{ij}}$ on weights of $\g$.
In this section, however, it will be useful to record the following.  In Type I, $s_{ij}$ transposes the $i$th and $(p+j)$th coordinates.  In Types II and III, $s_{ij}$ transposes and negates the $i$th and $j$th coordinates.  In Type II, we also have $s_{ii}$, which negates the $i$th coordinate.

\subsection{Congruent blocks in parabolic category $\O$}

Recall the parabolic Verma module $N_\la$ from~\eqref{Verma}.  We define $L_\la$ to be the unique simple quotient of $N_\la$.  We now introduce a category in which these modules $N_\la$ and $L_\la$ are the basic objects.  Let $(\g,\k)$ be a Hermitian symmetric pair.  We denote by $\O(\g,\k)$ the full subcategory of $U(\g)$-mod whose objects $M$ satisfy the following:
\begin{itemize}
    \item $M$ is a finitely generated $U(\g)$-module;
    \item $M$ is a locally finite $U(\q)$-module;
    \item $M$ is a semisimple $U(\k)$-module.
\end{itemize}
The simple modules in $\O(\g,\k)$ are parametrized by $\Lplusk$, via the correspondence $L_\la \longleftrightarrow \la \in \Lplusk$.  
For $\la \in \Lplusk$, let $\chi_\la$ be the infinitesimal character of the ordinary Verma module $M_\la$, and therefore of its quotients $N_\la$ and $L_\la$.  
We let $\O(\g,\k)_\la$ denote the full subcategory of $\O(\g,\k)$ whose objects are the modules whose composition factors have the infinitesimal character $\chi_\la$.  
Given $\la \in \Lplusk$, the category $\O(\g,\k)_\la$ contains finitely many simple modules, namely the modules $L_\mu$ such that $\mu = w \cdot \la \in \Lplusk$ for some $w \in \W$.

We define a \emph{block} in $\O(\g,\k)$ by separating indecomposable modules which are homologically unrelated via the Ext functor.  
We write $\B_\la$ to denote the block containing $L_\la$.  
The block $\B_\la$ is a subcategory of $\O(\g,\k)_\la$, and furthermore, Enright and Shelton~\cites{ES87, ES89} showed that each $\O(\g,\k)_\la$ decomposes into at most two blocks.  
For a block $\B$, we define the \emph{poset of $\B$} to be
\[
\Lambda(\B) \coloneqq \{ \mu \in \Lplusk \mid L_\mu \text{ is an object in $\B$}\},
\]
via the usual ordering on $\h^*$. 
From now on, we reserve the symbol $\la$ to denote a \textit{quasidominant weight}: that is to say, $\la$ is the unique maximal element of $\Lambda(\B_\la)$.  
We write $\mu$ when referring to an arbitrary element of $\Lambda(\B_\la)$.  
We say that $\la \in \Lplusk$ is \emph{regular} if $\langle \la+\rho,\:\al^\vee\rangle \neq 0$ for all $\al \in \Phi^+$.  
If $\la \in \Lplusk$ is regular, then $\O(\g,\k)_\la = \B_\la$ is itself a block, called a \emph{regular block}; in this case, $\Lambda(\B_\la) \cong \kW$ as posets.

In order to capture three of the four important aspects of the situation in Example~\ref{ex:D4 and C3}, we follow~\cite{Armour}*{Def.~3.4.1} in defining the following notion of congruent blocks:

\begin{dfn}
\label{def:congruence}
    Let $(\g,\k)$ and $(\g',\k')$ be Hermitian symmetric pairs.  Let $\B$ be a block in $\O \coloneqq \O(\g,\k)$ and $\B'$ a block in $\O' \coloneqq \O(\g',\k')$.  We say that $\B$ is \emph{congruent} to $\B'$ if
\begin{enumerate}
    \item we have an isomorphism of posets $\Lambda(\B) \cong \Lambda(\B')$, where we write $\mu \mapsto \mu'$;
    \item for all $\mu,\nu \in \Lambda(\B)$ where $\mu < \nu$, and for all $i \geq 0$, we have $\operatorname{Ext}^i_{\O}(N_\mu, L_\nu) \cong \operatorname{Ext}^i_{\O'}(N_{\mu'},L_{\nu'})$;
    \item for all $\mu \in \Lambda(\B)$, we have $\dim F_\mu = \dim F_{\mu'}$.
\end{enumerate}
\end{dfn}

If $\la$ is the maximal element in the poset $\Lambda(\B)$, then properties (1) and (2) above guarantee that $L_\la$ is a Kostant module in the language of~\cite{Enright-Hunziker-RepTh}.  
Therefore $L_\la$ has a BGG resolution \cite{Enright-Hunziker}*{Thm.~2.8}, even though (as in the singular case described in the following subsection) $L_\la$ may not be a finite-dimensional module.
Revisiting Example~\ref{ex:D4 and C3} in light of Definition~\ref{def:congruence}, we can now say that the regular block $\O(\ssD_4,\ssA_3)_0$ is congruent to the regular block $\O(\ssC_3,\ssA_2)_0$.

\subsection{Enright--Shelton reduction}
\label{sub:ES}

Recall that whenever $\la \in \Lplusk$ is regular, we have $\O(\g,\k)_\la = \B_\la \cong \kW$ as posets.  
If, on the other hand, $\langle \la + \rho, \: \al^\vee\rangle = 0$ for some $\al \in \Phi$, then we say that $\la$ is \emph{singular}, and $\al$ is a \emph{singular root} with respect to $\la$.  
We also say that $\B_\la$ is a \emph{singular block}.  
(For all six families in Table~\ref{table:WC}, note that $\la$ is singular.)  
Loosely speaking, the process of \emph{Enright--Shelton reduction} converts a singular weight (plus $\rho$) in $\Lplusk$ into a regular weight (plus $\rho'$) in $\Lambda^{\!+\!}(\k')$, where $(\g',\k')$ is a certain Hermitian symmetric pair whose rank is less than that of $(\g, \k)$. 
 In the following discussion, we will explain the details of this reduction for the specific instances of $\g$ and $\g'$ listed in Table~\ref{table:WC}.  
 We will write a superscript $\flat$ to denote the result of Enright--Shelton reduction:
\begin{equation*}
    \mu + \rho \xrightarrow{\quad{\rm ES}\quad} (\mu + \rho)^\flat = \mu' + \rho',
\end{equation*}
which induces the map
\begin{equation}
\label{ES reduction}
    \mu \longmapsto \mu' = (\mu+\rho)^\flat - \rho'.
\end{equation}
The reduction is invertible, and we will write a superscript $\sharp$ to denote its inverse:
\[
\mu' + \rho' \xrightarrow{\quad{\rm ES}^{-1}\quad} (\mu'+\rho')^\sharp = \mu + \rho,
\]
which induces the inverse map of~\eqref{ES reduction}, namely
\begin{equation}
\label{ES reduction inverse}
    \mu' \longmapsto \mu = (\mu' + \rho')^\sharp - \rho.
\end{equation}

\begin{theorem}[\cites{ES87,ES89}]
\label{thm:ES}
    Let $(\g, \k)$ be a Hermitian symmetric pair.  
    Let $\la \in \Lplusk$ with $L_\la \neq N_\la$, and let $\la' \in \Lambda^{\!+\!}(\k')$ be the regular weight obtained from $\la$ by Enright--Shelton reduction.  
    Then the map~\eqref{ES reduction} induced by Enright--Shelton reduction restricts to an isomorphism of posets
    \[
        \Lambda(\B_\la) \longrightarrow \Lambda(\B_{\la'})
    \]
    satisfying conditions (1) and (2) of Definition~\ref{def:congruence}.
\end{theorem}

\begin{table}[t]
\centering
\input{Table_main.tex}
\caption{Congruence of singular and regular blocks $\B_\la$ and $\B_{\la'}$.  (See Theorem~\ref{thm:Cong and Conj}.)  
On the right side of the table, we put $p=P-m$, $q = Q-m$, and $n = N-m$.  
To avoid trivial cases, we let $1 \leq k < \text{rank of $(\g,\k)$}$, although Type II remains of interest for $k=0$.  
Note that $\omega^*_k$ is the $N$-tuple whose last $k$ coordinates are $-1$, with $0$'s elsewhere.  We write $\omega'_i$ for the fundamental weights of $\g'$.}
\label{table:WC}
\end{table}

We now detail the process of Enright--Shelton reduction for each of the six families in Table~\ref{table:WC}.  
Note that we label each family with respect to $\g'$ rather than $\g$, reflecting our philosophy that the regular blocks are easier to understand than the singular blocks.  
We use capital letters $P$, $Q$, and $N$ to describe the rank of $\g$; then defining $m$ to be the number of coordinates deleted via reduction (details below), we describe the rank of $\g'$ using the lower-case letters $p\coloneqq P-m$, $q\coloneqq Q-m$, and $n \coloneqq N-m$.  
In each family, the parameter $k$ ranges over all positive integers such that
\[
k < {\rm rank} (\g,\k) = \begin{cases}
    \min\{P,Q\}, & \g = \ssA_{P+Q-1},\\
    N, & \g = \ssC_N,\\
    \lfloor N/2 \rfloor, & \g = \ssD_N.\\
\end{cases}
\]

\textbf{Type I} ($\g = \ssA_{P+Q-1}$, $\g' = \ssA_{p+q-1}$; $\la = -k\omega_P$, $\la' = k\omega'_p$).  
Suppose $\mu+\rho = (a_1, \ldots, a_P; b_1, \ldots, b_Q)$.  
Then $\mu +\rho$ contains a singularity for each instance of an equality $a_i = b_j$, in which case a singular root is $\be_{ij} = \ep_i - \ep_{P+j}$.  
Hence to perform the reduction on $\mu + \rho$, we delete all coordinate pairs $a_i, b_j$ such that $a_i = b_j$.  
Then $m$ is the number of such pairs.

 For example, let $P=4$ and $Q=3$, with $k=2$.  
 Then $\la = (-2,-2,-2,-2;0,0,0)$, and $\la + \rho = (4,3,2,1;2,1,0)$. 
 Therefore by~\eqref{ES reduction}, we have
 \[
    \la' = (\la+\rho)^\flat - \rho' = (4,3,\mathbf{2},\mathbf{1};\mathbf{2},\mathbf{1},0)^\flat - \rho' = (4,3;0) - (2,1;0) = (2,2;0) = 2\omega'_2.
 \]
Note that the reduction deletes $m=2$ pairs of coordinates, so that $p=2$, $q=1$, and indeed $\la' = k\omega'_p$.  
As an example of an arbitrary element of $\Lambda(\B_\la)$, we choose $\mu = (-2,-3,-3,-3;1,1,1)$.  
Then $\mu + \rho = (4,2,1,0;3,2,1)$, and we have
\[
    \mu' = (4,\mathbf{2},\mathbf{1},0;3,\mathbf{2},\mathbf{1})^\flat - \rho' = (4,0;3) - (2,1;0) = (2, -1; 3).
\]

\textbf{Type II} ($\g = \ssD_N$, $\g' = \ssC_n$; $\la = -2k\omega_N$, $\la = k\omega'_n$).  
This case is exceptional among the six families in Table~\ref{table:WC}, because the $\g'$ obtained by Enright--Shelton reduction is actually $\ssD_{n+1}$, not $\ssC_n$.  
In order to obtain the result in Table~\ref{table:WC}, we compose Enright--Shelton reduction with a further reduction from $\ssD_{n+1}$ to $\ssC_n$, by deleting the 0 from $(\mu + \rho)^\flat$.  
This second reduction also satisfies conditions (1) and (2) in Definition~\ref{def:congruence}.  
The reason for our modification here is this: for the other five families, it turns out that Enright--Shelton reduction satisfies condition (3) as well, producing congruent blocks, but in Type II, the extra reduction is necessary to fulfill condition (3).

Suppose $\mu + \rho = (a_1, \ldots, a_N)$, where the $a_i$ are integers. 
Then $\mu+\rho$ contains a singularity for each instance of an equality $a_i = -a_j < 0$, in which case a singular root is $\be_{ij} = \ep_i + \ep_j$.  
To perform the reduction on $\mu + \rho$, we delete all coordinate pairs $a_i, a_j$ such that $a_i = -a_j$.  
Then as mentioned above, we perform a second reduction by deleting the coordinate 0.

For example, let $N=8$, with $k=2$.  
Then $\la = (-2,\ldots,-2)$, and $\la+\rho=(5,4,\ldots,-1,-2)$.  
Therefore, we have
\[
\la+\rho \xrightarrow{\quad{\rm ES}\quad}   (\la+\rho)^\flat = (5,4,3,\mathbf{2},\mathbf{1},0,\mathbf{-1},\mathbf{-2})^\flat =  (5,4,3,0) \xrightarrow{\text{delete $0$}} (5,4,3)
\]
as a weight in Type $\ssC_3$.  Then, treating $(5,4,3)$ as the term $(\la+\rho)^\flat$ in~\eqref{ES reduction}, we have
\[
\la' = (\la+\rho)^\flat - \rho' = (5,4,3)-(3,2,1) = (2,2,2) = 2\omega_3.
\]
Ultimately we have deleted $m=5$ coordinates.  As an example of an arbitrary element of $\Lambda(\B_\la)$, we take $\mu = (-2,-4,-4,-4,-4,-4,-4,-4)$.
Then $\mu + \rho = (5,2,1,0,-1,-2,-3,-4)$, and we have
\[
(5,\mathbf{2}, \mathbf{1},0,\mathbf{-1},\mathbf{-2},-3,-4)^\flat = (5,0,-3,-4) \xrightarrow{\text{delete $0$}} (5,-3,-4),
\]
and so
\[
\mu'= (5,-3,-4) - (3,2,1) = (2,-5,-5).
\]
(We will revisit this setting in Example~\ref{ex:D8 and C3}.)

\textbf{Type IIIa} ($\g = \ssC_N$, $\g' = \ssD_n$; $\la = -\frac{k}{2}\omega_N$, $\la'=k\omega'_n$).  Suppose $\mu + \rho = (a_1, \ldots, a_N)$ with the $a_i$ either all integers or all half-integers.  
Then $\mu + \rho$ contains a singularity for each instance of an equality $a_i = -a_j < 0$, in which case a singular root is $\be_{ij}$; moreover, if $k$ is even, then $\mu + \rho$ contains another singularity at the coordinate $a_i=0$, which means that we have a (long) singular root $\be_{ii} = 2\ep_i$.  
To perform the reduction on $\mu + \rho$, we delete all pairs of opposite coordinates, along with 0 (if applicable).

For example, let $N=7$ and $k=4$.  
Then $\la = (-2,\ldots,-2)$, and $\la + \rho = (5,4,3,2,1,0,-1)$.  
Then we reduce by deleting $m=3$ coordinates, and we have
\[
\la' = (\la+\rho)^\flat-\rho' = (5,4,3,2,\mathbf{1},\mathbf{0},\mathbf{-1})^\flat - \rho' = (5,4,3,2)-(3,2,1,0) = (2,2,2,2) = 4\omega'_4.
\]
On the other hand, if $k=5$, then $\la = \left(-\frac{5}{2}, \ldots, -\frac{5}{2}\right)$ and $\la+\rho = \left(\frac{9}{2}, \frac{7}{2}, \ldots, -\frac{1}{2}, -\frac{3}{2}\right)$.  
We now delete $m=4$ coordinates, and obtain
\[
\la' = \left(\frac{9}{2}, \frac{7}{2}, \frac{5}{2}, \mathbf{\frac{3}{2}}, \mathbf{\frac{1}{2}}, \mathbf{-\frac{1}{2}}, \mathbf{-\frac{3}{2}}\right)^{\!\flat} - \rho' = \left(\frac{9}{2}, \frac{7}{2}, \frac{5}{2}\right) - (2,1,0) = \left(\frac{5}{2}, \frac{5}{2}, \frac{5}{2}\right) = 5 \omega'_3.
\]

\textbf{Type IIIb} ($\g = \ssC_N$, $\g' = \ssD_n$; $\la = -\frac{k}{2}\omega_N + \omega^*_k$, $\la'=k\omega'_{n-1}$).  
The reduction procedure is the same as in Type IIIa. 
As an example, let $N=7$ and $k=4$.  Then $\la = (-2,-2,-2,-3,-3,-3,-3)$, and $\la + \rho = (5,4,3,1,0,-1,-2)$.  We have
\[
\la' = (\la+\rho)^\flat-\rho' = (5,4,3,\mathbf{1},\mathbf{0},\mathbf{-1},-2)^\flat - \rho' = (5,4,3,-2)-(3,2,1,0) = (2,2,2,-2) = 4\omega'_3.
\]
When $k$ is odd, the half-integral case works out similarly.

\textbf{Type IIIc} ($\g = \ssD_N$, $\g' = \ssD_n$; $\la = -(2k-1)\omega_N$, $\la'=(2k+1)\omega'_n$).  
The reduction procedure is the same as in Type II, and there is no need for the extra reduction step since in this case we always obtain half-integer coordinates. 
As an example, let $N=4$ and $k=1$.  
Then $\la = \left(-\frac{1}{2}, \ldots, -\frac{1}{2}\right)$ and $\la + \rho = \left(\frac{5}{2}, \frac{3}{2}, \frac{1}{2}, -\frac{1}{2}\right)$.  
We reduce by deleting $m=2$ coordinates, and we have
\[
\la' = \left(\frac{5}{2}, \frac{3}{2}, \mathbf{\frac{1}{2}}, \mathbf{-\frac{1}{2}}\right)^{\!\flat}-\rho' = 
\left(\frac{5}{2}, \frac{3}{2}\right) - (1,0) = \left(\frac{3}{2}, \frac{3}{2}\right) = 3\omega'_2.
\]

\textbf{Type IIId} ($\g = \ssD_N$, $\g' = \ssD_n$; $\la = -(2k-1)\omega_N + \omega^*_{2k+1}$, $\la'=(2k+1)\omega'_{n-1}$).  
The reduction procedure is the same as in Type IIIc. As an example, let $N=6$ and $k=2$.  
Then $\la = \left(-\frac{3}{2} -\frac{5}{2},-\frac{5}{2},-\frac{5}{2},-\frac{5}{2},-\frac{5}{2}\right)$ and $\la + \rho = \left(\frac{7}{2}, \frac{3}{2}, \frac{1}{2}, -\frac{1}{2}, -\frac{3}{2}, -\frac{5}{2}\right)$.  
We reduce by deleting $m=4$ coordinates, and we have
\[
\la' = \left(\frac{7}{2}, \mathbf{\frac{3}{2}}, \mathbf{\frac{1}{2}}, \mathbf{-\frac{1}{2}}, \mathbf{-\frac{3}{2}}, -\frac{5}{2}\right)^{\!\flat}-\rho' = 
\left(\frac{7}{2}, -\frac{5}{2}\right) - (1,0) = \left(\frac{5}{2}, \frac{5}{2}\right) = 5\omega'_1.
\]

\subsection{Twisted posets of the regular blocks}

In order to see how conjugate partitions are related to congruent blocks, we introduce a ``twist'' to the posets $\Lambda(\B_\la)$ and $\Lambda(\B_{\la'})$.
Let $\zeta$ be the unique fundamental weight of $\g$ that is orthogonal to $\Phi(\k)$; likewise, let $\zeta'$ be the unique fundamental weight of $\g'$ orthogonal to $\Phi(\k')$.  
Let $\be$ be the highest root of $\g$, and $\be'$ the highest root of $\g'$. 
We note that $\langle \zeta,\be^\vee\rangle = \langle \zeta',\be'^\vee\rangle  
= 1$.  

\begin{dfn}
\label{def:Lambda tilde}
    Let $\la$ and $\la'$ belong to one of the families in Table~\ref{table:WC}.  
    We define the \emph{twisted posets}
    \begin{align*}
        \tL(\B_\la) &\coloneqq \{ \mu - \langle \la, \be^\vee\rangle \zeta \mid \mu \in \Lambda(\B_\la)\},\\
        \tL(\B_{\la'}) & \coloneqq \{ \mu' - \langle \la',\be'^\vee\rangle \zeta' \mid \mu' \in \Lambda(\B_{\la'})\}.
    \end{align*}
\end{dfn}

Note that, except in Types IIIb and IIId, the weight $\langle \la,\be^\vee\rangle\zeta$ is the same as $\la$, and $\langle \la',\be'^\vee\rangle\zeta'$ is the same as $\la'$. 
Hence outside Types IIIb and IIId, the twisted posets have the weight 0 as their maximal element.  
For Types IIIb and IIId, $\langle \la,\be^\vee\rangle \zeta$ is just the $\la$ from Types IIIa and IIIc, respectively; the same is true for $\langle \la',\be'^\vee\rangle\zeta'$ and the $\la'$.   
Clearly $\Lambda(\B_\la) \cong \tL(\B_\la)$ and $\Lambda(\B_{\la'}) \cong \tL(\B_{\la'})$ as posets, and any poset map $\Lambda(\B_\la)\longrightarrow \Lambda(\B_{\la'})$ induces a unique map $\tL(\B_\la)\longrightarrow \tL(\B_{\la'})$ between twisted posets.  
In the context of BGG resolutions, the twist amounts to tensoring with the $1$-dimensional $\k$-module $F_{-\langle \la, \be^\vee\rangle \zeta}$ or the $1$-dimensional $\k'$-module $F_{-\langle\la',\be'^\vee\rangle\zeta'}$. 
Therefore, we will write
\begin{align}
\label{def:L tilde}
\begin{split}
\widetilde{L}_\la &\coloneqq L_\la \otimes F_{-\langle \la,\be^\vee\rangle\zeta,}\\
\widetilde{L}_{\la'} &\coloneqq L_{\la'} \otimes F_{-\langle \la',\be'^\vee\rangle\zeta'.}\end{split}
\end{align}

In what follows, we will always write $\Phi_w$ and $\prescript{\k'\!}{}{\W}$ in reference to $(\g', \k')$, rather than $(\g,\k)$.
In Lemma~\ref{lemma:pi'} below, we establish the final column of Table~\ref{table:WC}, which is an explicit description of the elements of $\tL(\B_{\la'})$.  
The flow of our argument in each type is this: we use Enright--Shelton reduction on $\la$ to obtain $\la'$, which determines the filling of diagrams $[\Phi_w]_{\la'}$.  
From this filling we show that $\rows\,\st{w}_{\la'}$ is the result of uniformly lengthening the arms or legs of $\st{w}$, whose shape we already understand from~\eqref{stack-Type I} and~\eqref{stack-ASC}.

\begin{lemma}
    \label{lemma:pi'}
    Let $\la$ and $\la'$ belong to one of the six families in Table~\ref{table:WC}.  Then $\tL(\B_{\la'})$ is the set of all weights $\pi'^*$ shown in the last column of Table~\ref{table:WC}.
\end{lemma}

\begin{proof}

\textbf{Type I} ($\la = -k\omega_P$).
We have $\la + \rho=(P+Q-1-k, \ldots, Q-k; \: Q-1, \ldots, 0)$, where the ellipses denote coordinates decreasing by 1.  
Thus the $(P+1-k)$th coordinate is $Q-1$, meaning that the string of $k$ coordinates before the semicolon equals the string of $k$ coordinates after the semicolon.  
Hence in the Enright--Shelton reduction, we delete these $m=k$ pairs of coordinates, so that $\g' = \ssA_{P+Q-2k-1} = \ssA_{p+q-1}$ and $\rho' = (p+q-1, \ldots, 0)$. 
Thus
    \begin{align*}
       (\la+\rho)^\flat = \la' + \rho' &= (p+q+k-1, \ldots, q+k;q-1, \ldots, 0),\\
       \la' &= (k,\ldots,k;0,\ldots,0) = k\omega'_p.
    \end{align*}
    It is clear that $d_i = 1$ for all $i \neq p$, while $d_p=k+1$.  
    Hence every non-diagonal entry in $[\Phi_w]_{\la'}$ is 1, while every diagonal entry is $k+1$. 
    Therefore $\cols\,[\Phi_w]_{\la'}$ is the result of adding $k$ to each leg of $[\Phi_w]$ and then taking the conjugate; likewise, $\rows\,[\Phi_w]_{\la'}$ is the result of adding $k$ to each arm of $[\Phi_w]$.  
    Therefore, setting $[\Phi_w] = (\al|\be)$, and recalling that $m=k$, we use~\eqref{rowcol} to obtain 
    \[
        \rows\,\st{w}_{\la'} = ((\be+m\mid\al)^*;\:(\al+m \mid \be)).
    \]
    By Theorem~\ref{thm:w-dot-la} and Definition~\eqref{def:Lambda tilde}, each element of $\tL(\B_{\la'})$ equals $\rows\,\st{w}_{\la'}^*$ for a unique $w \in \prescript{\k'\!}{}{\W}$.  
    We therefore have 
\[
    \tL(\B_{\la'}) = \Big\{(\al+m\mid\be)^* \otimes (\be+m \mid \al) \:\Big|\: (\al|\be) \in \Par(p \times q)\Big\},
\]
    as desired.
    
    \textbf{Type II} ($\la = -2k\omega_N$).  
    We have $\la+\rho=(N-k-1, \ldots, -k)$, where the ellipsis denotes coordinates decreasing by 1.  
    Thus the final $2k+1$ coordinates are the string $(k, \ldots, -k)$, all of which (except the 0) are deleted via Enright--Shelton reduction.  
    We then have $(\la+\rho)^\flat = (N-k-1, \ldots, k+1, 0)$. 
    After deleting the $0$ (as described in Section~\ref{sub:ES}), we have deleted $m=2k+1$ coordinates, so that $\g' = \ssC_{N-2k-1} = \ssC_n$ and $\rho' = (n,\ldots,1)$.  
    Thus
    \begin{align*}
        \la'+\rho' &= (k+n,\ldots,k+1),\\
        \la' &= (k,\ldots,k) = k\omega'_n.
    \end{align*}
    It is clear that $d_i = 1$ for all $i\neq n$, while $d_n = k+1$. 
    Hence every non-diagonal entry in $[\Phi_w]_{\la'}$ is $1$, while every diagonal entry is $k+1$.  
    Therefore $\rows\,\st{w}_{\la'}$ is the result of adding $2k$ to each arm of $\st{w}$.  
    But if $[\Phi_w] = (\al|0)$, then by~\eqref{stack-ASC} we have $\st{w} = (\al+1\mid \al)$, and so $\rows\,\st{w}_{\la'} = (\al+2k+1 \mid \al)= (\al+m \mid \al)$.  
    By Theorem~\ref{thm:w-dot-la} and Definition~\eqref{def:Lambda tilde}, each element of $\tL(\B_{\la'})$ equals $\rows\,\st{w}_{\la'}^*$ for a unique $w \in \prescript{\k'\!}{}{\W}$, and hence $\tL(\B_{\la'}) = \{(\al+m \mid \al)^* \mid \al_1 < n\}$, as desired. 

    \textbf{Type IIIa} ($\la = -\frac{k}{2}\omega_N$).  
    We have $\la+\rho=(N-\frac{k}{2}, \ldots, 1-\frac{k}{2})$, where the ellipsis denotes coordinates decreasing by 1.  
    Thus the $(k-1)$th coordinate from the end equals $k-1-\frac{k}{2} = \frac{k}{2}-1$, and so the final $k-1$ coordinates are the string $(\frac{k}{2}-1, \ldots, 1-\frac{k}{2})$, which is deleted via Enright--Shelton reduction.  
    Hence $m=k-1$, so that $\g' = \ssD_{N-k+1} = \ssD_n$ and $\rho' = (n-1,\ldots,0)$.  
    Thus
    \begin{align*}
        (\la+\rho)^\flat = \la'+\rho' &= \left(n-1+\frac{k}{2},\ldots,\frac{k}{2} \right),\\
        \la' &= \left(\frac{k}{2},\ldots,\frac{k}{2}\right) = k\omega'_n.
    \end{align*}
    We therefore have $d_i = 1$ for all $i\neq n$, while $d_n = k+1$. 
    Hence all entries in $[\Phi_w]_{\la'}$ are $1$, except for the odd diagonal entries, which are $k+1$.  
    Therefore in $\st{w}_{\la'}$, the entries $k+1$ occur in consecutive vertical pairs: explicitly, in positions $(2i-1,\:2i-1)$ and $(2i,\: 2i-1)$, for $i = 1, \ldots, h \coloneqq \lceil \rk[\Phi_w]/2\rceil$. 
    It follows that $\rows\,\st{w}_{\la'}$ is the result of adding $k$ to the first $h$ row pairs in $\st{w}$, which forces $\rk \rows\,\st{w}_{\la'} = 2h$.  
    If $[\Phi_w] = (\al|0)$, then by~\eqref{stack-ASC} we have $\st{w} = (\al\mid \al+1)$, and so $\rows\,\st{w}_{\la'} = (\al+k \mid \al+1)$, where if $\rk \al$ is odd then we augment $\al$ by inserting $(\ldots, -1 \mid \ldots, 0)$.  
    By allowing $\al_1<n$ rather than $\al_1 < n-1$, we can rewrite partitions of this form as $\rows\,\st{w}_{\la'} = (\al+k-1 \mid \al) = (\al + m \mid \al)$ with even rank.  
    The rest follows as in the previous cases.

    \textbf{Type IIIb} ($\la = -\frac{k}{2}\omega_N+\omega_k^*$).
    We have 
    \[
        \la + \rho = \Big(\underbrace{N - \tfrac{k}{2}, \ldots, \tfrac{k}{2}+1}_{N-k},\underbrace{\tfrac
        {k}{2}-1, \ldots, -\tfrac{k}{2}}_k\Big),
    \]
    where the ellipses denote coordinates decreasing by 1.  
    Thus the final $k$ coordinates, except for the very last one, are all deleted via Enright--Shelton reduction.  
    Hence $m=k-1$, so that $\g' = \ssD_{N-k+1} = \ssD_n$ and $\rho' = (n-1, \ldots, 0)$.  Thus
    \begin{align*}
        (\la + \rho)^\flat = \la' + \rho' &= \left(n-1+\frac{k}{2}, \ldots, \frac{k}{2}+1,-\frac{k}{2}\right),\\
        \la' &= \left(\frac{k}{2},\ldots, \frac{k}{2},-\frac{k}{2}\right) = k\omega'_{n-1}.
    \end{align*}
    We therefore have $d_i = 1$ for all $i \neq n-1$, while $d_{n-1} = k+1$.  
    Hence $[\Phi_w]_{\la'}$ is the same as in Type IIIa, except that $k+1$ occurs as the \emph{even} diagonal entries.  
    As a result, $\rows\,\st{w}_{\la'}$ has \emph{odd} rank, but is not in general a true partition since its first arm is not longer than its second arm: indeed, it is obtained from $\st{w}$ by adding $k$ to the first $\lfloor \rk[\Phi_w]/2 \rfloor$ row pairs \textit{beneath} the first row. 
    Recall, however, that in Type IIIb, we define $\tL(\B_{\la'})$ in terms of the $\la'$ from Type IIIa, namely $k\omega'_n$; since the difference of the two $\la'$s is $k\ep_1^*$, we have
    \[
        \tL(\B_{\la'}) = \left\{(k\ep_1 + \rows\,\st{w}_{\la'})^* \: \middle| \: w \in \prescript{\k'\!}{}{\W}\right\}.
    \]
    This addition of $k\ep_1$ adds back the ``missing'' $k$ to the first arm of $\rows\,\st{w}_{\la'}$, so that elements of $\tL(\B_{\la'})$ are the duals of the partitions $(\al + m \mid \al)$ with odd rank, where $\al_1 < n$.

    \textbf{Type IIIc} ($\la = -(2k-1)\omega_N$).  
    We have $\la+\rho=(N-k-\frac{1}{2},\ldots,-k+\frac{1}{2})$, where the ellipsis denotes coordinates decreasing by 1.  
    Thus the final $2k$ coordinates are the string $(k-\frac{1}{2}, \ldots, -k+\frac{1}{2})$, which is deleted via Enright--Shelton reduction.  
    Hence $m=2k$, so that $\g' = \ssD_{N-2k} = \ssD_n$ and $\rho' = (n-1,\ldots,0)$.  
    Thus
    \begin{align*}
        (\la+\rho)^\flat = \la'+\rho' &= \left(n+k-\frac{1}{2}, \ldots, k+\frac{1}{2}\right),\\
        \la' &= \left(\frac{2k+1}{2},\ldots,\frac{2k+1}{2}\right) = (2k+1)\omega'_n.
    \end{align*}
    The rest of the argument is identical to Type IIIa, where $m=2k$ instead of $m=k-1$. 
    
    \textbf{Type IIId} ($\la = -(2k-1)\omega_N + \omega_{2k+1}^*$).
    We have 
    \[
        \la + \rho = \Big(\underbrace{N - k - \tfrac{1}{2}, \ldots, k+\tfrac{3}{2}}_{N-2k-1},\underbrace{k-\tfrac
        {1}{2}, \ldots, -k-\tfrac{1}{2}}_{2k+1}\Big),
    \]
    where the ellipses denote coordinates decreasing by 1.  
    Thus the final $2k+1$ coordinates, except for the very last one, are all deleted via Enright--Shelton reduction.  
    Hence $m=2k$, so that $\g' = \ssD_{N-2k} = \ssD_n$ and $\rho' = (n-1, \ldots, 0)$.  
    Thus
    \begin{align*}
        (\la + \rho)^\flat = \la' + \rho' &= \left(n+k-\frac{1}{2}, \ldots, k+\frac{3}{2},-k-\frac{1}{2}\right),\\
        \la' &= \left(\frac{2k+1}{2},\ldots, \frac{2k+1}{2},-\frac{2k+1}{2}\right) = (2k+1)\omega'_{n-1}.
    \end{align*}
    The rest of the argument is identical to Type IIIb, where $m=2k$ instead of $m=k-1$.
\end{proof}

\subsection{Main result: congruent blocks and conjugate partitions}
\label{sub:proofs}

We arrive at our main result:

\begin{theorem}
\label{thm:Cong and Conj}
    Let $\la$ and $\la'$ belong to one of the six families in Table~\ref{table:WC}.  
    Then $\B_\la$ is congruent to $\B_{\la'}$.  Moreover, we have an isomorphism of twisted posets
    \begin{align*}
        \tL(\B_\la) &\longrightarrow \tL(\B_{\la'}),\\
        \pi^* &\longmapsto \pi'^*
\end{align*}
such that $\pi$ and $\pi'$  are conjugate partitions.
\end{theorem}

\begin{rem}
\label{rem:butterfly}
In Type I, the weight $\pi' = \rows\,\st{w}_{\la'}$ is not a true partition, but rather a pair of partitions, where the dual operation $( \: )^*$ has been applied to the first partition. 
 Hence $\pi'$ is represented by a corner-to-corner stacking of the two Young diagrams, where the first is rotated 180 degrees so that its row lengths are considered negative.  
 (Recall the stacked diagrams $\st{w}$ in Type I, which took the same form.)  
 Seeing as how such stacked diagrams resemble a butterfly, we will refer to each of the two diagrams as a ``wing'' of $\pi'$. 
 Likewise, $\pi$ is represented by a butterfly diagram.  
 The claim in the theorem is that $\pi'$ and $\pi$ are conjugates, i.e., we can obtain one from the other by reflecting about the 45-degree axis through the center of the butterfly.
\end{rem}

From now on, thanks to Lemma~\ref{lemma:pi'}, we will write $\pi'^*$ for an arbitrary element of $\tL(\B_{\la'})$, where $\pi'$ is a partition (except in Type I, where $\pi$ contains $q$ negative coordinates followed by $p$ positive coordinates). 
By Theorem~\ref{thm:ES}, there is a poset isomorphism $\Lambda(\B_\la) \longrightarrow \Lambda(\B_{\la'})$ induced by Enright--Shelton reduction, which further induces a poset isomorphism $\tL(\B_\la) \longrightarrow \tL(\B_{\la'})$.
We denote the preimage of $\pi'^*$ by writing $\pi^* \in \tL(\B_{\la})$, without making any assumptions about the nature of $\pi = (\pi^*)^*$ itself.  
Explicitly, we must have
\[
\pi = [(\pi'^*+\la'+\rho')^\sharp - (\la+\rho)]^*.
\]
(As before, in Types IIIb and IIId we must use the $\la$ and $\la'$ from Types IIIa and IIIc, respectively.)  
Since $\pi'^* = w\cdot\la'-\la'$ for some $w \in \prescript{\k'\!}{}{\W}$, the claim in Theorem~\ref{thm:Cong and Conj} is that the two weights
\begin{align}
\label{pi formula}
\begin{split}
    \pi'&=[w(\la'+\rho')-(\la'+\rho')]^*,\\
    \pi &= [w(\la'+\rho')^\sharp - (\la+\rho)]^*
    \end{split}
\end{align}
are truly conjugate partitions for all $w \in \prescript{\k'\!}{}{\W}$.  
Before proving Theorem~\ref{thm:Cong and Conj}, we present a detailed example that illuminates the way in which these conjugate partitions arise.  
Our approach is to begin with $\Lambda(\B_{\la'})$, which we understand completely thanks to Lemma~\ref{lemma:w-dot-lambda}, and then reverse the Enright--Shelton reduction to pass to $\Lambda(\B_{\la})$.  
This philosophy --- taking the regular block as our starting point in order to understand the singular block --- is the reason for our labeling the various types in terms of $\g'$ rather than $\g$.

\begin{figure}[t]
    \centering
    \input{Example_Conjugates_C.tex}
    \caption{Illustration of Example~\ref{ex:D8 and C3}.  
    The symbol $\bbox$ denotes the string $(2,1,0,\neg{1}, \neg{2})$ deleted via Enright--Shelton reduction (and subsequent deletion of 0). 
    For typographical clarity, the bars denote negatives.}
    \label{fig:example conjugates Type II}
\end{figure}
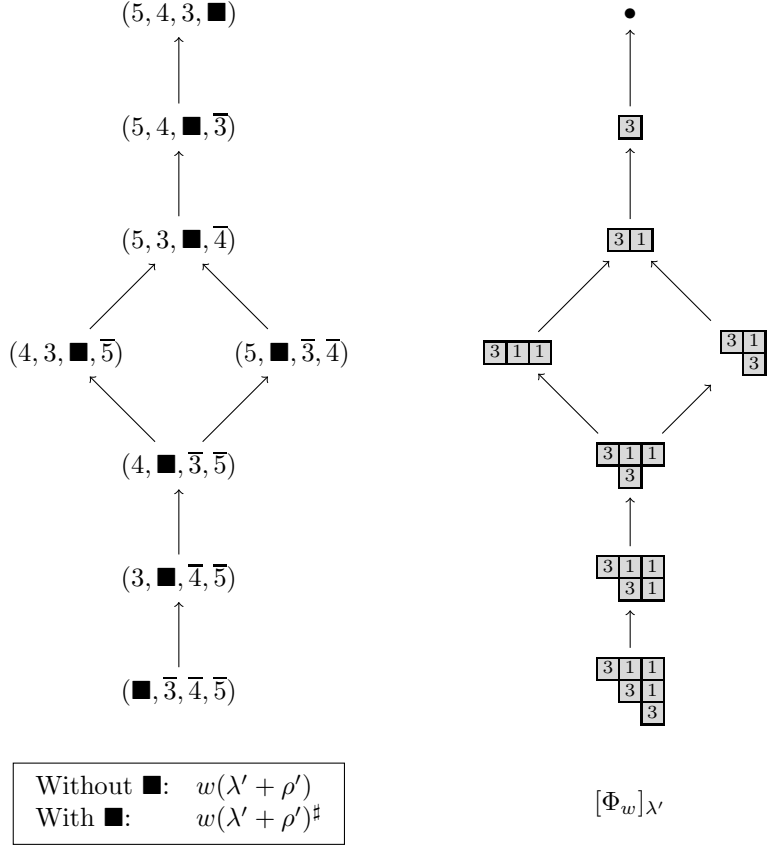

\begin{ex}
\label{ex:D8 and C3}
    We revisit our previous example in Type II, where $N=8$ and $k=2$.  
    (See the case-by-case descriptions at the end of Section~\ref{sub:ES}.)
    For typographical clarity, we write negatives as bars over the coordinates; hence we have $\g = \ssD_8$ with $\la = (\neg{2}, \ldots, \neg{2})$, and $\g' = \ssC_3$ with $\la' = (2,2,2)$.  Thus $\la'+\rho' = (5,4,3)$.
    
We refer the reader to Figure~\ref{fig:example conjugates Type II} throughout this example. 
On the right, we represent $w \in \prescript{\k'\!}{}{\W}$ by the diagram $[\Phi_w]_{\la'}$.  
On the left, we simultaneously depict \emph{both} posets $\tL(\B_{\la'})$ and $\tL(\B_{\la})$.  
The symbol $\bbox \coloneqq (2,1,0,\neg{1}, \neg{2})$ denotes the string of $m=5$ coordinates deleted via Enright--Shelton reduction (and subsequent deletion of the 0).  
In this way, ignoring the $\bbox$ gives us $w(\la'+\rho')$, while retaining the $\bbox$ gives us $w(\la'+\rho')^\sharp$.  (We abuse notation slightly by writing $\sharp$ to reverse both the 0-deletion and the Enright--Shelton reduction.) 
Upon subtracting either $\la'+\rho'$ or $\la+\rho$, we have the elements of either $\tL(\B_{\la'})$ or $\tL(\B_{\la})$.

Our goal is to understand why (the duals of) corresponding elements $\pi'^* \in \tL(\B_{\la'})$ and $\pi^* \in \tL(\B_{\la})$ are conjugate partitions.  
In order to compare inductively the construction of $\pi'$ and $\pi$, it suffices to consider the general case where $w$ covers $v$, i.e., where $[\Phi_v]_{\la'}$ is joined from above to $[\Phi_w]_{\la'}$ in Figure~\ref{fig:example conjugates Type II}.  
In other words, we consider the effect of adding one box to $[\Phi_v]_{\la'}$ to obtain $[\Phi_w]_{\la'}$.  
In the base case, at the top of Figure~\ref{fig:example conjugates Type II} where $w = {\rm id}$, we have from~\eqref{pi formula} that $\pi' = \pi = 0$.  
Therefore, again by~\eqref{pi formula}, it actually suffices to compare the differences
\begin{equation}
\label{differences}
    w(\la'+\rho') - v(\la'+\rho') \qquad \text{and} \qquad w(\la'+\rho')^\sharp -   v(\la'+\rho')^\sharp,
\end{equation}
which (upon taking the dual) will tell us how $\pi'$ and $\pi$ are constructed as $[\Phi_w]_{\la'}$ is built box by box.
We adopt the convention of building $[\Phi_w]_{\la'}$ by rows from top to bottom, adding boxes in a given row from left to right.  
Hence for our purposes, there are at most two ways to add a box to $[\Phi_v]_{\la'}$: either add a new box along the diagonal, or add a box to the bottom row.

\bigskip

\textbf{Case 1: adding a diagonal box.}
Suppose that $[\Phi_w]_{\la'}$ is obtained by adding the $i$th diagonal box to $[\Phi_v]_{\la'}$.
This box corresponds to the root $2\ep_{n+1-i}$, and by the proof of Lemma~\ref{lemma:pi'} the entry in this box is $k+1 = 3$.
Thus by~\eqref{w-dot as boxes}, $w(\la'+\rho')$ is obtained from $v(\la' + \rho')$ by subtracting $3 \cdot 2\ep_{n+1-i}$, that is, subtracting $6$ from the $i$th coordinate from the right.
We claim (but we save for the actual proof) that this coordinate in $v(\la' + \rho')$ is necessarily $k+1 = 3$, which therefore is negated by our subtraction.
In Figure~\ref{fig:example conjugates Type II}, we thus have
    \begin{equation}
    \label{add diag box}
        v(\la' + \rho') = ( \ldots, \posarrow{3}{\substack{\text{$i$th from the right,}\\\text{ignoring the $\bbox$}}}, \bbox, \ldots) \quad \leadsto \quad  w(\la' + \rho') = \quad (\ldots, \bbox, \neg{3}, \ldots).
    \end{equation}
    Therefore, on one hand, ignoring the $\bbox$ in~\eqref{add diag box} gives us
    \[
        w(\la'+\rho') - v(\la' + \rho') = (0, \ldots, 0,\posarrow{\neg{6}}{\text{$i$th from the right}},0, \ldots, 0),
    \]
    which (upon taking the dual) creates the $i$th arm in $\pi'$, with length $m=6-1=5$. 
    On the other hand, including the $\bbox$ in~\eqref{add diag box} gives us
    \[
        w(\la'+\rho')^\sharp - v(\la' + \rho')^\sharp = (0, \ldots,0,\underbrace{\neg{1},\ldots,\neg{1}}_{\mathclap{\substack{\text{$m+1$ coordinates},\\
        \text{ending $i$th from the right}}}},0,\ldots,0),
    \]
    which (upon taking the dual) creates the $i$th leg in $\pi$, with length $m=5$.  
    (This fact requires that each previous leg must be strictly longer than $m$, which will be clear from Case 2 below.)  
    Hence the case of adding a diagonal box preserves the conjugate shapes of $\pi'$ and $\pi$.

    As a specific example of Case 1, in Figure~\ref{fig:example conjugates Type II} we choose $[\Phi_v]_{\la'} = \ytableaushort[*(lightgray)]{31}$, and add a diagonal box to begin row $i=2$, thus obtaining $[\Phi_w]_{\la'}$ to the southeast.  
    On the left side, we thus have
    \[
        (5,3,\bbox,\neg{4}) \quad \leadsto \quad (5,\bbox,\neg{3},\neg{4}).
    \]
    We confirm that this has the effect of negating the $3$, which was indeed in position $i=2$ from the right.  
    Moreover, we have the differences $(5,\neg{3},\neg{4}) - (5,3,\neg{4}) = (0,\neg{6},0)$ and $(5,\bbox,\neg{3},\neg{4}) - (5,3,\bbox,\neg{4}) = (0,\neg{1},\neg{1},\neg{1},\neg{1},\neg{1},\neg{1},0)$.

\bigskip

    \textbf{Case 2: adding a non-diagonal box.}  Suppose that $[\Phi_w]_{\la'}$ is obtained by adding the $j$th non-diagonal box to the $i$th row of $[\Phi_v]_{\la'}$, which we assume is the bottom row.  
    This box corresponds to the root $\ep_{n+1-i} + \ep_{n+1-i-j}$, and by the proof of Lemma~\ref{lemma:pi'} the entry in this box is 1.  
    Thus by~\eqref{w-dot as boxes}, $w(\la'+\rho')$ is obtained from $v(\la'+\rho')$ by subtracting $1 \cdot (\ep_{n+1-i} + \ep_{n+1-i-j})$, that is, subtracting 1 from the $i$th and $(i+j)$th coordinates from the right.  
    We claim (but we reserve for the proof)  that these coordinates are necessarily $\neg{j+k}$ and $j+k+1$, which are therefore transposed and negated by our subtraction.  
    In our current example where $k=2$, these two coordinates are $\neg{j+2}$ and $j+3$. 
    In Figure~\ref{fig:example conjugates Type II}, we thus have
    \begin{equation}
    \label{add to bottom row}
        v(\la'+\rho') = (\ldots, \posarrow{j+3}{\substack{\text{$(i+j)$th} \\ \text{from the right,}\\ \text{ignoring the $\bbox$}}}, \ldots\ldots, \bbox, \posarrow{\neg{j+2}}{\substack{\text{$i$th} \\ \text{from the right}}}, \ldots) \quad \leadsto \quad  w(\la'+\rho') = (\ldots, j+2, \ldots, \bbox, \neg{j+3}, \ldots).
    \end{equation}
    Therefore, on one hand, ignoring the $\bbox$ in~\eqref{add to bottom row} gives us
    \[
        w(\la' + \rho') - v(\la'+\rho') = (0, \ldots, 0,\posarrow{\neg{1}}{\substack{\text{$(i+j)$th} \\ \text{from right}}}, 0, \ldots\ldots, 0, \posarrow{\neg{1}}{\substack{\text{$i$th} \\ \text{from right}}}, 0, \ldots, 0),
    \]
    which (upon taking the dual) adds a box to the $i$th arm and the $i$th leg of $\pi'$. 
    (Note that each previous leg of $\pi$ must be strictly longer than $j$, by the very fact that we are able to add the $j$th box to the bottom row of $[\Phi_v]_{\la'}$.) 
    On the other hand, including the $\bbox$ in~\eqref{add to bottom row} gives us
    \[
        w(\la'+\rho')^\sharp - v(\la'+ \rho')^\sharp = (0, \ldots, 0,\posarrow{\neg{1}}{\substack{\text{$(m+i+j)$th} \\ \text{from right}}}, 0, \ldots\ldots, 0, \posarrow{\neg{1}}{\substack{\text{$i$th} \\ \text{from right}}}, 0, \ldots, 0),
    \]
    which (upon taking the dual) also adds a box to the $i$th arm and the $i$th leg of $\pi$.  
    Hence the case of adding a non-diagonal box preserves the conjugate shapes of $\pi'$ and $\pi$.

As a specific example of Case 2, we again choose $[\Phi_v]_{\la'} = \ytableaushort[*(lightgray)]{31}$, but this time we add box $j=2$ to the bottom (i.e., only) row $i=1$ to obtain $[\Phi_w]_{\la'}$ to its southwest.  
On the left side, we thus have
    \[
        (5,3,\bbox,\neg{4}) \quad \leadsto \quad (4,3,\bbox,\neg{5}).
    \]
We confirm that this has the effect of subtracting $1$ from the coordinates $\neg{j+2} = \neg{4}$ and $j+3 = 5$, which were in positions $i=1$ and $i+j = 3$, counting from the right.  
Moreover, we have the differences  $(4,3,\neg{5}) - (5,3,\neg{4}) = (\neg{1},0,\neg{1})$ and $(4,3,\bbox,\neg{5}) - (5,3,\bbox,\neg{4}) = (\neg{1},0,0,0,0,0,0,\neg{1})$.
\end{ex}

Having concluded our preliminary example in Type II, we proceed to prove Theorem~\ref{thm:Cong and Conj}.  
The proof, for each of the types, merely makes rigorous the same idea that drives Example~\ref{ex:D8 and C3}; for this reason, we encourage the reader to begin reading the proof for Type II before the other types.

\begin{proof}[Proof of Theorem~\ref{thm:Cong and Conj}] 
Once we have shown for each type that $\pi'$ and $\pi$ are conjugate partitions, the congruence of blocks $\B_{\la'}$ and $\B_{\la}$ will follow immediately from Definition~\ref{def:congruence}, Lemma~\ref{lemma:pi'}, and Theorems~\ref{theorem:ID-dim-GLn} and~\ref{thm:dim GLn pairs}.  
Hence we prove the conjugate property for each type:

\bigskip

\textbf{Type I.} As for Type II below, the proof relies on the following analogue of~\eqref{claim in main proof}: if $[\Phi_w] = (\al - 1 \mid \be - 1)$ has rank $r$, then we have
\begin{align}
    \label{claim Type I}
    \begin{split}
    w(c_p, \ldots, c_1&;d_1,\ldots,d_q)\\
    = &(c_p, \ldots, \widehat{c_{\be_1}}, \ldots, \widehat{c_{\be_r}}, \ldots, c_1,d_{\al_r}, \ldots,d_{\al_1}; c_{\be_1}, \ldots, c_{\be_r},d_1, \ldots, \widehat{d_{\al_r}}, \ldots, \widehat{d_{\al_1}}, \ldots, d_q),
    \end{split}
\end{align}
where the hats denote missing coordinates.  
This can be shown inductively using the same method as in~\eqref{claim in main proof}. 

Recall from Lemma~\ref{lemma:pi'} that $\la+\rho = (p+q+k-1,\ldots,q+k,\bbox;\bbox,q-1,\ldots,0)$, where $\bbox = (q+k-1, \ldots, q)$ is the string deleted on either side of the semicolon via Enright--Shelton reduction.  
Therefore $\la'+\rho' = (p+q+k-1,\ldots,q+k;q-1,\ldots,0)$.  
By~\eqref{pi formula}, in the base case where $w = {\rm id}$ we have $\pi' = \pi = 0$, and so just as in Example~\ref{ex:D8 and C3}, it suffices to consider the differences~\eqref{differences} when $[\Phi_w]_{\la'}$ is obtained by adding a single box to $[\Phi_v]_{\la'}$.
We consider three cases for the location of this additional box: namely, the main diagonal,  the bottommost arm, or the rightmost leg.

First suppose we obtain $[\Phi_w]_{\la'}$ by adding the $i$th diagonal box to $[\Phi_v]_{\la'}$.
This box corresponds to the root $\ep_{p+1-i}-\ep_{p+i}$, and by the proof of Lemma~\ref{lemma:pi'} the entry in this box is $k+1$.
Thus by~\eqref{w-dot as boxes}, $w(\la'+\rho')$ is obtained from $v(\la' + \rho')$ by subtracting $(k+1)(\ep_{p+1-i}-\ep_{p+i})$, that is, by subtracting $k+1$ from the $i$th coordinate to the left of the semicolon, and adding $k+1$ to the $i$th coordinate to the right of the semicolon.
It follows from~\eqref{claim Type I} that these coordinates are $q+k$ on the left, and $q-1$ on the right, and hence we have
\[
( \ldots,\posarrow{q+k}{\text{$i$th to the left}},\bbox,\ldots; \ldots, \bbox,\posarrow{q-1}{\text{$i$th to the right}},\ldots) \quad \leadsto \quad ( \ldots,\bbox,\posarrow{q-1}{\text{$i$th to the left}},\ldots; \ldots, \posarrow{q+k}{\text{$i$th to the right}},\bbox,\ldots),
\]
counting coordinates from the semicolon, without counting the $\bbox$.  
On one hand, ignoring the $\bbox$ gives us
 \[
        w(\la'+\rho') - v(\la'+\rho') = (0, \ldots, 0, \posarrow{\neg{k+1}}{\text{$i$th to the left}},0, \ldots,0;0,\ldots,0,\posarrow{k+1}{\text{$i$th to the right}},0,\ldots,0),
 \]
 which (upon taking the dual) creates the $i$th \emph{arm} of length $m=k$ in each of the wings of $\pi'$.  
 (Recall Remark~\ref{rem:butterfly} on the butterfly diagrams that represent $\pi'$ and $\pi$.)  
 On the other hand, including the $\bbox$ gives us
 \[
    w(\la'+\rho')^\sharp - v(\la' + \rho')^\sharp = (0,\ldots,0,\underbrace{\neg{1},\ldots,\neg{1}}_{\mathclap{\substack{\text{$m+1$ coordinates,} \\ \text{ending $i$th to the left}}}},0,\ldots,0;0,\ldots,0,\underbrace{1,\ldots,1}_{\mathclap{\substack{\text{$m+1$ coordinates,} \\ \text{starting $i$th to the right}}}},0,\ldots,0),
 \]
which (upon taking the dual) creates the $i$th \emph{leg} of length $m$ in each of the wings of $\pi$.

The remaining two cases of adding a box work out similarly, following from~\eqref{claim Type I}.
In particular, adding a box to the bottommost (say $i$th) arm of $[\Phi_v]_{\la'}$ corresponds to the root $\ep_{p+1-i} - \ep_{p+j}$ (for some $j>i$), and by the proof of Lemma~\ref{lemma:pi'} the entry in this box is 1.
Thus by adding this box, we add 1 to the $i$th arm of each wing of $\pi'$, and we add 1 to the $i$th \emph{leg} of each wing of $\pi$.
Similarly, by adding a box to the rightmost (say $j$th) leg of $[\Phi_v]_{\la'}$, we add 1 to the $j$th leg of each wing of $\pi'$, and we add 1 to the $j$th arm of each wing of $\pi$.
Hence in all three cases, $\pi$ is obtained from $\pi'$ by a reflection about the 45-degree axis, and so the two are conjugates.

\bigskip

\textbf{Type II.}  Our only task is to verify the two inductive claims made in Example~\ref{ex:D8 and C3}.  
In Case 1, we claimed that whenever it is possible to add the $i$th diagonal box to $[\Phi_v]_{\la'}$, the weight $v(\la'+\rho')$ must have $k+1$ as its $i$th coordinate from the right.  
In Case 2, we claimed that whenever it is possible to add the $j$th non-diagonal box to the $i$th row of $[\Phi_v]_{\la'}$, the weight $v(\la'+\rho')$ must have $\neg{j+k}$ as its $i$th coordinate, and $k+j+1$ as its $(i+j)$th coordinate, counting from the right.  
Both of these claims will be immediate once we prove the following fact, where $[\Phi_v] = (\al-1\mid 0)$, so that the row lengths are given by $\al_1 > \cdots > \al_r > 0$:
\begin{equation}
\label{claim in main proof}
v(c_n, \ldots, c_1) = (c_n, \ldots, \widehat{c_{\al_1}}, \ldots, \widehat{c_{\al_r}}, \ldots, c_1, \neg{c_{\al_r}}, \ldots, \neg{c_{\al_1}}),
\end{equation}
where the hats denote missing coordinates.  
To prove~\eqref{claim in main proof}, first consider the case $r=1$, so that by~\eqref{w-in-terms-of-v} we have $v = s_{n,n+1-\al_1} \cdots s_{n,n}$.  
We proceed by induction on $\al_1$.  
In the base case $\al_1=1$, we have $v = s_{n,n}$, and so the effect of applying $v$ is to negate $c_1$.  
Thus $v(c_n, \ldots, c_1) = (c_n,\ldots, c_2, \neg{c_1})$, which agrees with~\eqref{claim in main proof}.  
Now assuming that~\eqref{claim in main proof} holds for $r=1$, let $[\Phi_w]$ be obtained from $[\Phi_v]$ by adding the $(\al_1+1)$th box in row 1. 
Then $w(c_n, \ldots, c_1)$ is obtained from $v(c_n, \ldots, c_1)$ by transposing and negating the $1$st and $(\al_1+1)$th coordinates from the right; therefore we have
\begin{equation}
\label{r=1}
w(c_n, \ldots, c_1) = s_{n,n-\al_1}(c_n, \ldots, \posarrow{c_{\al_1+1}}{\text{$(\al_1+1)$th from right}},\widehat{c_{\al_1}}, \ldots\ldots, c_1,\posarrow{\neg{c_{\al_1}}}{\text{$1$st from right}}) = (c_n, \ldots, \widehat{c_{\al_1+1}}, \ldots, c_1,\neg{c_{\al_1+1}})
\end{equation}
as desired.  
Proceeding by induction on $r$, we assume~\eqref{claim in main proof} and suppose that  $[\Phi_w]$ is obtained from $[\Phi_v]$ by adding a row of length $\al_{r+1}<\al_r$.  
But then $w(c_n, \ldots, c_1)$ is obtained from $v(c_n, \ldots, c_1)$ by applying the result~\eqref{r=1}, having replaced $1$ with $r+1$, to the string $(c_{\al_{r}-1}, \ldots, c_1)$, which thus becomes $(c_{\al_{r}-1}, \ldots, \widehat{c_{\al_{r+1}}}, \ldots, c_1, \neg{c_{\al_{r+1}}})$.
This proves~\eqref{claim in main proof}, and upon setting $(c_n,\ldots,c_1) = \la'+\rho' = (k+n, \ldots, k+1)$, the two claims from Example~\ref{ex:D8 and C3} follow immediately.

\bigskip
\textbf{Types IIIabcd.} The proof is essentially identical to Type II.  
The analogue to~\eqref{claim in main proof} is the following: if $[\Phi_w] = (\al - 2 \mid 0)$ with rank $r$, then we have
\[
    w(c_n,\ldots,c_1) = (c_n, \ldots, \widehat{c_{\al_1}}, \ldots, \widehat{c_{\al_r}}, \ldots, (-1)^r c_1, \neg{c_{\al_r}}, \ldots, \neg{c_{\al_1}}),
\]
so that there is always an even number of negated coordinates.  
The only other substantial difference occurs in Types IIIb and IIId, due to the twisting by the $\la'$ from Types IIIa and IIIc, respectively.  
Hence in Types IIIb and IIId, we must check in the base case ($w = {\rm id}$) that $\pi'$ and $\pi$ are conjugates.  
Indeed, in both types, when $w = {\rm id}$, we have $\pi'$ a single row and $\pi$ a single column of the same length; in Type IIIb this length is $k$, while in Type IIId it is $2k+1$. 
The rest of the proof imitates Type II, and so the details are left to the reader. \end{proof}

\section{Hilbert series and generalized Littlewood identities}
\label{section:HS}

\subsection{Hilbert series of $\widetilde{L}_\la$}

In this section, we derive the Hilbert series of the infinite-dimensional modules $\widetilde{L}_{\la}$ for the families in Table~\ref{table:WC}.  Recall from \eqref{def:L tilde} that $\widetilde{L}_\la \coloneqq L_\la \otimes F_{-\langle \la,\be^\vee\rangle\zeta}$.

Let $M$ be a highest weight $\g$-module with highest weight $\la$, and weight space decomposition $M = \bigoplus_{\mu \leq \la} M_\mu$.  
Recall the distinguished element $h_0 \in \mathfrak{z}(\k)$ from the beginning of Section \ref{sub:Hermitian pairs}.  The \emph{Hilbert series} of $M$ is the formal power series
\begin{equation}
    \label{Hilbert series definition}
    H_M(t) \coloneqq \sum_{\mu \leq \la}\dim(M_\mu) t^{-\mu(h_0)}.
\end{equation}
In the setting of this paper, where the coordinates of $\la$ are either all integers or all half-integers, each of our Hilbert series is an element of  $\mathbb{Z}[[t^{1/2}]]$, and can be written as a rational function of the form
\[
H_M(t) = \frac{P(t)}{(1-t)^d},
\]
with $P \in \mathbb Z[t^{1/2}]$ such that $P(1) \neq 0$.  
Then $P(1)$ is the Bernstein degree of $M$, and $d$ is the Gelfand--Kirillov (GK) dimension of $M$.
\begin{table}[t]
\centering
\input{Table_Hilbert.tex}
\caption{Data for the Hilbert series $H_{\widetilde{L}}(t) = \frac{P(t)}{(1-t)^d}$ of $\widetilde{L}=\widetilde{L}_\la \coloneqq L_\la \otimes F_{-\langle \la,\be^\vee\rangle \zeta}$.  
Just as in Table~\ref{table:WC}, $m$ denotes the number of coordinates deleted via Enright--Shelton reduction, in passing from $\g$ to $\g'$.
Recall that we write $p=P-m$, $q=Q-m$, and $n=N-m$.}
\label{table:Hilbert-series}
\end{table}

\begin{theorem}
\label{thm:Hilbert series}
    For each of the six types in Table~\ref{table:WC}, let $\widetilde{L} = \widetilde{L}_\la$ as defined in \eqref{def:L tilde}.  Then
    \[
        H_{\widetilde{L}}(t) = \frac{P(t)}{(1-t)^d},
    \]
    where $P(t)$ and $d$ are given in Table~\ref{table:Hilbert-series}.  Furthermore, we have $P(t) = H_{\widetilde{L}'}(t)$, where $\widetilde{L}' = \widetilde{L}_{\la'}$.
\end{theorem}

\begin{proof}

By the transfer theorem in~\cite{Enright-Hunziker}*{p.~623} relating Hilbert series to Enright--Shelton reduction, we have
\begin{equation}
    \label{transfer theorem}
    H_{\widetilde{L}}(t) = \frac{\dim F_\la}{\dim F_{\la'}} \cdot \frac{H_{\widetilde{L}'}(t)}{(1-t)^d},
\end{equation}
where $d = \dim(\p^+) - \dim(\p'^+)$.  Due to the congruence of blocks $\B_{\la}$ and $\B_{\la'}$ in Table~\ref{table:WC}, we have $\dim F_\la / \dim F_{\la'} = 1$ in~\eqref{transfer theorem}. 
It remains to verify $H_{\widetilde{L}'}(t)$ and the GK dimension $d$ for each type below.

\bigskip

\textbf{Types I and II.}  Note that in these types, we have $\langle \la',\be'^\vee\rangle \zeta' = \la' = k \zeta'$, where  $\zeta'$ is the unique fundamental weight orthogonal to $\Phi(\k')$, as in Definition~\ref{def:Lambda tilde}.  
Recall that the $\k'$-module $F_{-k\zeta'}$ is $1$-dimensional.  
By Theorem 3.1 in~\cite{EHW}, we have
    \begin{equation}
        \label{schmid decomp}
        \widetilde{L}_{\la'} = L_{k \zeta'} \otimes F_{-k\zeta'} \cong \bigoplus_{k \geq m_1 \geq \cdots \geq m_r \geq 0} F_{-(m_1\gamma_1 + \cdots + m_r \gamma_r)}
    \end{equation}
    as a $\k'$-module, where $\gamma_1 < \cdots < \gamma_r$ are Harish-Chandra's strongly orthogonal noncompact roots for $\g'$, with $\gamma_1 \in \Pi$. 
    Expanding the $\gamma_i$ into standard coordinates, we obtain the following well-known specializations (recall that $\la'$ depends on $k$ in Table~\ref{table:WC}):
\begin{equation}
\label{Schmid I-IIIa}
\renewcommand{\arraystretch}{1.5}
\begin{array}{lll}
 \text{Type I:} & \widetilde{L}_{\la'} \cong \bigoplus_\nu \F{\nu^*}{p} \otimes \F{\nu}{q}, & \nu \in \Par(\min\{p,q\} \times k).\\
 \text{Type II:} & \widetilde{L}_{\la'} \cong \bigoplus_\nu \F{\nu^*}{n}, & \nu \in \Par(n \times 2k) \text{ with even rows.}\\
 \end{array}
\end{equation}
Each summand in~\eqref{schmid decomp} contributes $\dim(F_{-(m_1 \gamma_1 + \cdots + m_r \gamma_r)})t^{m_1 + \cdots + m_r}$ to $H_{{\widetilde{L}'}}(t)$.  Since $\nu^* = -\sum_i m_i \gamma_i$, it is easy to check that in Type I, we have $|\nu| = \sum_i m_i$, while in Type II we have $|\nu| = 2\sum_i m_i$.  
We compute the GK dimension $d$ using Proposition~\ref{prop:BGG}:
\begin{equation*}
\renewcommand{\arraystretch}{1.5}
\begin{array}{ll}
 \text{Type I:} &  d = PQ - pq = (p+k)(q+k) - pq = k(p+q+k).\\
 \text{Type II:} & d = \binom{N}{2} - \binom{n+1}{2} = \binom{n+2k+1}{2} - \binom{n+1}{2} = k(2n+2k+1).
\end{array}
\end{equation*}

\bigskip

\textbf{Types IIIabcd.}  Now let $\g' = \ssD_n$.  Recall that in Types IIIa and IIIb we have $\langle \la',\be'^\vee\rangle \zeta' = k\omega'_n$, while in Types IIIc and IIId we have $\langle \la',\be'^\vee\rangle \zeta' = (2k+1)\omega'_n$.  
By Theorem 22 in~\cite{Enright-Hunziker}, for $a,b \in \mathbb{N}$, we have the following decomposition as a module for $\k' = \gl_n$:
\begin{equation}
\label{so_n decomp}
    L_{a\omega'_{n-1} + b\omega'_n} \otimes F_{-(a+b)\zeta'} \cong \bigoplus_\nu \F{\nu^*}{n},
\end{equation}
where $\nu \in \Par(n \times (a+b))$ such that exactly $a$ columns have odd length.  We therefore have the following:
\begin{equation}
\label{so_n IIIa-d}
\renewcommand{\arraystretch}{1.5}
\begin{array}{|l|l|l|l|l|}
\hline

\text{Type} & \la' & a & b & \nu \in \ldots \\ \hline
\text{IIIa}
 & k\omega'_n & 0 & k & \Par(n \times k), \text{ with even columns}\\ \hline
\text{IIIb} &  k\omega'_{n-1} & k & 0 & \Par(n \times k), \text{ with odd columns}\\ \hline
 \text{IIIc} &(2k+1)\omega'_{n} & 0 & 2k+1 & \Par(n \times (2k+1)), \text{ with even columns}\\ \hline
 \text{IIId} &  (2k+1)\omega'_{n-1} & 2k+1 & 0 & \Par(n \times (2k+1)), \text{ with odd columns}\\ \hline
\end{array}
\end{equation}
(As before, the phrase ``even/odd columns'' means that all column lengths are even/odd.)  Each summand in~\eqref{so_n decomp} contributes $\dim(\F{\nu}{n})t^{|\nu|/2}$ to $H_{{\widetilde{L}'}}(t)$. 
We compute the GK dimension $d$ using Proposition~\ref{prop:BGG}:
\[
\renewcommand{\arraystretch}{1.5}
\begin{array}{ll}

\text{Types IIIa and IIIb:} & d = \binom{N+1}{2} - \binom{n}{2} = \binom{n+k}{2} - \binom{n}{2} = k(2n+k-1)/2.\\

\text{Types IIIc and IIId:} & d = \binom{N}{2} - \binom{n}{2} = \binom{n+2k}{2} - \binom{n}{2} = k(2n+2k-1).
\end{array}
\]
This completes the proof.
\end{proof}

\subsection{Hilbert series for modules of invariants and semi-invariants}

In Types I--IIIa, $H_{\widetilde{L}}(t)$ is the Hilbert series of the determinantal variety containing the matrices in $\M_{P,Q}$, $\AM_N$, or $\SM_N$ with rank at most $k$, $2k$, or $k$, respectively.  
More specifically, the determinantal variety is the associated variety of $L_\la$, which is called the \emph{$k$th Wallach representation}; see~\cites{EW,Enright-Hunziker,EHP}.  
The Hilbert series of these determinantal varieties can also be interpreted combinatorially from a Stanley decomposition of the coordinate ring; this decomposition is obtained from the Bj\"orner shelling \cite{Bjorner} of the $k$th order complex on the poset of matrix coordinates, and can be understood entirely in terms of non-intersecting lattice paths and a modified RSK correspondence.  
Details can be found in \cites{Sturmfels,Krattenthaler}, \cite{Herzog}, and \cite{Conca} for Types I, II, and IIIa, respectively.

From the perspective of classical invariant theory, let $\g$ be of Type $\ssA$, $\ssC$, or $\ssD$, and let $H$ be the complex classical group such that $(H, \g)$ is a dual pair (in the sense of Howe duality); hence $H$ is $\GL_k$, $\Or_k$, or $\Sp_{2k}$, respectively.  Let $V$ be the defining representation of $H$, and let either $W = (V^*)^{\oplus P} \oplus V^{\oplus Q}$ (if $\g = \ssA_{P+Q+1}$) or $W = V^{\oplus N}$ (if $\g = \ssC_N$ or $\ssD_N$).  If $\la$ is as in Types I--IIIa, then as $\g$-modules, we have
\[
\widetilde{L}_\la \cong \C[W]^H \coloneqq \{ f \in \C[W] \mid f(hw) = f(w) \text{ for all } w \in W, \: h \in H\},
\]
where the right-hand side is called the \emph{algebra of invariants}.  
Due to the fact (i.e., the first fundamental theorem of classical invariant theory) that the fundamental variants are quadratics, it is customary to write
\[
    H_{\C[W]^{H}}(t) = H_{\widetilde{L}}(t^2).
\]

If $\g = \ssC_N$ and $H = \Or_k$, with $\la$ as in Type IIIb, then as $\g$-modules we have
\[
\widetilde{L}_\la \cong \C[W]^{H,\det} \coloneqq \{ f \in \C[W] \mid f(hw) = \det h \cdot f(w) \text{ for all } w \in W, \: h \in H \},
\]
where the right-hand side is called the \emph{module of semi-invariants} with respect to the character $\chi = \det$.

\begin{ex}
In Type IIIb, let $N=3$ and $k=1$, so that $H=\Or_1 = \{\pm 1\}$.  
The module $\C[W]^{H,\det}$ of semi-invariants is the span of the odd-degree monomials in $\C[x_1,x_2,x_3]$. 
In Table~\ref{table:Hilbert-series}, we have $m=0$, so that $n=3$ and $d=3$.  Since $\nu$ ranges over partitions in $\Par(3 \times 1)$ with columns of odd length, the sum over $\nu$ ranges over the two  partitions $(1)$ and $(1,1,1)$.  
We observe that $\F{(1)}{3}=\C^3$ and $\F{(1,1,1)}{3} = \bigwedge^3\C^3$, with dimensions $3$ and $1$, respectively.  
Hence from Table~\ref{table:Hilbert-series}, we have the following Hilbert series (after replacing $t$ by $t^2$):
\[
H_{\C[W]^{H,\det}}(t) = \frac{3t+t^3}{(1-t^2)^3.}
\]
This Hilbert series reflects the fact that $\C[W]^{H,\det}$ is a free module over $\C[x_1^2, x_2^2, x_3^2]$ with three generators $x_1,$ $x_2$, $x_3$ of degree 1, and one generator $x_1 x_2 x_3$ of degree 3.
\end{ex}

\subsection{Generalized Littlewood identities}

Recall from Section~\ref{section:ID's and BGG} that the classical identities~\eqref{Dual-Cauchy},~\eqref{Littlewood-C}, and~\eqref{Littlewood-D} can each be interpreted as the Euler characteristic of the BGG resolution of the trivial representation of $\g'$ in Type I, II, and III, respectively.  
In this subsection, by combining our results in Tables~\ref{table:WC} and~\ref{table:Hilbert-series}, we generalize these identities by considering the BGG resolution of the finite-dimensional $\g'$-module $\widetilde{L}_{\la'}$.  
We first point out one of Littlewood's identities~\cite{Littlewood}*{(11.9;5)} that has not yet appeared in this paper: 
\begin{equation}
    \label{Littlewood-B}
    \tag{IIIab}
    \prod_{\mathclap{1 \leq i \leq n}} (1-x_i) \prod_{\mathclap{1 \leq i<j \leq n}}(1-x_i x_j) = \sum_{\substack{\pi: \\ \pi = \pi' }}(-1)^{(|\pi|+\rk \pi)/2} s_\pi(x_1, \ldots, x_n),
\end{equation}
where the sum is over all self-conjugate partitions $\pi$ with at most $n$ parts.
See~\cite{Macdonald}*{Ex.~5.9} for an interpretation of~\eqref{Littlewood-B} in terms of the root system for Type $\mathsf{B}$.

\begin{theorem}
\label{thm:new IDs}
    For each $k \in \mathbb{N}$, we have the three identities in Table~\ref{table:identities}.  
    Moreover,
    \begin{itemize}
    \item upon setting $k=0$, our identities reduce to the classical identities~\eqref{Dual-Cauchy},~\eqref{Littlewood-C}, and~\eqref{Littlewood-D};
    \item upon setting $k=1$, the sum of our identities in Types IIIa and IIIb reduces to the Littlewood identity~\eqref{Littlewood-B}.
    \end{itemize}
\end{theorem}

\begin{table}[ht]
\centering
\input{Table_IDs.tex}
\caption{Identities generalizing the classical identities~\eqref{Dual-Cauchy}--\eqref{Littlewood-D}.  
The shorthand $\bfx$ and $\bfy$ are as in Table~\ref{table:Type123}.}
\label{table:identities}
\end{table}

A word on notation: the identities in Table~\ref{table:identities} can be stated without any reference to congruent blocks or Enright--Shelton reduction, and so we have omitted the prime symbol on the partitions $\pi$.  
We therefore continue this convention in the proof below, despite using the fact that these partitions are elements of the poset $\tL(\B_{\la'})$ from earlier in the paper.  
Likewise, we omit the prime symbol on the weights $\mu$, which earlier we denoted by $\mu' \in \Lambda(\B_{\la'})$.

\begin{proof}[Proof of Theorem~\ref{thm:new IDs}]

On one hand, by using~\eqref{Schmid I-IIIa} and~\eqref{so_n IIIa-d}, we can easily write down the character of $\widetilde{L}_{\la'}$ as the sum of Schur polynomials ranging over certain partitions $\nu$; this is the sum on the left-hand side of the identities in Table~\ref{table:identities}.  

On the other hand, $\ch \widetilde{L}_{\la'}$ must equal the alternating sum of the characters of the parabolic Verma modules occurring in the BGG resolution for $\widetilde{L}_{\la'}$.  
Since tensoring with a $1$-dimensional $\k'$-module is an exact functor, we can obtain the BGG resolution of $\widetilde{L}_{\la'}$ from that of $L_{\la'}$, by replacing each $N_{\mu}$ with $\widetilde{N}_{\mu} \coloneqq N_{\mu} \otimes F_{-\langle \la',\be'^\vee\rangle \zeta'}$.  
Then 
\begin{equation}
\label{Etilde char ID}
    \ch \widetilde{L}_{\la'} = \sum_{\mu \in \Lambda(\B_{\la'})} (-1)^{\ell(\mu)}\ch \widetilde{N}_{\mu},
\end{equation}
where $\ell(\mu)$ is defined to be $\ell(w)$ for the unique $w \in \prescript{\k'\!}{}{\W}$ such that $\mu = w \cdot \la'$. 
As noted above, the set $\{\widetilde{N}_{\mu} \mid \mu \in \Lambda(\B_{\la'})\}$ of Verma modules occurring in~\eqref{Etilde char ID} equals the set $\{N_{\pi^*} \mid \pi^* \in \tL(\B_{\la'})\}$, which we understand completely from Table~\ref{table:WC}.  
Hence we can rewrite~\eqref{Etilde char ID} as
\begin{equation}
\label{Etilde char ID final}
    \ch \widetilde{L}_{\la'} = \sum_{\pi^* \in \tL(\B_{\la'})} (-1)^{\ell(\pi)} \ch N_{\pi^*}
\end{equation}
where $\ell(\pi)$ is defined to be $\ell(\mu)$ such that $\mu \in \Lambda(\B_{\la'})$ is the preimage of $\pi^*$.
Thus, by~\eqref{length-size-diagram}, computing $\ell(\pi)$ is just a matter of recovering $|[\Phi_w]|$ from $\rows\,\st{w}_{\la'}$, which we compute in Table~\ref{table:identities}.  
Hence the right-hand side of each identity in the table is given by~\eqref{Etilde char ID final}, except that we multiply both sides by $\ch S(\p^-)$, found in Table~\ref{table:Type123}.

Setting $k=0$ forces $\nu = 0$, so that the $\sum_\nu$ sum is 1.  
For Types I and II, we then immediately recover the classical identities~\eqref{Dual-Cauchy} and~\eqref{Littlewood-C}. 
In Types IIIa and IIIb, the $k=0$ case is degenerate because the $\pi$'s are simply the stacked diagrams $\st{w}$, as in the proof of Proposition~\ref{prop:BGG}; hence the condition on the parity of the rank of $\pi$ disappears, and the two identities collapse into the Littlewood identity~\eqref{Littlewood-D}. 

Setting $k=1$ in Types IIIa and IIIb, we see that $\nu$ must be a single column and so $s_\nu(\bfx)$ is an elementary symmetric polynomial; therefore $\sum_{\nu} s_\nu(\bfx) = \prod_i (1+x_i)$. 
Adding these two identities together, we obtain
\begin{equation}
\label{almost Littlewood IIIab}
    \prod_{i < j} (1-x_i x_j) \prod_{i} (1+x_i) = \sum_\pi(-1)^{(|\pi| - k \rk \pi)/2} s_\pi(\bfx),
\end{equation}
where (since $m=k-1 = 0$) the sum ranges over self-conjugate partitions $\pi$, just as in~\eqref{Littlewood-B}.  
Now let $-\bfx \coloneqq (-x_1, \ldots, -x_n)$.  Since $s_\pi(\bfx)$ is homogeneous of degree $|\pi|$, we have $s_\pi (-\bfx) = (-1)^{|\pi|} s_\pi(\bfx)$. 
Moreover, since $\frac{|\pi|+ \rk \pi}{2} + \frac{|\pi| - \rk \pi}{2} = |\pi|$, the two addends have the same parity if and only if $|\pi|$ is even, and therefore
\[
(-1)^{(|\pi| - \rk \pi)/2}s_\pi(-\bfx) = (-1)^{(|\pi|+ \rk \pi)/2} s_\pi(\bfx).
\]
Thus, if we make the substitutions $x_i \mapsto -x_i$ in~\eqref{almost Littlewood IIIab}, we recover the Littlewood identity~\eqref{Littlewood-B}.
\end{proof}

\begin{rem}
The analogous identities corresponding to Types IIIc and IIId reduce to those in Types IIIa and IIIb, upon replacing $2k$ by $k-1$.  
We can, however, view the identities for Types IIIa and IIIb as the two extreme special cases of a more general identity, since they follow from the first two special cases~\eqref{so_n IIIa-d} of the general branching rule~\eqref{so_n decomp}.  
Specifically, for $a,b \in \mathbb{N}$, we have 
\begin{equation}
\label{identity IIIab}
\prod_{i < j} (1 - x_i x_j) \sum_\nu s_\nu(\bfx) = \sum_\pi (-1)^{\ell(\pi)} s_\pi(\bfx),
\end{equation}
where $\nu \in \Par(n\times (a+b))$ has $a$ odd columns and $b$ even columns, and where the right-hand sum ranges over all partitions of the form
\[
\pi = 
\begin{cases}
(\al+a+b-1 \mid \al) \text{ with odd rank}, & a=0,\\
(\al+a+b-1 \mid \al) \text{ with even rank}, & b=0,\\
(\al+a+b-1,a-1 \mid\al,0) \text{ with odd rank} & \\ \qquad \text{or} & \\
(\al+a+b-1, b-1 \mid \al,0) \text{ with even rank},
& a,b>0,
\end{cases}
\]
for $\al_1 < n$.  
(When $a=b=0$, the first two cases collapse into the single case in the Littlewood identity~\eqref{Littlewood-D}, as mentioned above.)  
The Frobenius symbols in the last case describe an additional arm, whose length is either $a-1$ or $b-1$.  
In \eqref{identity IIIab} we also have the more complicated formula 
\[
\ell(\pi) = \frac{|\pi|- 2(a+b)\lfloor\frac{\rk \pi}{2}\rfloor - (-1)^{\rk \pi} a}{2}.
\]
\end{rem}

\section{Open problems}
\label{sec:open probs}

\begin{prob}
    \label{prob:dimension equalities}
     Determine all equalities $\dim \F{\sigma}{s} = \dim \F{\tau}{t}$.
 \end{prob}

 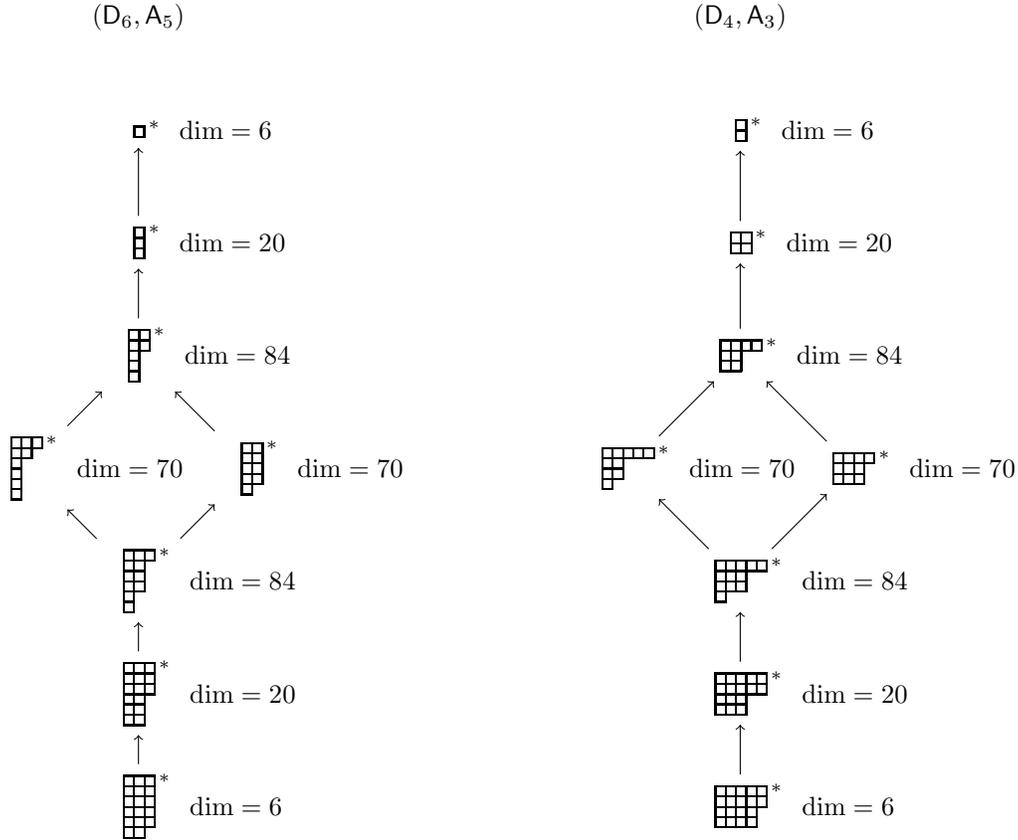
\begin{figure}[t]
     \centering
     \input{Sporadic_example.tex}
     \caption{A sporadic example of congruent blocks, where corresponding poset elements are not given by conjugate partitions.  
     On the left we have $\tL(\B_{\la})$, for $\g = \ssD_6$ with $\la = (\neg{1}, \neg{1}, \neg{1}, \neg{1}, \neg{1}, \neg{2})$.  
     On the right we have $\tL(\B_{\la'})$, for $\g' = \ssD_4$ with $\la' = (1,1,0,0)$.}
     \label{fig:sporadic example}
 \end{figure}

 Theorem~\ref{thm:dim GLn pairs} constitutes a first step toward the solution of Problem~\ref{prob:dimension equalities}, in giving the infinite families where $\sigma = (\al+m\mid \al)$, $\tau = \sigma'$, and $t = s+m$.  
 This problem is similar in flavor to the question of classifying the equalities $\binom{a}{b} = \binom{c}{d}$ among binomial coefficients.  
 There is an infinite family of these discovered by Lind~\cite{Lind} and Singmaster~\cite{Singmaster}; moreover, de Weger~\cite{deWeger} has conjectured that this family comprises \emph{all} nontrivial equalities, with the exception of seven sporadic instances.
 
 \begin{prob}
    \label{prob:classify congruent blocks}
    Classify all instances of congruent blocks (in the context of Hermitian symmetric pairs).
 \end{prob}
 
 As a starting point for Problem~\ref{prob:classify congruent blocks}, we highlight a sporadic example of congruent blocks $\B_{\la}$ and $\B_{\la'}$ occurring outside the families in Table~\ref{table:WC}.
 In particular, let $\g = \ssD_6$ with $\la = -2\omega_6 + \omega^*_1 = (\neg{1},\neg{1},\neg{1},\neg{1},\neg{1},\neg{2})$.  
 Applying Enright--Shelton reduction, we obtain $\g'=\ssD_4$ with $\la' = (1,1,0,0)$.  
 Upon twisting the weights in $\Lambda(\B_{\la})$ and $\Lambda(\B_{\la'})$ by subtracting $-2\omega_6$ and $2\omega_4$, respectively, we obtain the posets $\tL(\B_{\la})$ and $\tL(\B_{\la'})$ shown in Figure~\ref{fig:sporadic example}.  
 We label each poset element with the dimension of the $\k$-module (or $\k'$-module) with the corresponding highest weight.  
 Note that the blocks $\B_{\la}$ and $\B_{\la'}$ are indeed congruent, although their corresponding poset elements are not given by conjugate partitions.

 \subsection*{Acknowledgments} Originally our proofs of Theorems~\ref{theorem:ID-dim-GLn} and \ref{thm:dim GLn pairs} relied on the Weyl dimension formula; we would like to thank Daniel Herden for suggesting a simpler argument via the hook--content formula.
 We thank the anonymous referees for their helpful comments and suggestions.

\bibliographystyle{plain}
\bibliography{references}

\end{document}

%% file: Example_Introduction.tex
\ytableausetup{boxsize=.5em}

\[\begin{tikzpicture}
  \node (g) at (-4,6) {$(\ssD_4, \ssA_3)$};

  \node[label=right: {$\dim=1$}] (a) at  (-4,4.5) {$\bullet$};

  \node[label=right: {$\dim=6$}] (b) at  (-4,3) {$\phantom{*\,}\ydiagram[*(gray!40)]{1}*{1,1}^{\,*}$};

  \node[label=right: {$\dim=15$}] (c) at  (-4,1.5) {$\phantom{*\,}\ydiagram[*(gray!40)]{2}*{2,1,1}^{\,*}$};

  \node[label=right: {$\dim=10$}] (d) at (-5.5,0) {$\phantom{*\,}\ydiagram[*(gray!40)]{2,1+1}*{2,2,2}^{\,*}$};

  \node[label=right: {$\dim=10$}] (e) at (-2.5,0) {$\phantom{*\,}\ydiagram[*(gray!40)]{3}*{3,1,1,1}^{\,*}$};

  \node[label=right: {$\dim=15$}] (f) at  (-4,-1.5) {$\phantom{*\,}\ydiagram[*(gray!40)]{3,1+1}*{3,2,2,1}^{\,*}$};

  \node[label=right: {$\dim=6$}] (g) at  (-4,-3) {$\phantom{*\,}\ydiagram[*(gray!40)]{3,1+2}*{3,3,2,2}^{\,*}$};

  \node[label=right: {$\dim=1$}] (h) at  (-4,-4.5) {$\phantom{*\,}\ydiagram[*(gray!40)]{3,1+2,2+1}*{3,3,3,3}^{\,*}$};

  \draw [<-] (a) edge (b) (b) edge (c) (c) edge (d) (d) edge (f) (f) edge (g) (g) edge (h) (c) edge (e) (e) edge (f);

  \node (g) at (4,6) {$(\ssC_3, \ssA_2)$};

  \node[label=right: {$\dim=1$}] (a) at (4,4.5) {$\bullet$};

  \node[label=right: {$\dim=6$}] (b) at (4,3) {$\phantom{*\,}\ydiagram[*(gray!40)]{1+1}*{2}^{\,*}$};

  \node[label=right: {$\dim=15$}] (c) at (4,1.5) {$\phantom{*\,}\ydiagram[*(gray!40)]{1+2}*{3,1}^{\,*}$};

  \node[label=right: {$\dim=10$}] (d) at (2.5,0) {$\phantom{*\,}\ydiagram[*(gray!40)]{1+2,2+1}*{3,3}^{\,*}$};

  \node[label=right: {$\dim=10$}] (e) at (5.5,0) {$\phantom{*\,}\ydiagram[*(gray!40)]{1+3}*{4,1,1}^{\,*}$};

  \node[label=right: {$\dim=15$}] (f) at (4,-1.5) {$\phantom{*\,}\ydiagram[*(gray!40)]{1+3,2+1}*{4,3,1}^{\,*}$};

  \node[label=right: {$\dim=6$}] (g) at (4,-3) {$\phantom{*\,}\ydiagram[*(gray!40)]{1+3,2+2}*{4,4,2}^{\,*}$};

  \node[label=right: {$\dim=1$}] (h) at (4,-4.5) {$\phantom{*\,}\ydiagram[*(gray!40)]{1+3,2+2,3+1}*{4,4,4}^{\,*}$};

  \draw [<-] (a) edge (b) (b) edge (c) (c) edge (d) (d) edge (f) (f) edge (g) (g) edge (h) (c) edge (e) (e) edge (f);

\end{tikzpicture}\]

%% file: Table_basics.tex
\resizebox{\linewidth}{!}{
\begin{tblr}{colspec={|Q[m,c]|Q[m,c]|Q[m,c]|Q[m,c]|Q[m,c]|Q[m,c]|Q[m,c]|},stretch=1.5}

\hline

Type & $\g$ & {$\g =$ \\ $\left\{\left[\begin{smallmatrix}A&B\\C&D\end{smallmatrix}\right] \: \middle| \: \ldots \right\}$} & {$\k$ \\ $\left\{\left[\begin{smallmatrix}A&0\\0&D\end{smallmatrix}\right]\right\}$} & {$\p^+$ \\ $ \left\{\left[\begin{smallmatrix}
    0 & B \\ 0 & 0
\end{smallmatrix}\right]\right\}$} & $\ch S(\p^-)$ & $\nu \in \ldots$ \\ \hline[2pt]

I & $\sl_{p+q}$ & {$\tr A + \tr D = 0$} & $\mathfrak{s}(\gl_p \oplus \gl_q)$ & $\M_{p,q}$  & $\displaystyle\prod_{\substack{1 \leq i \leq p, \\ 1 \leq j \leq q \phantom{,}}}^{\phantom{A}} \frac{1}{1-x_i y_j} = \sum_\nu s_\nu(\bfx) s_\nu(\bfy)$ & $\Par(\min\{p,q\} \times \infty)$ \\ \hline

II & $\sp_{2n}$ & {$A=-D^t,$ \\ $B=B^t,$ \\$C=C^t$} & $\gl_n$ & $\SM_n$ & $\displaystyle\prod_{1 \leq i \leq j \leq n} \frac{1}{1-x_i x_j} = \sum_\nu s_\nu(\bfx)$ & {$\Par(n \times \infty),$ \\ even rows}\\ \hline

III & $\so_{2n}$ & {$A=-D^t,$ \\ $B=-B^t,$ \\ $C=-C^t$} & $\gl_n$ & $\AM_n$ & $\displaystyle\prod_{1 \leq i < j \leq n } \frac{1}{1-x_i x_j} = \sum_\nu s_\nu(\bfx)$ & {$\Par(n \times \infty),$ \\ even columns}\\ \hline

\end{tblr}
}

%% file: Empty_diagrams_all.tex
\begin{center}
\begin{tikzpicture}[scale=0.65,
every node/.style={scale=0.7},
baseline=(current bounding box.north)]

\foreach[parse=true] \y in {0,...,5} {
    \foreach[parse=true] \x in {0,...,3}
    {\draw [fill=lightgray] (\x,-\y) rectangle ++(1,-1);
    }}

\node[left] at (-1,-.5) {$\ep_{p}$};
\draw [->] (-.75,-.5) -- ++ (.5,0);
\node[left] at (-1,-1.5) {$\ep_{p-1}$};
\draw [->] (-.75,-1.5) -- ++ (.5,0);
\node[left] at (-1,-2.45) {$\vdots$};
\draw [->] (-.75,-2.5) -- ++ (.5,0);
\node[left] at (-1,-3.5) {$\ep_{3}$};
\draw [->] (-.75,-3.5) -- ++ (.5,0);
\node[left] at (-1,-4.5) {$\ep_{2}$};
\draw [->] (-.75,-4.5) -- ++ (.5,0);
\node[left] at (-1,-5.5) {$\ep_{1}$};
\draw [->] (-.75,-5.5) -- ++ (.5,0);

\node at (.5, 1) {$-\ep_{p+1}$};
\draw [->] (.5,.75) -- ++ (0,-.5);
\node at (1.5, 1) {$\cdots$};
\draw [->] (1.5,.75) -- ++ (0,-.5);
\node at (2.5, 1) {$\cdots$};
\draw [->] (2.5,.75) -- ++ (0,-.5);
\node at (3.5, 1) {$-\ep_{p+q}$};
\draw [->] (3.5,.75) -- ++ (0,-.5);

\node[scale=1.33] at (2, -7) {Type I};
    
\end{tikzpicture}
\hfill
\begin{tikzpicture}[scale=0.65,
every node/.style={scale=0.7},
baseline=(current bounding box.north)]

\foreach[parse=true] \y in {0,...,5} {
    \foreach[parse=true] \x in {\y,...,5}
    {\draw [fill=lightgray] (\x,-\y) rectangle ++(1,-1);
    }}

\node[left] at (-1,-.5) {$\ep_n$};
\draw [->] (-.75,-.5) -- ++ (.5,0);
\node[left] at (-1,-1.5) {$\ep_{n-1}$};
\draw [->] (-.75,-1.5) -- ++ (.5,0);
\node[left] at (-1,-2.45) {$\vdots$};
\draw [->] (-.75,-2.5) -- ++ (.5,0);
\node[left] at (-1,-3.5) {$\ep_3$};
\draw [->] (-.75,-3.5) -- ++ (.5,0);
\node[left] at (-1,-4.5) {$\ep_2$};
\draw [->] (-.75,-4.5) -- ++ (.5,0);
\node[left] at (-1,-5.5) {$\ep_1$};
\draw [->] (-.75,-5.5) -- ++ (.5,0);

\node at (.5, 1) {$\ep_{n}$};
\draw [->] (.5,.75) -- ++ (0,-.5);
\node at (1.5, 1) {$\ep_{n-1}$};
\draw [->] (1.5,.75) -- ++ (0,-.5);
\node at (2.5, 1) {$\cdots$};
\draw [->] (2.5,.75) -- ++ (0,-.5);
\node at (3.5, 1) {$\ep_{3}$};
\draw [->] (3.5,.75) -- ++ (0,-.5);
\node at (4.5, 1) {$\ep_{2}$};
\draw [->] (4.5,.75) -- ++ (0,-.5);
\node at (5.5, 1) {$\ep_{1}$};
\draw [->] (5.5,.75) -- ++ (0,-.5);

\node[scale=1.33] at (3, -7) {Type II};
    
\end{tikzpicture}
\hfill
\begin{tikzpicture}[scale=0.65,
every node/.style={scale=0.7},
baseline=(current bounding box.north)]

\foreach[parse=true] \y in {0,...,4} {
    \foreach[parse=true] \x in {\y,...,4}
    {\draw [fill=lightgray] (\x,-\y) rectangle ++(1,-1);
    }}

\node[left] at (-1,-.5) {$\ep_n$};
\draw [->] (-.75,-.5) -- ++ (.5,0);
\node[left] at (-1,-1.5) {$\ep_{n-1}$};
\draw [->] (-.75,-1.5) -- ++ (.5,0);
\node[left] at (-1,-2.45) {$\vdots$};
\draw [->] (-.75,-2.5) -- ++ (.5,0);
\node[left] at (-1,-3.5) {$\ep_3$};
\draw [->] (-.75,-3.5) -- ++ (.5,0);
\node[left] at (-1,-4.5) {$\ep_2$};
\draw [->] (-.75,-4.5) -- ++ (.5,0);

\node at (.5, 1) {$\ep_{n-1}$};
\draw [->] (.5,.75) -- ++ (0,-.5);
\node at (1.5, 1) {$\ep_{n-2}$};
\draw [->] (1.5,.75) -- ++ (0,-.5);
\node at (2.5, 1) {$\cdots$};
\draw [->] (2.5,.75) -- ++ (0,-.5);
\node at (3.5, 1) {$\ep_{2}$};
\draw [->] (3.5,.75) -- ++ (0,-.5);
\node at (4.5, 1) {$\ep_{1}$};
\draw [->] (4.5,.75) -- ++ (0,-.5);

\node[scale=1.33] at (2.5, -7) {Type III};
    
\end{tikzpicture}
\end{center}

%% file: Filled_diagrams_dij.tex
\begin{equation}
    \label{dij}
\begin{tikzpicture}[scale=0.8,
every node/.style={scale=0.8},
baseline=(current bounding box.north)]

\foreach[parse=true] \y in {0,...,5} {
    \foreach[parse=true] \x in {0,...,3}
    {\draw[fill=lightgray] (\x,-\y) rectangle ++(1,-1);
    }}
\node at (0.5,-0.5) {$d_{p1}$};
\node at (1.5,-0.5) {$d_{p2}$};
\node at (2.5,-0.5) {$\cdots$};
\node at (3.5,-0.5) {$d_{pq}$};
\node at (0.5,-1.45) {$\vdots$};
\node at (1.5,-1.45) {$\vdots$};
\node at (2.5,-1.5) {$\cdots$};
\node at (3.5,-1.45) {$\vdots$};
\node at (0.5,-2.45) {$\vdots$};
\node at (1.5,-2.45) {$\vdots$};
\node at (2.5,-2.5) {$\cdots$};
\node at (3.5,-2.45) {$\vdots$};
\node at (0.5,-3.5) {$d_{31}$};
\node at (1.5,-3.5) {$d_{32}$};
\node at (2.5,-3.5) {$\cdots$};
\node at (3.5,-3.5) {$d_{3q}$};
\node at (0.5,-4.5) {$d_{21}$};
\node at (1.5,-4.5) {$d_{22}$};
\node at (2.5,-4.5) {$\cdots$};
\node at (3.5,-4.5) {$d_{2q}$};
\node at (0.5,-5.5) {$d_{11}$};
\node at (1.5,-5.5) {$d_{12}$};
\node at (2.5,-5.5) {$\cdots$};
\node at (3.5,-5.5) {$d_{1q}$};

\node[scale=1.25] at (2, -7) {Type I};
    
\end{tikzpicture}
\qquad\qquad
\begin{tikzpicture}[scale=0.8,
every node/.style={scale=0.8},
baseline=(current bounding box.north)]

\foreach[parse=true] \y in {0,...,5} {
    \foreach[parse=true] \x in {\y,...,5}
    {\draw [fill=lightgray] (\x,-\y) rectangle ++(1,-1);
    }}
\node at (0.5,-0.5) {$d_{nn}$};
\node at (1.5,-0.5) {$\cdots$};
\node at (2.5,-0.5) {$\cdots$};
\node at (3.5,-0.5) {$d_{n3}$};
\node at (4.5,-0.5) {$d_{n2}$};
\node at (5.5,-0.5) {$d_{n1}$};
\node at (5.5,-5.5) {$d_{11}$};
\node at (5.5,-4.5) {$d_{21}$};
\node at (5.5,-3.5) {$d_{31}$};
\node at (5.5,-2.45) {$\vdots$};
\node at (5.5,-1.45) {$\vdots$};
\node at (4.5,-4.5) {$d_{22}$};
\node at (4.5,-3.5) {$d_{32}$};
\node at (4.5,-2.45) {$\vdots$};
\node at (4.5,-1.45) {$\vdots$};
\node at (3.5,-3.5) {$d_{33}$};
\node at (3.5,-2.45) {$\vdots$};
\node at (3.5,-1.45) {$\vdots$};
\node at (2.5,-2.5) {$\cdots$};
\node at (2
.5,-1.5) {$\cdots$};
\node at (1
.5,-1.5) {$\cdots$};

\node[scale=1.25] at (3, -7) {Type II};
    
\end{tikzpicture}
\qquad\qquad
\begin{tikzpicture}[scale=0.8,
every node/.style={scale=0.8},
baseline=(current bounding box.north)]

\foreach[parse=true] \y in {0,...,4} {
    \foreach[parse=true] \x in {\y,...,4}
    {\draw[fill=lightgray] (\x,-\y) rectangle ++(1,-1);
    }}
\node at (0.5,-0.5) {$\cdots$};
\node at (1.5,-0.5) {$\cdots$};
\node at (2.5,-0.5) {$d_{n3}$};
\node at (3.5,-0.5) {$d_{n2}$};
\node at (4.5,-0.5) {$d_{n1}$};
\node at (1.5,-1.5) {$\cdots$};
\node at (2.5,-1.45) {$\vdots$};
\node at (3.5,-1.45) {$\vdots$};
\node at (4.5,-1.45) {$\vdots$};
\node at (2.5,-2.5) {$d_{43}$};
\node at (3.5,-2.5) {$d_{42}$};
\node at (4.5,-2.5) {$d_{41}$};
\node at (3.5,-3.5) {$d_{32}$};
\node at (4.5,-3.5) {$d_{31}$};
\node at (4.5,-4.5) {$d_{21}$};

\node[scale=1.25] at (2.5, -7) {Type III};
    
\end{tikzpicture}
\end{equation}

%% file: Filled_diagrams_di.tex
\begin{equation}
    \label{di}
\begin{tikzpicture}[scale=0.8,
every node/.style={scale=0.8},
baseline=(current bounding box.north)]

\foreach[parse=true] \y in {0,...,5} {
    \foreach[parse=true] \x in {0,...,3}
    {\draw[fill=lightgray] (\x,-\y) rectangle ++(1,-1);
    }}
\node at (0.5,-0.5) {$d_{p}$};
\node at (1.5,-0.5) {$d_{p+1}$};
\node at (2.5,-0.5) {$\cdots$};
\node[scale=.7] at (3.5,-0.5) {$d_{p+q-1}$};
\node at (0.5,-1.5) {$d_{p-1}$};
\node at (0.5,-2.45) {$\vdots$};
\node at (0.5,-3.5) {$d_{3}$};
\node at (0.5,-4.5) {$d_{2}$};
\node at (0.5,-5.5) {$d_{1}$};

\foreach[parse=true] \y in {1,...,5} {
    \foreach[parse=true] \x in {1,...,3}
    {\node at (.5+\x,-.4-\y) {\rotatebox{-10}{$\ddots$}};}
}

\node[scale=1.25] at (2, -7) {Type I};
    
\end{tikzpicture}
\qquad\qquad
\begin{tikzpicture}[scale=0.8,
every node/.style={scale=0.8},
baseline=(current bounding box.north)]

\foreach[parse=true] \y in {0,...,5} {
    \foreach[parse=true] \x in {\y,...,5}
    {\draw [fill=lightgray] (\x,-\y) rectangle ++(1,-1);
    }}
\node at (0.5,-0.5) {$d_{n}$};
\node at (1.5,-0.5) {$d_{n-1}$};
\node at (2.5,-0.5) {$\cdots$};
\node at (3.5,-0.5) {$d_3$};
\node at (4.5,-0.5) {$d_2$};
\node at (5.5,-0.5) {$d_1$};

\foreach[parse=true] \y in {1,...,5} {
    \foreach[parse=true] \x in {\y,...,5}
    {\node at (.5+\x,-.4-\y) {\rotatebox{-10}{$\ddots$}};}
}

\node[scale=1.25] at (3, -7) {Type II};
    
\end{tikzpicture}
\qquad\qquad
\begin{tikzpicture}[scale=0.8,
every node/.style={scale=0.8},
baseline=(current bounding box.north)]

\foreach[parse=true] \y in {0,...,4} {
    \foreach[parse=true] \x in {\y,...,4}
    {\draw [fill=lightgray] (\x,-\y) rectangle  ++(1,-1);
    }}
\node at (0.5,-0.5) {$d_{n}$};
\node at (1.5,-0.5) {$d_{n-2}$};
\node at (2.5,-0.5) {$\cdots$};
\node at (3.5,-0.5) {$d_2$};
\node at (4.5,-0.5) {$d_1$};
\node at (1.5,-1.5) {$d_{n-1}$};
\node at (2.5,-2.5) {$d_n$};
\node at (3.5,-3.5) {$d_{n-1}$};

\foreach[parse=true] \y in {1,...,3} {
    \foreach[parse=true] \x in {\y,...,3}
    {\node at (1.5+\x,-.4-\y) {\rotatebox{-10}{$\ddots$}};}
}
\node at (4.5,-4.4) {\rotatebox{-10}{$\ddots$}};

\node[scale=1.25] at (2.5, -7) {Type III};
    
\end{tikzpicture}
\end{equation}

%% file: Example_Res_A.tex
\[\begin{tikzpicture}

\node (a) at (-2.5,0) {\ytableaushort[*(gray!40)]{13,41}};

\node (b) at (-1.2,0) {\ytableaushort[*(gray!40)]{13,4}};

\node (c) at (0,1) {\ytableaushort[*(gray!40)]{13}};

\node (d) at (0,-1) {\ytableaushort[*(gray!40)]{1,4}};

\node (e) at (1,0) {\ytableaushort[*(gray!40)]{1}};

\node (f) at (2,0) {$\bullet$};

\path [->]  (a) edge (b)
            (b) edge (c)
                edge (d)
            (c) edge (e)
            (d) edge (e)
            (e) edge (f);

\end{tikzpicture}
\]

%% file: Example_Res_C.tex
\[
\ytableausetup{smalltableaux}
\begin{tikzpicture}

\node (d) at (0,1) {\ytableaushort[*(gray!40)]{213}};

\node (a) at (-4.9,0) {\ytableaushort[*(gray!40)]{213,\none 21,\none \none 2}};

\node (b) at (-3.2,0) {\ytableaushort[*(gray!40)]{213,\none 21}};

\node (c) at (-1.5,0) {\ytableaushort[*(gray!40)]{213,\none 2}};

\node (e) at (1.2,0) {\ytableaushort[*(gray!40)]{21}};

\node (f) at (2.4,0) {\ytableaushort[*(gray!40)]{2}};

\node (g) at (3.4,0) {$\bullet$};

\node (h) at (0,-1) {\ytableaushort[*(gray!40)]{21,\none 2}};

\path [->]  (a) edge (b)
            (b) edge (c)
            (c) edge (d)
                edge (h)
            (d) edge (e)
            (h) edge (e)
            (e) edge (f)
            (f) edge (g);

\end{tikzpicture}
\]

%% file: Table_main.tex
\resizebox{\linewidth}{!}{
\begin{tblr}{|Q[m,c]|Q[m,c]|Q[m,c]|Q[m,c]||Q[m,c]||Q[m,c]|Q[m,c]|Q[m,c]|}\hline

Type & $\g$ & $\la$ & elements $\pi^* \in \tL(\B_\la)$ & $m$ & $\g'$ & $\la'$ & $\pi'^* \in \tL(\B_{\la'})$ \\ \hline[2pt]

I & $\ssA_{P+Q-1}$ & $-k\omega_P$ & {$(\al\mid \be+m)^*$\\$\otimes$\\$(\be \mid \al + m)$, \\[8pt] $(\al | \be) \in \Par(p \times q)$} & $k$ & $\ssA_{p+q-1}$ & $k \omega'_p$ & {$(\al + m \mid \be)^*$\\$\otimes$\\$(\be + m \mid \al)$, \\[8pt] $(\al | \be) \in \Par(p \times q)$}\\ \hline[2pt]

II & $\ssD_N$ & $-2k\omega_N$ & $(\al \mid \al + m)^*$ & $2k+1$ & $\ssC_n$ & $k\omega'_n$ & $(\al + m \mid \al)^*$ \\ \hline[2pt]

IIIa & $\ssC_N$ & $-\frac{k}{2}\omega_N$ & {$(\al \mid \al + m)^*$  \\ even rank} & $k-1$ & $\ssD_n$ & $k\omega'_n$ & {$(\al + m \mid \al)^*$  \\ even rank} \\ \hline

IIIb & $\ssC_N$ & $-\frac{k}{2}\omega_N + \omega_k^*$ & {$(\al \mid \al + m)^*$ \\ odd rank} & $k-1$ & $\ssD_n$ & $k\omega'_{n-1}$ & {$(\al + m \mid \al)^*$ \\ odd rank} \\ \hline[2pt]

IIIc & $\ssD_N$ & $-(2k-1)\omega_N$ & {$(\al \mid \al + m)^*$ \\ even rank} & $2k$ & $\ssD_n$ & $(2k+1)\omega'_n$ & {$(\al + m \mid \al )^*$ \\ even rank} \\ \hline

IIId & $\ssD_N$ & $-(2k-1)\omega_N + \omega^*_{2k+1}$ & {$(\al \mid \al + m)^*$ \\ odd rank} & $2k$ & $\ssD_n$ & $(2k+1)\omega'_{n-1}$ & {$(\al \mid \al + m)^*$ \\ odd rank} \\ \hline

\end{tblr}
}

%% file: Example_Conjugates_C.tex
\[\begin{tikzpicture}

  \node (a) at  (-3,4.5) {$(5,4,3,\bbox)$};

  \node (b) at  (-3,3) {$(5,4,\bbox,\overline{3})$};

  \node (c) at  (-3,1.5) {$(5,3,\bbox,\overline{4})$};

  \node (d) at (-4.5,0) {$(4,3,\bbox,\overline{5})$};

  \node (e) at (-1.5,0) {$(5,\bbox,\overline{3},\overline{4})$};

  \node (f) at (-3,-1.5) {$(4,\bbox,\overline{3},\overline{5})$};

  \node (g) at  (-3,-3) {$(3, \bbox, \overline{4}, \overline{5})$};

  \node (h) at  (-3,-4.5) {$(\bbox,\overline{3}, \overline{4}, \overline{5})$};

  \node [draw] (i) at (-3,-6) {\begin{tabular}{ll}
     Without $\bbox$:  & $w(\la'+\rho')$ \\
     With $\bbox$:  & $w(\la'+\rho')^\sharp$
  \end{tabular}};

  \draw [<-] (a) edge (b) (b) edge (c) (c) edge (d) (d) edge (f) (f) edge (g) (g) edge (h) (c) edge (e) (e) edge (f);

  \node (a) at (3,4.5) {$\bullet$};

  \node (b) at (3,3) {\ytableaushort[*(gray!30)]{3}};

  \node (c) at (3,1.5) {\ytableaushort[*(gray!30)]{31}};

  \node (d) at (1.5,0) {\ytableaushort[*(gray!30)]{311}};

  \node (e) at (4.5,0) {\ytableaushort[*(gray!30)]{31,\none 3}};

  \node (f) at (3,-1.5) {\ytableaushort[*(gray!30)]{311,\none 3}};

  \node (g) at (3,-3) {\ytableaushort[*(gray!30)]{311,\none 31}};

  \node (h) at (3,-4.5) {\ytableaushort[*(gray!30)]{311,\none 31,\none \none 3}};

  \node (i) at (3,-6) {$[\Phi_w]_{\la'}$};

  \draw [<-] (a) edge (b) (b) edge (c) (c) edge (d) (d) edge (f) (f) edge (g) (g) edge (h) (c) edge (e) (e) edge (f);

\end{tikzpicture}\]

%% file: Table_Hilbert.tex
\resizebox{\linewidth}{!}{
\begin{tblr}{colspec = {|Q[m,c]|Q[m,c]|Q[m,c]|Q[m,c]|Q[m,c]|Q[m,c]|Q[m,c]|}, row{2-6} = {25pt}}

\hline

Type & $\g$ & $\la$ & $m$ &  $d$ & $\nu \in \ldots$ & $P(t)$ \\ \hline[2pt]

I & $\ssA_{P+Q-1}$ & $-k\omega_P$ & $k$ & $k(p+q+k)$ & $\Par(\min\{p,q\}\times k)$ & $\sum_\nu \left(\dim \F{\nu}{p} \dim \F{\nu}{q}\right) t^{|\nu|}$ \\ \hline[2pt]

II & $\ssD_N$ & $-2k\omega_N$ & $2k+1$ & $k(2n+2k+1)$ & {$\Par(n \times 2k)$, \\ even rows} & \SetCell[r=5]{m} $\sum_\nu \left(\dim \F{\nu}{n}\right) t^{|\nu|/2}$ \\ \hline[2pt]

IIIa & $\ssC_N$ & $-\frac{k}{2} \omega_N$ & $k-1$ & $k(2n+k-1)/2$ & {$\Par(n\times k)$, \\ even columns} & \\ \hline

IIIb & $\ssC_N$ & $-\frac{k}{2}\omega_N + \omega_k^*$ & $k-1$ & $k(2n+k-1)/2$ & {$\Par(n \times k)$, \\ odd columns} & \\ \hline[2pt]

IIIc & $\ssD_N$ & $-(2k-1)\omega_N$ & $2k$ &  $k(2n+2k-1)$ & {$\Par(n \times (2k+1))$, \\ even columns} & \\ \hline

IIId & $\ssD_N$ & $-(2k-1)\omega_N+\omega^*_{2k+1}$ & $2k$ & $k(2n+2k-1)$ & {$\Par(n \times (2k+1))$, \\ odd columns} \\ \hline

\end{tblr}
}

%% file: Table_IDs.tex
\resizebox{\linewidth}{!}{
\begin{tblr}{vlines, rows={m}, columns = {c}}

\hline

Type & Identity & $\nu \in \ldots$ & $\pi$ & $\ell(\pi)$ \\ \hline[2pt]

I & $\begin{aligned}\displaystyle\prod_{i,j} (1 - x_i y_j) &\sum_\nu s_\nu(\bfx)s_\nu(\bfy)\\
=& \sum_\pi (-1)^{\ell(\pi)} s_{\pi_x}(\bfx) s_{\pi_y}(\bfy)
\end{aligned}$ & $\Par(\min\{p,q\}\times k)$ & {$(\al|\be)\in \Par(p \times q)$;\\$\pi_x = (\al + k \mid \be)$,\\$\pi_y = (\be + k \mid \al)$} & $|\pi|-k \rk \pi$\\ \hline

II & $\begin{aligned}\displaystyle\prod_{i \leq j} (1 - x_i x_j) &\sum_\nu s_\nu(\bfx)\\ = &\sum_\pi (-1)^{\ell(\pi)} s_\pi(\bfx)\end{aligned}$  & {$\Par(n \times 2k)$,\\even rows} & {$(\al + 2k + 1 \mid \al)$,\\$\al_1 < n$} & $\dfrac{|\pi| - 2 k \rk \pi}{2}$ \\ \hline

IIIa & \SetCell[r=2]{c} $\begin{aligned}\displaystyle \prod_{i < j} (1 - x_i x_j) & \sum_\nu s_\nu(\bfx) \\ = &\sum_\pi (-1)^{\ell(\pi)} s_\pi(\bfx)\end{aligned}$ &{$\Par(n\times k)$,\\ even columns} & 
{$(\al+k-1\mid\al)$,\\even rank,\\$\al_1 < n$} & \SetCell[r=2]{c} $\dfrac{|\pi|- k \rk \pi}{2}$ \\ \hline

IIIb & & {$\Par(n\times k)$,\\ odd columns} & {$(\al+k-1\mid\al)$,\\odd rank,\\$\al_1 < n$} & \\ \hline
\end{tblr}
}

%% file: Sporadic_example.tex
\[\ytableausetup{boxsize=3.5pt}
\begin{tikzpicture}

  \node (g) at (-4,6) {$(\ssD_6, \ssA_5)$};

  \node[label=right: {$\dim = 6$}] (a) at  (-4,4.5) {$\phantom{*\,}\ydiagram{1}^{\,*}$};

  \node[label=right: {$\dim = 20$}] (b) at  (-4,3) {$\phantom{*\,}\ydiagram{1,1,1}^{\,*}$};

  \node[label=right: {$\dim = 84$}] (c) at  (-4,1.5) {$\phantom{*\,}\ydiagram{2,2,1,1,1}^{\,*}$};

  \node[label=right: {$\dim = 70$}] (d) at (-5.5,0) {$\phantom{*\,}\ydiagram{3,2,1,1,1,1}^{\,*}$};

  \node[label=right: {$\dim = 70$}] (e) at (-2.5,0) {$\phantom{*\,}\ydiagram{2,2,2,2,1}^{\,*}$};

  \node[label=right: {$\dim = 84$}] (f) at  (-4,-1.5) {$\phantom{*\,}\ydiagram{3,2,2,2,1,1}^{\,*}$};

  \node[label=right: {$\dim = 20$}] (g) at  (-4,-3) {$\phantom{*\,}\ydiagram{3,3,3,2,2,2}^{\,*}$};

  \node[label=right: {$\dim = 6$}] (h) at  (-4,-4.5) {$\phantom{*\,}\ydiagram{3,3,3,3,3,2}^{\,*}$};

  \draw [<-] (a) edge (b) (b) edge (c) (c) edge (d) (d) edge (f) (f) edge (g) (g) edge (h) (c) edge (e) (e) edge (f);

  \node (g) at (4,6) {$(\ssD_4, \ssA_3)$};

  \node[label=right: {$\dim = 6$}] (a) at (4,4.5) {$\phantom{*\,}\ydiagram{1,1}^{\,*}$};

  \node[label=right: {$\dim = 20$}] (b) at (4,3) {$\phantom{*\,}\ydiagram{2,2}^{\,*}$};

  \node[label=right: {$\dim = 84$}] (c) at (4,1.5) {$\phantom{*\,}\ydiagram{4,2,2}^{\,*}$};

  \node[label=right: {$\dim = 70$}] (d) at (2.5,0) {$\phantom{*\,}\ydiagram{5,2,2,1}^{\,*}$};

  \node[label=right: {$\dim = 70$}] (e) at (5.5,0) {$\phantom{*\,}\ydiagram{4,3,3}^{\,*}$};

  \node[label=right: {$\dim = 84$}] (f) at (4,-1.5) {$\phantom{*\,}\ydiagram{5,3,3,1}^{\,*}$};

  \node[label=right: {$\dim = 20$}] (g) at (4,-3) {$\phantom{*\,}\ydiagram{5,5,3,3}^{\,*}$};

  \node[label=right: {$\dim = 6$}] (h) at (4,-4.5) {$\phantom{*\,}\ydiagram{5,5,4,4}^{\,*}$};

  \draw [<-] (a) edge (b) (b) edge (c) (c) edge (d) (d) edge (f) (f) edge (g) (g) edge (h) (c) edge (e) (e) edge (f);

\end{tikzpicture}
\]